\documentclass[11pt,a4paper, twoside]{article}
\usepackage[utf8]{inputenc}

\usepackage[all]{xy}
\usepackage{latexsym}
\usepackage[margin=3cm]{geometry}
\usepackage[dvipdfmx]{graphicx}
\usepackage{hyperref}
\usepackage{xurl}
\usepackage{mymacro}
\usepackage{tikz-cd}
\usepackage{mathtools}

\usepackage[
backend=biber,
style=ext-alphabetic,
maxnames = 5,
maxalphanames = 5,
maxcitenames  = 5,
minalphanames = 5,
minnames      = 5,
sorting=nyt,
giveninits,
articlein=false,
url=false,  
doi=false,
eprint=true,
isbn=false
]{biblatex}
\usepackage{fancyhdr}
\newcommand\shorttitle{\footnotesize{Topological constraints on clean Lagrangian intersections via microlocal sheaf theory}}
\newcommand\authors{\footnotesize{Tomohiro Asano and Yukihiro Okamoto}}

\title{Topological constraints on clean Lagrangian intersections via microlocal sheaf theory}
\author{Tomohiro Asano \and Yukihiro Okamoto}
\date{}

\begin{document}

\maketitle

\begin{abstract}
\noindent
Fix a knot $K_0$ in $\mathbb{R}^3$ and consider a Lagrangian submanifold $L$ of $T^*\mathbb{R}^3$ that is isotopic to the conormal bundle of $K_0$ by a compactly supported Hamiltonian isotopy and intersects the zero section $\mathbb{R}^3$ cleanly along a knot.
In this paper, using microlocal sheaf theory and some results in $3$-manifold theory, we prove that the knot type of $K_1\coloneqq L\cap \mathbb{R}^3$ in $\mathbb{R}^3$ is strictly constrained from the knot type of $K_0$.
Specifically, we deduce the existence of a surjective group homomorphism $\pi_1(\mathbb{R}^3\setminus K_0) \to \pi_1(\mathbb{R}^3\setminus K_1)$ preserving the longitude and meridian with respect to the Seifert framing.
Moreover, combining with a previous work by the second author, we obtain a rigidity result which was only known for the unknot: If $K_0$ is the $(2,q)$-torus knot or the figure-eight knot, $K_1$ must have the same knot type as $K_0$.
\end{abstract}

\section{Introduction}

\subsection{Background}

Let $L_1$ and $L_2$ be Lagrangian submanifolds of a symplectic manifold $(X,\omega)$ and consider their intersection $L_1\cap L_2$.
Here, we suppose that $L_1$ and $L_2$ have a \textit{clean intersection} (for the definition, see Definition \ref{def-clean-int}).
In this case, the intersection $L_1 \cap L_2$ is a submanifold of $L_1$ (and $L_2$) and may have a positive dimension.
When $\dim (L_1 \cap L_2)>0$, unlike the case of transversal intersections, it is a non-trivial question how $L_1 \cap L_2$ is embedded in $L_1$ (and $L_2$) as a submanifold.

With a perspective of low dimensional topology,
an interesting case would be when $\dim X =6$ and $L_1 \cap L_2$ is an embedded circle, that is, a knot in the $3$-manifold $L_i$ for $i=1,2$.
For a pair of Lagrangian $3$-spheres, Smith made a question concerning the
`persistence (or
rigidity) of unknottedness of Lagrangian intersections'
under Hamiltonian diffeomorphisms \cite[Question 1.5]{GP}.
When $X$ is a plumbing of two copies of $T^*S^3$ along the unknots in $S^3$, there are results related to this question by Ganatra--Pomerleano \cite[Proposition 6.29]{GP}, Smith--Wemyss \cite[Proposition 4.10]{SmW} and Asplund--Li \cite[Theorem 1.3]{AL}.
In particular, \cite[Theorem 1.1]{AL} shows the persistence of unknottedness of clean Lagrangian intersections under a `nearby Hamiltonian diffeomorphism'.

In the present paper, we consider the case where $X=T^*\bR^3$ with a standard symplectic structure.
For any smooth knot $K$ in $\bR^3$, let $T^*_K\bR^3$ denote its conormal bundle.
Note that $T^*_K\bR^3$ and the zero section $0_{\bR^3}$ intersect cleanly along $K \subset 0_{\bR^3} = \bR^3$.
Let $\Ham_c(T^*\bR^3)$ denote the group of compactly supported Hamiltonian diffeomorphisms on $T^*\bR^3$. 
Given a knot $K$ in $\bR^3$,  we ask the following:
\begin{question}\label{Q-1}
Is there $\varphi \in \Ham_c(T^*\bR^3)$ such that $\varphi(T^*_{K}\bR^3)$ and $0_{\bR^3}$ have a clean intersection along a knot whose knot type is different from $K$?
\end{question}
Here, the knot type means the isotopy class of a knot in $\bR^3$.
\begin{remark}
Let us remark three points about Question \ref{Q-1}:
\begin{enumerate}
\item The condition that $\varphi$ is a Hamiltonian diffeomorphism is essential.
Indeed, given any pair of knots $K_0$ and $K_1$, there always exists a $C^{\infty}$ isotopy $(\varphi^s)_{s\in [0,1]}$ on $T^*\bR^3$ with a compact support such that $\varphi^1(T^*_{K_0}\bR^3)$ and $0_{\bR^3}$ have a clean intersection along $K_1$.
\item As far as the authors know, no example has been found.
\item The connectedness of the clean intersection $\varphi(T^*_{K_0}\bR^3)\cap 0_{\bR^3}$ is also an important condition, and it does not seem obvious how to generalize the main results below to clean intersections along links with multiple components.
See Remark \ref{rem-connected-intersection} for a technical reason.
\end{enumerate}
\end{remark}

Some results related to Question \ref{Q-1} were obtained by \cite{O23} and \cite{O25}.
For example, \cite[Theorem 1.3]{O23} gives a negative answer to Question \ref{Q-1} when $K$ is the unknot in $\bR^3$.
(In fact, this answer can also be deduced from  \cite[Proposition 6.29]{GP} mentioned above. See \cite[Section 1]{O23}.)
In several cases where $K$ has a non-trivial knot type, \cite[Theorem 1.4 and 1.5]{O23} and \cite[Theorem 1.1]{O25} give some mild constraints on the knot types of $\varphi(T^*_K\bR^3) \cap 0_{\bR^3}$.

\subsection{Main results}

\begin{figure}[htbp]
\centering

\begin{tikzpicture}[scale=0.8]

\draw (0,0) -- (5,0) ; 
\draw (3.5,2) -- (3.5,-2);
\fill (1.5,0) circle (3pt) ;
\fill (3.5,0) circle (3pt) ;

\draw (3.4,0) node[below right]{$K_0$} ;
\draw (1.4,0) node[below right]{$K_1$} ;
\draw (0.2,-0.1) node[above]{$0_{\bR^3}$} ;
\draw (0,2) node[]{$T^*\bR^3$} ;
\draw (3.4,1.8) node[right]{$T^*_{K_0}\bR^3$} ;

\draw (7,0) -- (12,0);
\draw[densely dotted] (10.5,2) -- (10.5,-2);
\draw (10.5,2) to [out=-90, in=90] (10.5, 1.7) to [out=-90, in=45] (8.7,0.4) to [out=225, in=45] (8.3, -0.4) to [out=225, in=90] (8, -0.7)
to [out=-90, in=90] (10.5,-1.7) to [out= -90, in =90] (10.5,-2) ;
\fill (8.5,0) circle (3pt) ;

\draw (8.4,0) node[below right]{$K_1$} ;
\draw (7.2,-0.1) node[above]{$0_{\bR^3}$} ;
\draw (9,1.3) node[]{$\varphi(T^*_{K_0}\bR^3)$} ;

\end{tikzpicture}
\caption{In the left-hand side, the bullets represent the knots $K_0$ and $K_1$ in $0_{\bR^3}=\bR^3$.
We suppose that $\varphi (T^*_{K_0}\bR^3)$ intersects $0_{\bR^3}$ cleanly along $K_1$ as in the right-hand side.}\label{figure-clean}
\end{figure}
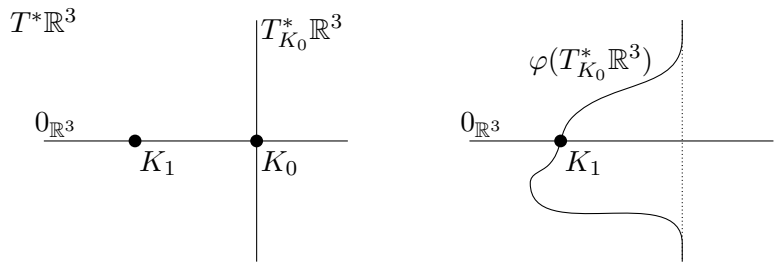

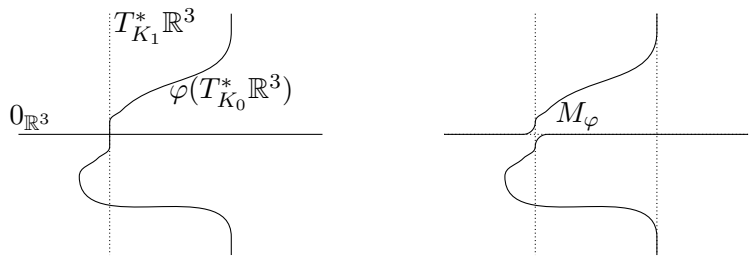
\begin{figure}[htbp]
\centering

\begin{tikzpicture}[scale=0.8]

\draw (0,0) -- (5,0);
\draw[densely dotted] (1.5,2) -- (1.5,-2);
\draw (3.5,2) to [out=-90, in=90] (3.5, 1.7) to [out=-90, in=45] 
(1.7,0.4) to [out=225, in=90] (1.5, 0.2) to [out=-90, in=90] (1.5,-0.2) to [out=-90, in=45] (1.3,-0.4) to [out=225, in=90] (1, -0.7)
to [out=-90, in=90] (3.5,-1.7) to [out= -90, in =90] (3.5,-2) ;

\draw (0.2,-0.1) node[above]{$0_{\bR^3}$} ;
\draw (1.4,1.8) node[right]{$T^*_{K_1}\bR^3$} ;
\draw (3.5, 0.7) node[]{$\varphi (T^*_{K_0}\bR^3)$} ;

\draw[densely dotted] (7,0) -- (12,0);
\draw[densely dotted] (10.5,2) -- (10.5,-2);
\draw[densely dotted] (8.5,2) -- (8.5,-2);
\draw (10.5,2) to [out=-90, in=90] (10.5, 1.7) to [out=-90, in=45] 
(8.7,0.4) to [out=225, in=90] (8.5, 0.2)  to [out=-90, in=0] (8.3,0)  to [out=180, in=0] (7,0);
\draw (12,0) to [out=180, in=0] (8.7,0) to [out=180, in=90] (8.5,-0.2) to [out=-90, in=45] (8.3,-0.4) to [out=225, in=90] (8, -0.7)
to [out=-90, in=90] (10.5,-1.7) to [out= -90, in =90] (10.5,-2) ;

\draw (9.2,-0.1) node[above]{$M_{\varphi}$} ;

\end{tikzpicture}
\caption{
We may assume that $\varphi(T^*_{K_0}\bR^3)$ agrees with $T^*_{K_1}\bR^3$ in a neighborhood of $0_{\bR^3}$ (the left-hand side).
The $3$-manifold $M_{\varphi}$ is diffeomorphic to the Lagrangian submanifold obtained by a Polterovich surgery on $\varphi(T^*_{K_0}\bR^3)\cup 0_{\bR^3}$ along $K_1$ (the right-hand side).
}\label{figure-surgery}
\end{figure}

Let $K_0$ and $K_1$ be knots in $\bR^3$.
Let $\varphi \in \Ham_c(T^*\bR^3)$ and suppose that $\varphi (T^*_{K_0}\bR^3)$ and $0_{\bR^3}$ intersect cleanly along $K_1\subset 0_{\bR^3} \cong \bR^3$.
See Figure \ref{figure-clean}.

In order to explain the key result, let us prepare some notation.
In Section \ref{subsec-clean}, we will introduce
a $3$-manifold $M_{\varphi}$ which is diffeomorphic to a Lagrangian submanifold of $T^*\bR^3$ obtained by a Polterovich surgery of $\varphi(T^*_{K_0}\bR^3) \cup 0_{\bR^3}$ along $K_1$.
See the right-hand side of Figure \ref{figure-surgery}.
(In this paper, we just think of $M_{\varphi}$ as a $3$-manifold, and do not use the property that it is an exact Lagrangian submanifold of $T^*\bR^3$.)
Let $\Lambda_{K_0}$ be the unit conormal bundle of $K_0$.
There are two natural embeddings
\[ \xymatrix{ M_{\varphi} & \ar[l]_-{\sigma_0} \Lambda_{K_0} \ar[r]^-{e_0} & \bR^3 \setminus K_0 ,
}\]
where $\sigma_0$ is an embedding onto
$\{(x,\xi)\in T^*_{K_0}\bR^3 \mid |\xi|=r\}$ for sufficiently large $r>0$, and
$e_0$ is an embedding onto the boundary of a tubular neighborhood of $K_0$ in $\bR^3$.
For their explicit descriptions, see Section \ref{subsec-clean}.
They induce group homomorphisms between the fundamental groups
\[ \xymatrix{ \pi_1(M_{\varphi}) & \ar[l]_-{(\sigma_0)_*} \pi_1(\Lambda_{K_0}) \ar[r]^-{(e_0)_*} & \pi_1(\bR^3 \setminus K_0) .
} \]
Here, we fix $y_0\in \Lambda_{K_0}$ and choose $\sigma_0(y_0)$, $y_0$ and $e_0(y_0)$ respectively as the base points.

\begin{theorem}[Theorem \ref{thm-ml-preserve}]\label{thm-1}
There exists a surjective group homomorphism
\[\tilde{h} \colon \pi_1 (\bR^3 \setminus K_0) \to \pi_1 (M_{\varphi})\]
such that $\tilde{h}\circ (e_0)_* = (\sigma_0)_*$.
\end{theorem}

In this paper, we say that a surjective homomorphism $\tilde{h}\colon \pi_1 (\bR^3 \setminus K_0) \to \pi_1 (M_{\varphi})$ is \emph{peripheral-pair-preserving} if $\tilde{h} \circ (e_0)_* = (\sigma_0)_*$. 
A surjective homomorphism $h \colon \pi_1(\bR^3\setminus K_0) \to \pi_1(M_{\varphi})$ is constructed through the sheaf theoretic argument outlined in Section \ref{subsec-idea-proof}.
We also find its relation with $e_0$ and $\sigma_0$ in Proposition \ref{prop-GL_n-rep}.
This surjective homomorphism $h$ is upgraded to a peripheral-pair-preserving one $\tilde{h}$ in Section \ref{subsec-peripheral} through discussion on the fundamental groups of certain $3$-manifolds and the JSJ decomposition of a compact $3$-manifold related to $M_{\varphi}$.

From Theorem \ref{thm-1}, we can extract topological constraints on the pair of knots $(K_0,K_1)$.
One of the consequences is the following.
\begin{theorem}[Theorem \ref{thm-trefoil} and \ref{thm-figure-8}]\label{thm-2}
If $K_0$ is the $(2,q)$-torus knot or the figure-eight knot, then $K_1$ is isotopic to $K_0$ in $\bR^3$.
\end{theorem}
Notably, this gives a negative answer to Question \ref{Q-1} when $K$ is the $(2,q)$-torus knot or the figure-eight knot.

\begin{remark}
When $K_0$ has a knot type as in Theorem \ref{thm-2}, one can see that $K_1$ is either isotopic to $K_0$ or the unknot
by using Theorem \ref{thm-1} and some known results about the fundamental groups of knot complements.
(We remark that Proposition \ref{prop-GL_n-rep} before enhanced to Theorem \ref{thm-1} is sufficient for this purpose.) 
The case where $K_1$ is the unknot (when $K_0$ is not the unknot) is excluded by \cite[Theorem 1.1]{O25}.
The proof of \cite[Theorem 1.1]{O25} relies on computations of the Legendrian contact homology of $\Lambda_{K_i}$ with coefficients in $\bZ[H_2(U^*\bR^3,\Lambda_{K_i})]$ for $i=0,1$, where $U^*\bR^3$ is the unit cotangent bundle of $\bR^3$, and at present the authors do not have its alternative proof using microlocal sheaf theory. 
\end{remark}

We also obtain a constraint on the isotopy class of $\varphi^{-1}(K_1)= T^*_{K_0}\bR^3\cap \varphi^{-1}(0_{\bR^3})$ in the $3$-manifold $T^*_{K_0}\bR^3\cong S^1\times \bR^2$.
\begin{theorem}
Suppose that $K_0$ is isotopic to $K_1$ in $\bR^3$.
Then, $\varphi^{-1}(K_1)$ is isotopic to $K_0=T^*_{K_0}\bR^3 \cap 0_{\bR^3}$ in $T^*_{K_0}\bR^3$.
\end{theorem}
Note that for any pair of isotopic knots $(K_0,K_1)$ in $\bR^3$, there always exists $\varphi \in \Ham_c(T^*\bR^3)$ such that $\varphi(T^*_{K_0}\bR^3)$ intersects $0_{\bR^3}$ cleanly along $K_1$.

In the following three subsections, we present relations between the knots $K_0,K_1\subset \bR^3$ in terms of some knot invariants.

\subsubsection{Surjective homomorphism between knot groups}

Fix orientations of $K_0$ and $K_1$.
They are equipped with the Seifert framing.
Then, for $i=0,1$, we have two elements $l_i,m_i \in \pi_1(\bR^3\setminus K_i)$ representing the longitude and the meridian respectively.

\begin{theorem}[Corollary \ref{cor-peripheral-preserve}]
There exists a surjective group homomorphism
\[\tilde{k} \colon \pi_1(\bR^3\setminus K_0) \to \pi_1(\bR^3 \setminus K_1)\]
such that by changing the orientation of $K_1$ if necessary, 
$\tilde{k}(l_0)=l_1$ and $\tilde{k}(m_0)=m_1$.
\end{theorem}

There have been a number of studies on surjective homomorphisms between knot groups.
For example, see \cite{KSW, SW, KitanoSuzukiII, SW2, ORS, AgLiu, ALSS, BKN}.
See also Remark \ref{rem-epi-knot} and Remark \ref{rem-epi-peripheral}.

\begin{remark}
Our setup about clean Lagrangian intersections gives
a new geometric situation where an epimorphism between knot groups emerges.
\end{remark}

\subsubsection{On enhanced A-polynomials}

For a knot $K$, there is a polynomial $A_K \in \bZ [\lambda,\mu]$ associated to the character variety of $\SL_2(\bC)$-representations of $\pi_1(\bR^3\setminus K)$, so called the \textit{A-polynomial} of $K$.
It is a classical knot invariant introduced by \cite{CCGLS}.
Ni-Zang \cite{NiZ} considered a slight enhancement of A-polynomial, which we denote by $\tilde{A}_K$.
It is a polynomial in $\bZ[\lambda,\mu]$ defined by omitting curves from reducible $\SL_2(\bC)$-representations of $\pi_1(\bR^3\setminus K)$.

\begin{theorem}[Theorem \ref{thm-divide-A}]\label{thm-intro-tildeA}
$\tilde{A}_{K_0}$ is divisible by $\tilde{A}_{K_1}$.
\end{theorem}

For any knot $K$, $A_K$ is the least common multiple of $\tilde{A}_K$ and $\lambda-1$, so this theorem implies that $A_{K_0}$ is divisible by $A_{K_1}$, which is already known by \cite[Proposition 4.2]{SW}.
For the proof of \cite[Proposition 4.2]{SW}, the surjectivity of a homomorphism between knot groups is not important, while the surjectivity of $k$ is essential for Theorem \ref{thm-intro-tildeA}.
We utilize the fact that the pullback by $k$ preserves the irreducibility of ($\SL_2(\bC)$-)representations.

\subsubsection{\texorpdfstring{On KCH representations and a relation to \cite[Theorem 1.1]{O23}}{On KCH representations and a relation to [Oka25a, Theorem 1.1]}}\label{subsubsec-intro-KCH}

Let $K$ be an oriented knot.
An $n$-dimensional representation  $\rho \colon \pi_1(\bR^3\setminus K) \to \GL_n(\bC)$ is called a \textit{KCH representation} if the image of the meridian is diagonalizable and has $1$ as an eigenvalue with multiplicity $n-1$.
It is known that a KCH representations induces an augmentation of the knot DGA of $K$ (or the Chekanov-Eliashberg DGA of $\Lambda_K$ with coefficients in $\bZ[H_1(\Lambda_K)]$) to $\bC$ \cite[Theorem 5.11]{Ng-intro}.
From the set of irreducible $n$-dimensional KCH representations, we will define a subset $\tilde{U}^n_K$ of $(\bC^*)^2$ in Section \ref{subsec-KCH-rep}.
The result by Cornwell \cite[Theorem 1.2]{cornwell-2} implies that
the augmentation variety $V_K$ of $K$, which was defined by Ng \cite[Definition 5.4]{Ng} and used in \cite{O23}, can be recovered by $V_K = (\bC^*\times \{1\}) \cup \bigcup_{n=0}^{\infty} \tilde{U}^n_K$.

\begin{theorem}[Theorem \ref{thm-KCH}]
For every $n\in \bZ_{\geq 0}$, $\tilde{U}^n_{K_1}$ is a subset of $ \tilde{U}^n_{K_0}$.
\end{theorem}
In particular, we obtain a relation
\[V_{K_1}\subset V_{K_0},\]
which is a refinement of \cite[Theorem 1.1]{O23}.

\subsection{Idea of proof}\label{subsec-idea-proof}

As we will discuss in Section \ref{subsec-clean},
we may assume that $\varphi(T^*_{K_0}\bR^3)$ agrees with $T^*_{K_1}\bR^3$ in a neighborhood of $0_{\bR^3}$ as drawn in the left-hand side of Figure \ref{figure-surgery}.
By a symplectomorphism
\[F\colon \bR\times U^*\bR^3 \to T^*\bR^3 \setminus 0_{\bR^3} ,\ (r, (x;\xi)) \mapsto (x;e^r \xi),\]
we obtain an exact Lagrangian cobordism $L_{\varphi} \coloneqq F^{-1}(\varphi(T^*_{K_0}\bR^3)\setminus K_1)$ in $\bR\times U^*\bR^3$ whose positive end is $\Lambda_{K_0}$ and negative end is $\Lambda_{K_1}$.

The proof of Theorem \ref{thm-1} relies on microlocal sheaf theory \cite{KS90, Gui23}.
Fix a field $\bfk$.
For any $C^{\infty}$ manifold $X$, $\Sh(X)$ denotes the ($\infty$-)category of sheaves valued in the ($\infty$-)category of complexes of $\bfk$-vector spaces on $X$.
For any closed subset $A$ of $T^*X$, $\Sh_A(X)$ denotes the full subcategory consisting of sheaves whose microsupports are contained in $A$.
In addition, let $\Loc(X)$ denote the full subcategory of $\Sh(X)$ consisting of cohomologically locally constant sheaves.
For their precise definitions, see Section \ref{subsec-prelim}.

The idea of the proof is derived from the two preceding works:
\begin{itemize}
\item Li \cite{Li, Li2} showed that
given an exact Lagrangian cobordism between Legendrian submanifolds of the boundary of a Weinstein manifold,
we can associate a fully faithful functor between certain categories of microsheaves.
In the present case, we obtain a fully faithful functor associated to $L_{\varphi}$
\begin{align}\label{cob-functor} \Phi_{L_{\varphi}} \colon \Sh_{0_{\bR^3}\cup T^*_{K_1}\bR^3 }(\bR^3) \times_{\Loc(\Lambda_{K_1})} \Loc (L_{\varphi}) \hookrightarrow \Sh_{0_{\bR^3}\cup T^*_{K_0}\bR^3}(\bR^3) . 
\end{align}
\item
Shende \cite{Shende} proved that the Legendrian isotopy class of $\Lambda_K$ is a complete isotopy invariant of knots $K$.
To recover the fundamental group of $\bR^3\setminus K$
from the category $\Sh_{0_{\bR^3}\cup T^*_{K}\bR^3} (\bR^3)$,
a semiorthogonal decomposition 
\[\Sh_{0_{\bR^3}\cup T^*_{K}\bR^3} (\bR^3) = \la \Loc(K) , \Loc(\bR^3\setminus K) \ra\]
was used in \cite[Theorem 7]{Shende}.
For the precise definition of this decomposition, see \cite[Section 2]{Shende}. 
\end{itemize}

Combining the functor $\Phi_{L_{\varphi}}$ and the semiorthogonal decomposition, we aim to relate $\Loc(\bR^3\setminus K_0)$ with $\Loc(\bR^3\setminus K_1)$ and then extract a relation between the fundamental groups of $\bR^3\setminus K_0$ and $\bR^3\setminus K_1$. 
However, there is a technical difficulty due to the factor $\Loc(L_{\varphi})$ in the domain of $\Phi_{L_{\varphi}}$.
That is, in the category $\Sh_{0_{\bR^3}\cup T^*_{K_1}\bR^3}(\bR^3) \times_{\Loc(\Lambda_{K_1})} \Loc(L_{\varphi})$, we do not have a semiorthogonal decomposition analogous to $\Sh_{0_{\bR^3}\cup T^*_{K_0}\bR^3}(\bR^3)$.
(In this paper, we use the notion of the split Verdier sequence instead of the semiorthogonal decomposition for $\infty$-categories.)

\subsection{Comparison with an approach using Chekanov-Eliashberg DGAs}

In the present setting, the Chekanov-Eliashberg DGA
$(\mathcal{A}_*(\Lambda_{K_i}), \partial)$ of $\Lambda_{K_i}$ is defined with coefficients in $\bZ[H_1 (\Lambda_{K_i})]$.
Moreover, we can construct a DGA map over $\bZ[H_1(\Lambda_{K_0})]$
\[  \Psi_{L_{\varphi}} \colon \mathcal{A}_*(\Lambda_{K_0})  \to \mathcal{A}_*(\Lambda_{K_1}) \otimes_{\bZ[H_1(\Lambda_{K_1})]} \bZ[H_1(L_{\varphi})] \]
associated to the exact Lagrangian cobordism $L_{\varphi}$. See \cite{Ekholm, EHK} and \cite[Section 3]{O23} for more details.
(In fact, there is an isomorphism $\bZ[H_1(\Lambda_{K_i})]\to \bZ[H_1(L_{\varphi})]$ for $i=0,1$ induced by a natural embedding, so we may identify these rings.)
Fix a field $\bfk$ and let $\mathit{Aug}(\Lambda_{K_i};\bfk)$ be an $A_{\infty}$-category defined as in \cite[Section 3.2.1]{BC} whose objects are maps $\epsilon \colon \mathcal{A}_*(\Lambda_{K_i})\to \bfk$, called \textit{augmentations} of the DGA $(\mathcal{A}_*(\Lambda_{K_i}), \partial)$ to $\bfk$.
We remark that this $A_{\infty}$-category is non-unital, and distinguished from the $A_{\infty}$-category $\mathit{Aug}_+(\Lambda)$ with unit defined by \cite{NRSSZ15} for Legendrian links $\Lambda$ in $\bR^3$.
For any augmentation $\epsilon$ of $\mathcal{A}_*(\Lambda_{K_1})$, the composite map $\epsilon\circ \Psi_{L_{\varphi}}$ is an augmentation of $\mathcal{A}_*(\Lambda_{K_0})$.
We expect that this correspondence is extended to a functor between the categories
\[ \Psi_{L_{\varphi}}^*\colon \mathit{Aug}(\Lambda_{K_1};\bfk) \to \mathit{Aug}(\Lambda_{K_0};\bfk) ,\  \epsilon \mapsto \epsilon \circ \Psi_{L_{\varphi}} . \]
A similar functor is mentioned in \cite[Section 1.1]{Li} and compared with the functor (\ref{cob-functor}).
As explained in \cite[Section 1.1]{Li} and \cite[Section 1.1]{Li2}, Li's work was motivated by a relation between a microlocal sheaf category and an augmentation category based on the work of Ganatra--Pardon--Shende \cite{GPS24} and Legendrian surgery formula \cite{BEE12, EL23,AE22}.

The advantage of the functor (\ref{cob-functor}) is its fully faithfulness, from which the surjectivity of the group homomorphism in Theorem \ref{thm-1} is derived.
To the authors' best knowledge, we do not know whether $\Psi_{L_{\varphi}}^*$ is fully faithful.
Note that for the categories $\mathit{Aug}_+(\Lambda)$ of Legendrian knots $\Lambda$,
a counterpart of \cite[Theorem 1.1]{Li} is known. 
See Remark \ref{rem-Pan} below.

\begin{remark}\label{rem-Pan}
Let $\Lambda_+$ and $\Lambda_-$ be Legendrian knots in $\bR^3$ with a standard contact structure.
Let $\Sigma$ be an exact Lagrangian cobordisms from $\Lambda_-$ to $\Lambda_+$.
Under certain conditions on $\Sigma$, Pan \cite[Theorem 1.6]{Pan} constructed a functor from $\mathit{Aug}_+(\Lambda_-)$ to $\mathit{Aug}_+(\Lambda_+)$ associated to $\Sigma$
and showed its fully faithfulness.
\end{remark}

We can deduce some information about the knot type of $K$ from $(\mathcal{A}_*(\Lambda_{K}),\partial)$.
The approach in \cite{O23} was to use the augmentation variety $V_K$ defined from the set of augmentations of $(\mathcal{A}_*(\Lambda_{K}),\partial)$ to $\bC$.
As mentioned in Section \ref{subsubsec-intro-KCH}, $V_{K}$ is closely related to KCH representations of $\pi_1(\bR^3\setminus K)$.
Ekholm-Ng-Shende \cite{ENS} considered an enhancement of $(\mathcal{A}_*(\Lambda_{K}),\partial)$ and showed that its homology together with an additional structure recovers the group $\pi_1(\bR^3\setminus K)$.
This is analogous to the process in \cite{Shende} recovering $\pi_1(\bR^3\setminus K)$ from $\Sh_{0_{\bR^3}\cup T^*_K\bR^3}(\bR^3)$.

\subsection{Organization of paper}

In Section \ref{sec-geometry}, we give a geometric setting about clean Lagrangian intersections. We  introduce some geometric objects, such as the $3$-manifold $M_{\varphi}$, and study their basic properties.
In Section \ref{sec-sheaf}, we apply microlocal sheaf theory to this setting.
In Section \ref{subsec-cobordism-functor}, a cobordism functor from \cite{Li} is introduced.
From this functor, we construct in Section \ref{subsec-hom-h} a group homomorphism $h$ from $\pi_1(\bR^3\setminus K_0)$ to $\pi_1(M_{\varphi})$. 
In Section \ref{subsec-surj}, we prove the surjectivity of $h$ (Proposition \ref{prop-h-surj}), together with Proposition \ref{prop-GL_n-rep} about representations of $\pi_1(\bR^3\setminus K_0)$ and $\pi_1(M_{\varphi})$.
In Section \ref{sec-constraint}, we see that the results in Section \ref{subsec-surj} give   constraints on the knot types of clean Lagrangian intersections in terms of knot groups, A-polynomials and KCH-representations.
In particular, combining with \cite[Theorem 1.1]{O25}, Theorem \ref{thm-2} is proved.
In Section \ref{subsec-peripheral}, we upgrade Proposition \ref{prop-GL_n-rep} to Theorem \ref{thm-1} by using $3$-manifold theory.

\

\noindent
\textbf{Acknowledgments.}
We thank Wenyuan Li for helpful explanations of his results and for valuable discussions. 
We are grateful to Yuta Nozaki for explanations and discussions on low-dimensional topology. 
We also thank Yuichi Ike for helpful discussions. 
TA is supported by JSPS KAKENHI Grant Number JP24K16920 and JST, CREST Grant Number JPMJCR24Q1, Japan. 
YO is supported by JSPS KAKENHI Grant Number JP25KJ0270.

\section{Geometric setup}\label{sec-geometry}

\subsection{Notation}
Let $X$ be a $C^{\infty}$ manifold.
Consider the cotangent bundle $T^*X$ of $X$.
For any $x\in X$, we write $(x;\xi) \in T^*X$ to denote $\xi \in T^*_x X$.
Let $\pi_X\colon T^*X \to X$ be the bundle projection.
The total space $T^*X$ has a canonical Liouville $1$-form $\lambda_X\in \Omega^1(T^*X)$ defined by $(\lambda_X)_{(x;\xi)} = \xi \circ (d\pi_X)_{(x;\xi)}$.
We denote the zero section by $0_{X}$ and identify it with $X$.
The group $\bR_+$ acts on $T^*X$ by $r\cdot (x;\xi) = (x;r\xi)$ for every $r\in \bR_+$.
The cosphere bundle of $X$ is defined by $T^{*,\infty}X \coloneqq (T^*X \setminus 0_{X})/\bR_+$.

When a Riemannian metric on $X$ is fixed,
the unit cotangent bundle of $X$ is defined by $U^*X \coloneqq \{(x;\xi)\in T^*X \mid |\xi|=1 \}$. There is a natural diffeomorphism
\[ \psi^{\infty} \colon U^*X \to T^{*,\infty}X ,\ (x;\xi) \mapsto [(x;\xi)].\]
Let us define a $1$-form $\alpha_{X} \coloneqq \rest{ \lambda_X}{U^*X}$ on $U^*X$. Then, it is a contact $1$-form on $U^*X$ and $d \psi^{\infty}(\ker \alpha_X)$ gives a canonical contact structure on $T^{*,\infty}X$.

For any $C^{\infty}$ submanifold $K$ of $X$, its conormal bundle $T^*_K X$ is a Lagrangian submanifold of $T^*X$ defined by
\[ T^*_K X \coloneqq \{(x;\xi)\in T^* X \mid x\in K,\ \la \xi, v \ra= 0 \text{ for every }v\in T_x K \}. \]
We define the unit conormal bundle of $K$ by
\[\Lambda_K \coloneqq T^*_KX \cap U^*X. \]
Since $\rest{\alpha_X}{\Lambda_K} \in \Omega^1(\Lambda_K)$ vanishes, $\Lambda_K$ is a Legendrian submanifold of $(U^*X,\alpha_X)$.

Let $H\colon T^*X \times [0,1] \to \bR ,\ ((x;\xi),s) \mapsto H_s(x;\xi)$ be a $C^{\infty}$ Hamiltonian with compact support.
Then, a flow $(\varphi^H_s)_{s\in [0,1]}$ on $T^*X$ is generated by a time-dependent vector field $(X_s)_{s\in[0,1]}$ determined by $(d\lambda_X)(\cdot , X_s) = d H_s$.
We define
\begin{align*} &\Ham_c(T^* X) \\
&\coloneqq \{ \varphi = \varphi_1^H \mid H\colon T^*X \times [0,1] \to \bR \text{ is a }C^{\infty} \text{ Hamiltonian with compact support} \} . \end{align*}

\subsection{Clean Lagrangian intersection along knots}\label{subsec-clean}

As a $C^{\infty}$ manifold of dimension $3$,
we focus on $\bR^3$. It is equipped with a standard metric $\sum_{i=1}^3 dx_i\otimes dx_i$, where $(x_1,x_2,x_3)$ is the coordinate of $\bR^3$. This metric fixes an identification between $T^*\bR^3$ and $T \bR^3$.
For any $r>0$, let us denote
$D_{r}^*\bR^3 \coloneqq\{ (x;\xi) \in T^*\bR^3 \mid |\xi| < r \}$.

We recall the definition of clean intersection of submanifolds.
\begin{definition}\label{def-clean-int}
Let $L_0$ and $L_1$ be $C^{\infty}$ submanifolds of $T^*\bR^3$.
We say $L_0$ and $L_1$ have a \textit{clean intersection} along $L_0\cap L_1$ if the following hold:
\begin{itemize}
\item $L_0\cap L_1$ is a $C^{\infty}$ submanifold of both $L_0$ and $L_1$.
\item $T_p(L_0\cap L_1)= T_p L_0 \cap T_pL_1$ for every $p\in L_0\cap L_1$.
\end{itemize}
\end{definition}

Let us give a geometric setup for Question \ref{Q-1}.
Let $K_0$ and $K_1$ be knots in $\bR^3$.
For $\varphi \in \Ham_c(T^*\bR^3)$, we suppose that $\varphi(T^*_{K_0}\bR^3)$ and $0_{\bR^3}$ have a clean intersection along $K_1$ in $0_{\bR^3}\cong \bR^3$
(recall Figure \ref{figure-clean}).
By \cite[Lemma 2.3]{O23}, there exists $\varphi' \in \Ham_c(T^*\bR^3)$ supported in a neighborhood of $K_1$ in $T^*\bR^3$ such that $(\varphi'\circ \varphi) (T^*_{K_0}\bR^3) $ coincides with $T^*_{K_1}\bR^3 $ in a neighborhood of $0_{\bR^3}$.
By rewriting $\varphi'\circ \varphi$ as $\varphi$,
we may assume that for sufficiently small $\varepsilon>0$,
\[  \varphi(T^*_{K_0}\bR^3) \cap D_{\varepsilon}^*\bR^3 = T^*_{K_1}\bR^3 \cap D_{\varepsilon}^*\bR^3 . \]
See the left picture of Figure \ref{figure-surgery}.

We consider a diffeomorphism
\begin{align}\label{psi-symplectization}
\psi \colon \bR\times U^*\bR^3 \to T^*\bR^3 \setminus 0_{\bR^3} ,\ (r, (x;\xi)) \mapsto (x;e^r\xi).
\end{align}
Then, $\bR\times U^*\bR^3$ together with the $1$-form $\psi^*(\lambda_{\bR^3}) = e^r \alpha_{\bR^3}$ gives the symplectization of the contact manifold $(U^*\bR^3,\alpha_{\bR^3})$.
Consider an exact Lagrangian submanifold of $T^*\bR^3\setminus 0_{\bR^3}$
\[ L_{\varphi} \coloneqq \varphi(T^*_{K_0}\bR^3) \setminus K_1. \]
Then, $\psi^{-1}(L_{\varphi})$ is an exact Lagrangian submanifold of the symplectization of $U^*\bR^3$.

Moreover, as discussed in \cite[Section 2.2]{O23}, $\psi^{-1}(L_{\varphi})$ is an exact Lagrangian cobordism in $\bR\times U^*\bR^3$ whose positive end is $\Lambda_{K_0}$ and negative end is $\Lambda_{K_1}$.
Indeed, there exist $a_0, a_1 \in \bR$ with $a_0 > a_1$ such that $\psi^{-1}(L_{\varphi}) \cap ([a_1,a_0]\times U^*\bR^3)$ is compact and
\begin{align*}
\psi^{-1}(L_{\varphi}) \cap ([a_0,\infty) \times U^*\bR^3 ) & = [a_0,\infty) \times \Lambda_{K_0} , \\
\psi^{-1}(L_{\varphi}) \cap ( (-\infty,a_1] \times U^*\bR^3 ) & = (-\infty,a_1] \times \Lambda_{K_1}.
\end{align*}
In addition, if we take a function $f\colon \psi^{-1}(L_{\varphi}) \to \bR$ such that $df = \rest{(e^r \alpha_{\bR^3})}{\psi^{-1}(L_{\varphi})}$, then $df = 0$ on $(-\infty,a_1]\times \Lambda_{K_1}$. Since $\Lambda_{K_1}$ is connected, $f$ is constant on $(-\infty,a_1]\times \Lambda_{K_1}$.

Let us prepare several notations for $i\in \{0,1\}$:
\begin{itemize}
\item We have an embedding
\[ \sigma_i \colon \Lambda_{K_i} \to L_{\varphi} ,\ (x;\xi) \mapsto \psi ( a_i, (x;\xi) ) = (x;e^{a_i}\xi).  \]
$\sigma_i$ can also be considered as an embedding $\sigma_i\colon \Lambda_{K_i} \to T^*_{K_i} \bR^3\setminus 0_{\bR^3}$.
\item 
We take a tubular neighborhood of $K_i$ in $\bR^3$ by
\[ N_{K_i} \coloneqq \{ x+v \mid x\in K_i,\ v\in (T_xK)^{\perp} \text{ and } |v| < \varepsilon \} \]
for sufficiently small $\varepsilon >0$.
Via the identification $T^*\bR^3 \cong T\bR^3$ fixed by the standard metric, we define 
\begin{align}\label{emb-ei} 
e_i\colon \Lambda_{K_i} \to \partial N_{K_i} ,\ (x;\xi) \mapsto x+ \varepsilon \cdot \xi.
\end{align}
\item 
Fix a base point $y_i$ of $\Lambda_{K_i}$. 
We also fix an orientation of the knot $K_i$ and take its Seifert framing.
Then, we take $m_i,l_i \in \pi_1(\Lambda_{K_i},y_i)$
such that the elements
\[(e_i)_*(m_i), (e_i)_*(l_i) \in \pi_1(\partial N_{K_i} , e_i(y_i))\]
give the meridian and the longitude respectively of the framed knot $K_i$.
\end{itemize}

Let us also denote by $m_i,l_i\in H_1(\Lambda_{K_i};\bZ)$ the homology classes corresponding to $m_i,l_i \in \pi_1(\Lambda_{K_i},y_i)$.
It is easy to see that
$H_1(L_{\varphi};\bZ)$ is a free abelian group with a basis
\[\{(\sigma_1)_*(m_1),(\sigma_0)_*(l_0)\}.\]
The homology classes $(\sigma_0)_*(m_0)$ and $(\sigma_1)_*(l_1)$ can be described as follows.
For the proof, see \cite[Lemma 2.6 and the equation (4) in its proof]{O25}.
\begin{proposition}\label{prop-H1L}
There exist $a\in \{1,-1\}$ and $b\in \bZ$ such that
\[ \begin{array}{cc} (\sigma_0)_*(m_0)= a\cdot (\sigma_1)_*(m_1), & (\sigma_1)_*(l_1) = a\cdot (\sigma_0)_*(l_0) + b\cdot (\sigma_1)_*(m_1), \end{array}\]
in $H_1(L_{\varphi};\bZ)$.
\end{proposition}

This proposition shows that $(\sigma_i)_* \colon H_1(\Lambda_{K_i};\bZ) \to H_1(L_{\varphi};\bZ)$ is an isomorphism for $i=0,1$.
In addition, as homology classes in $H_1(\varphi (T^*_{K_0}\bR^3))$, we have
\[\begin{array}{ccc}(\sigma_1)_*(l_1)=[K_1] , & (\sigma_0)_*(l_0) = \varphi_*([K_0]), & (\sigma_1)_*(m_1)=0 ,\end{array}\]
and thus, the second equation of Proposition \ref{prop-H1L} implies that
\begin{align}\label{K1_K0_homologous} 
[K_1] = a\cdot  \varphi_*([K_0])
\end{align}
in $H_1(\varphi (T^*_{K_0}\bR^3))$ for $a=\pm 1$.

Lastly, let us introduce a $3$-manifold $M_{\varphi}$ which is used to state Theorem \ref{thm-1}.
We have two $3$-manifolds with boundary:
\[\begin{array}{cc} \bR^3 \setminus N_{K_1} , & L_{\varphi} \setminus \psi((-\infty,a_1)\times \Lambda_{K_1}). \end{array}\]
Gluing their boundaries by the diffeomorphism $\sigma_1\circ e_1^{-1} \colon \partial N_{K_1} \to \psi(\{a_1\} \times \Lambda_{K_1})$, we obtain a non-compact $3$-manifold without boundary
\begin{align}\label{glue-M_phi}
M_{\varphi} \coloneqq \left( \bR^3 \setminus N_{K_1} \right) \cup_{\sigma_1\circ e_1^{-1}} \left( L_{\varphi} \setminus \psi((-\infty,a_1)\times \Lambda_{K_1}) \right). 
\end{align}
Let us take $y'_1 \coloneqq \sigma_1(y_1) = e_1(y_1) \in M_{\varphi}$ as a base point.
We note that
\[ H_1(M_{\varphi};\bZ) \cong H_1(\bR^3 \setminus K_1;\bZ) \cong \bZ\]
since $(\sigma_1\circ e_1^{-1})_*\colon H_1(\partial N_{K_1};\bZ) \to H_1(L_{\varphi})$ is an isomorphism by Proposition \ref{prop-H1L}.

The next result will be used several times. Let $D^2\subset \bR^2$ be the unit disk centered at $0$.
\begin{theorem}[{\cite[Theorem 1.1]{Gab}}]\label{thm-Gabai}
Let $K$ be a knot in $S^1\times D^2$ such that $[K] = \pm [S^1\times \{0\}]$ in $H_1(S^1\times D^2;\bZ)$.
If $M$ is a $3$-manifold obtained from $S^1\times D^2$ by a non-trivial surgery on $K$, then one of the following holds:
\begin{enumerate}
\item $M=S^1\times D^2$. In this case, $K$ is isotopic to $S^1\times \{0\}$.
\item $M=M'\# W$, where $W$ is a closed $3$-manifold and $H_1(W;\bZ)$ is finite and non-trivial.
\item $M$ is irreducible and $\partial M$ is incompressible. 
\end{enumerate}
\end{theorem}
We remark that in the case $M=S^1\times D^2$, \cite[Theorem 1.1]{Gab} states that $K$ is either a $0$ or $1$-bridge braid in $S^1\times D^2$ under a weaker assumption on $K$. (For the definition of a $k$-bridge braid for $k\in \{0,1\}$, see \cite[\S 1]{Gab}. It is obtained as the closure of a braid.)
Assuming $[K] = \pm [S^1\times \{0\}]$, the algebraic intersection number of $K$ and $\mathrm{pt}\times D^2$ is $\pm 1$, and the $0$ or $1$-bridge braid $K$ must be isotopic to $S^1\times \{0\}$.

The next property of the fundamental group of $M_{\varphi}$ is important for the discussion in Section \ref{subsec-hom-h}, especially Proposition \ref{prop-h-mon}.
Recall that a group $G$ is called \textit{left-orderable} if it has a total ordering $<$ such that $g<h$ implies $fg<fh$ for any $f,g,h\in G$.
\begin{lemma}\label{lem-LO}
The group $\pi_1(M_{\varphi},y'_1)$ is left-orderable.
\end{lemma}
\begin{proof}
For convenience, we use a compactification $S^3 = \bR^3 \cup \{\infty\}$ of $\bR^3$ and consider a $3$-manifold
\[M'_{\varphi} \coloneqq \left( S^3 \setminus N_{K_1} \right) \cup_{\sigma_1\circ e_1^{-1}} \left( L_{\varphi} \setminus \psi((-\infty,a_1)\times \Lambda_{K_1}  \right). \]
The fundamental group of $M'_{\varphi}$ is isomorphic to $\pi_1(M_{\varphi},y'_1)$.

First, we observe the $3$-manifold $L_{\varphi}$.
By definition, $L_{\varphi}$ is homeomorphic to
\[T^*_{K_0}\bR^3 \setminus \varphi^{-1}(K_1)\]
via $\varphi$.
Through a trivialization $T^*_{K_0}\bR^3 \cong K_0\times \bR^2$, $\varphi^{-1}(K_1)$ is viewed as a knot in $K_0 \times \bR^2$ and it is homologous to $K_0 \times \{0\}$ up to sign by (\ref{K1_K0_homologous}).
Therefore, $L_{\varphi}$ is homeomorphic to the complement of a link $U \sqcup K'_1$ in $S^3$, where $U$ is the unknot and $K'_1$ is a knot whose linking number with $U$ is $\pm 1$ (in particular, nonzero).
This implies that $L_{\varphi}$ is an irreducible $3$-manifold.

Moreover, $(\sigma_1)_* \colon H_1(\Lambda_{K_1}) \to H_1(L_{\varphi})$ is an isomorphism by Proposition \ref{prop-H1L}, and thus the group homomorphism
\[ (\sigma_1)_*\colon \pi_1(\Lambda_{K_1}) \to \pi_1(L_{\varphi}) \]
is injective.
This means that the embedded torus $\sigma_1(\Lambda_{K_1})$ in $L_{\varphi}$ is an incompressible surface.

Next, we show that $M'_{\varphi}$ is an irreducible $3$-manifold.
The proof is divided into the following two cases:
\begin{itemize}
\item If $K_1$ is the unknot in $S^3$, then $S^3\setminus N_{K_1}$ is homeomorphic to a solid torus.
Therefore, $M'_{\varphi}$ is a $3$-manifold obtained from the (open) solid torus
\[ K_0\times \bR^2 \cong S^1 \times \{x\in \bR^2 \mid |x|<1\} \]
by a Dehn surgery along the knot $\varphi^{-1}(K_1)$.
Since $[\varphi^{-1}(K_1)] = \pm [K_0\times \{0\}]$ as we have seen above,
we can apply Theorem \ref{thm-Gabai}.
Since $H_1(M_{\varphi};\bZ) \cong \bZ$ and in particular has no torsion, $M'_{\varphi}$ is not a connected sum $ M'' \# W$, where $W$ is a closed $3$-manifold such that $H_1(W;\bZ)$ is finite and nontrivial. 
Therefore, it follows from Theorem \ref{thm-Gabai} that either (1) $M'_{\varphi}= S^1 \times \bR^2$ or (2) $M'_{\varphi}$ is irreducible  and $\sigma_0(\Lambda_{K_0})$ is incompressible.
In both cases, $M'_{\varphi}$ is irreducible.

\item If $K_1$ is knotted in $S^3$, then $S^3 \setminus K_1$ is an irreducible $3$-manifold and the embedded torus $\partial N_{K_1}$ is an incompressible surface.
Therefore, $M'_{\varphi}$ is separated into two irreducible $3$-manifolds
\[\begin{array}{cc}
S^3 \setminus N_{K_1}, & L_{\varphi} \setminus \psi((-\infty,a_1)\times \Lambda_{K_1} )\end{array}\]
along an embedded torus which is incompressible in both components.
From an argument as in \cite[page 11, (4)]{hatcher}, it follows that $M'_{\varphi}$ is also irreducible.
\end{itemize}

Now, we can apply a result in \cite{BRW}.
Since the orientable $3$-manifold $M'_{\varphi}$ is irreducible and there is a surjective group homomorphism $\pi_1(M'_{\varphi}) \to H_1(M'_{\varphi};\bZ) \cong \bZ$, it follows from \cite[Theorem 1.1 (1)]{BRW} that the group $\pi_1(M'_{\varphi})$ is left-orderable.
\end{proof}

\section{Application of a cobordism functor in sheaf theory}\label{sec-sheaf}

We fix the ground field $\bfk$ to be the field $\bC$ of complex numbers. 
However, all arguments prior to Proposition~\ref{prop-h-surj} are valid over any field. 

\subsection{Preliminaries}\label{subsec-prelim}

In this subsection, we review the necessary background on microlocal sheaf theory. 
Let $\Mod_\bfk$ be the $\infty$-category of complexes of $\bfk$-vector spaces, which is naturally equivalent to the derived $\infty$-category of $\bfk$ since $\bfk$ is a field.

We use the language of ($\Mod_\bfk$-enriched stable) $\infty$-categories in this paper. 
In order for the Kashiwara--Schapira stack defined below to satisfy the desired descent condition, a purely $1$-categorical treatment is insufficient, and one needs to work either with dg categories or with $\infty$-categories. 
In this paper, for objects $x,y$ of a $\Mod_\bfk$-enriched stable category, the notation $\Hom (x,y)$ will refer the corresponding object of $\Mod_\bfk$. 
For the theory of $\infty$-categoical enrichment, see Appendix C of \cite{HeyerMann} and the references therein.

\begin{remark}
Some of the results we cite are originally stated in terms of classical derived categories or dg categories.
Moreover, while \cite{KS90} works with classical bounded derived categories, we do not impose the corresponding boundedness conditions in this paper.

The treatment in terms of dg categories can be naturally reformulated in the language of $\infty$-categories. 
See \cite[Section~1.3.1]{LurieHA} and \cite{Cohn} for details. 
On the other hand, results formulated in the classical setting  may not automatically be interpreted into $\infty$-categorical results. 
At present, it is known from the results of \cite{robalo2018lemma} that if the coefficient category of sheaves is compactly generated, then almost all results of \cite{KS90} hold in the $\infty$-categorical setting without imposing boundedness conditions. 
\end{remark}

Let $X$ be a $C^{\infty}$ manifold and let $\Sh (X)$ denote the $\Mod_\bfk$-enriched stable category of sheaves of complexes of $\bfk$-vector spaces on $X$. 
For any $F \in \Sh(X)$, its microsupport $\CMS (F) $ is defined as in \cite[Definition 5.1.2]{KS90} to be the closure of the set
\[ \left\{ (x;\xi)\in T^*X\ \middle|
\begin{array}{l} \text{there exists a } C^{1} \text{ function } f \text{ on a neighborhood of }x \\
\text{such that } f(x)=0,\ (df)_x=\xi \text{ and }\left(\Gamma_{\{f\geq 0\}} F \right)_x\neq 0 \end{array}   \right\}. \]
It is a conical closed subset of $T^*X$.
For any closed subset $A$ of $T^*X$, 
let $\Sh_{A}(X)$ denote the full subcategory of $\Sh(X)$ consisting of $F\in \Sh(X)$ such that $\CMS (F) \subset A$.

Let $\Loc(X)$ denote the full subcategory of $\Sh(X)$ consisting of $F \in \Sh(X)$ whose cohomology sheaf $H^*(F)$ is locally constant on $X$.

We recall the definition of the Kashiwara--Schapira stack. 
Here, a stack means a sheaf valued in the ($\infty$-)category of the $\Mod_\bfk$-enriched stable categories and the exact functors. 
The underlying notion for this stack was investigated by Kashiwara--Schapira~\cite{KS90}. 
Here we follow the definition given by Nadler--Shende~\cite[Section 6]{NadlerShende20}. 

\begin{definition}\label{def-KS-stack}
The \textit{Kashiwara--Schapira stack} $\mush$ on $T^*X$
is the conic sheaf valued in the ($\infty$-)category of $\Mod_\bfk$-enriched stable categories associated to a conic presheaf $\mush^{\mathrm{pre}}$ defined by 
\[\mush^{\mathrm{pre}}(\Omega) \coloneqq \Sh (X) / \Sh_{T^*X \setminus \Omega}(X)\]
for each conical open subset $\Omega\subset T^*X$. 
For $F\in \mush(\Omega)$, the closed conical subset $\CMS (F)\cap \Omega$ of $\Omega $ is well-defined. 

Let $\wh{\Lambda}$ be a locally closed conical subset of $T^*X$.
A substack $\mush'_{\wh{\Lambda}}$ of $\mush$ is given by
\[\mush'_{\wh{\Lambda}}(\Omega)\coloneqq \{F\in \mush(\Omega) \mid \CMS (F)\cap \Omega\subset \wh{\Lambda}\cap \Omega \}.\]
We define the \textit{Kashiwara--Schapira stack} $\mush_{\wh{\Lambda}}$ on $\wh{\Lambda}$ as the inverse image $i_{\wh{\Lambda}}^*\mush'_{\wh{\Lambda}}$ of the substack $\mush'_{\wh{\Lambda}}$ by the inclusion $i_{\wh{\Lambda}}\colon\wh{\Lambda}\to T^*X$. 
\end{definition}

For an open subset $\wh{\Lambda}_0\subset \wh{\Lambda}$ and $F,G \in \mush_{\wh{\Lambda}}(\wh{\Lambda}_0)$, the assignment
\[
\wh{\Lambda}_0\supset U \mapsto \Hom_{\mush_{\wh{\Lambda}}(U)}(\rest{F}{U}, \rest{G}{U})
\]
satisfies a sheaf property and defines an object $\overline{\muhom}(F,G)$ of $\Sh (\wh{\Lambda}_0)$. 

For $F, G\in \Sh (X)$, $\muhom (F,G)\in \Sh (T^*X)$ is defined by Kashiwara--Schapira~\cite[Section 4.4]{KS90}.
The support of $\muhom (F,G)$ is contained in $\CMS (F)\cap \CMS (G)$. 
Two functors $\muhom$ and $\overline{\muhom}$ are closely related. 
Note that for $F\in \Sh (X)$ and an open subset $\wh{\Lambda}_0\subset \wh{\Lambda}$, if there is an open subset $\Omega_0\subset T^*X$ satisfying $\wh{\Lambda}_0 \subset\Omega_0$ and $\CMS (F)\cap \Omega_0\subset \wh{\Lambda}_0$, 
$F$ naturally defines an object $\frakm_{\wh{\Lambda}_0}(F)\in \mush_{\wh{\Lambda}}(\wh{\Lambda}_0)$. 

\begin{proposition}[{\cite[Theorem~6.1.2]{KS90}}]\label{prop-stack-hom}
    Let $F, G$ be objects of $\Sh (X)$ and $\wh{\Lambda}_0$ be an open subset of $\wh{\Lambda}$. Assume that there is an open subset $\Omega_0\subset T^*X$ satisfying  $\wh{\Lambda}_0 \subset\Omega_0$ and $(\CMS (F)\cup \CMS(G))\cap \Omega_0\subset \wh{\Lambda}_0$. Then there is a natural isomorphism 
    \[
    i_{\wh{\Lambda}_0}^*\muhom(F,G)\simeq \overline{\muhom}(\frakm_{\wh{\Lambda}_0}(F),\frakm_{\wh{\Lambda}_0}(G))
    \]
    where $i_{\wh{\Lambda}_0}\colon \wh{\Lambda}_0\to T^*X$ is the inclusion. 
\end{proposition}

Theorem~6.1.2 in \cite{KS90} is a statement about the stalks. 
Proposition~\ref{prop-stack-hom} is deduced from the theorem in \cite{KS90} since an isomorphism between sheaves can be checked on stalks in the present situation, where the base space is a manifold and the coefficient category $\Mod_\bfk$ is compactly generated. 
For this part, it is also essential that the homotopy category of the stalk category of the Kashiwara--Schapira stack corresponds to the category $\mathrm{D}^b(X;p)$ treated in \cite{KS90}. 

Hereafter, we mainly consider the stable category $\mush_{\wh{\Lambda}}(\wh{\Lambda})$ together with
a natural functor
\[ \frakm_{\wh{\Lambda}} \colon \Sh_{\wh{\Lambda}} (X) \to \mush_{\wh{\Lambda}}(\wh{\Lambda}) ,\  F\mapsto \frakm_{\wh{\Lambda}}(F). \]

Consider the case where $\wh{\Lambda}$ is a conical smooth Lagrangian submanifold. 
For an open subset $\wh{\Lambda}_0\subset \wh{\Lambda}$, an object $\scrF \in \mush_{\widehat{\Lambda}}(\widehat{\Lambda}_0)$ is said to be \emph{simple} (along $\widehat{\Lambda}_0$) if the identity morphism $\bfk_{\widehat{\Lambda}_0}\to \overline{\muhom}(\scrF,\scrF)$ is an isomorphism.  
By \cite[Proposition X.2.4]{Gui23}, if there exists a simple global object $\scrF\in \mush_{\widehat{\Lambda}}(\widehat{\Lambda})$, then we have a functor
\[ \overline{\muhom}(\scrF,\cdot) \colon \mush_{\widehat{\Lambda}} (\widehat{\Lambda}) \to \Loc (\widehat{\Lambda}),\]
and moreover, it gives an equivalence of categories.

Let $L$ be an exact Lagrangian of $T^*X$, not necessary conical or smooth.
We suppose that $L$ is connected and has a  stratification as $L= \coprod_{i=1}^m L_i$ by a family of smooth Lagrangian submanifolds $\{L_i\}_{i=1,\dots ,m}$.
Moreover, assume that there exists a continuous function $f\colon L \to \bR$ such that $\rest{f}{L_i}$ is a $C^{\infty}$ function and $d( \rest{f}{L_i} ) = \rest{\lambda_X}{L_i}$ for every $i\in \{1,\dots ,m\}$.
Note that $f$ is unique up to constant if it exists since $L$ is connected.
For such $L$, we construct a conical Lagrangian  $\wh{L}$ of $T^*(X\times \bR)$ by
\begin{align}\label{lift-of-Lagrangian} \wh{L} \coloneqq \coprod_{i=1}^m \{ (x,-f(x;\xi) ; \tau \xi,\tau) \in T^*(X\times \bR) \mid \tau >0,\ (x;\xi) \in L_i \} .
\end{align}
Then, the Kashiwara--Schapira stack $\mush_{\wh{L}}$ on $\wh{L}$ is defined as in Definition \ref{def-KS-stack}.
For simplicity, let us denote $\mush_{\wh{L}}(\wh{L})$ and $\frakm_{\wh{L}}$ by $\mush_L(L)$ and $\frakm_L$, respectively.
\begin{remark} If $L$ is conical Lagrangian in $T^*X$, we take $f\colon L\to \bR$ to be the constant function to $0$.
Then, using the map $i\colon X \to X\times \bR ,\  x \mapsto (x,0)$, we have an equivalence $i^* \colon \Sh_{\wh{L}}(X\times \bR) \overset{\sim}{\to} \Sh_{L}(X)$. Identifying $\Sh_{\wh{L}}(X\times \bR)$ with $\Sh_L(X)$ by $i^*$, we consider $\frakm_L$ as a functor defined on $\Sh_L(X)$.
Moreover, if $L$ contains $0_{X}$, this functor \[\frakm_L\colon \Sh_L(X) \to \mush_L(L)\]
is an equivalence.
\end{remark}

\subsection{The conormal of knots}\label{subsec-conormal}
We focus on the case where $X=\bR^3$.
For a knot $K$ in $\bR^3$, let us denote
\begin{align*}
L_K  & \coloneqq 0_{\bR^3} \cup T^*_K\bR^3, \\
\wh{\Lambda}_K & \coloneqq T^*_K\bR^3 \setminus 0_{\bR^3} = \bR_+\cdot \Lambda_K.
\end{align*}
In the same way as (\ref{emb-ei}) for $K_i$, we define $ e \colon \Lambda_K \to \bR^3 \setminus K$ to be an embedding onto the boundary of a tubular neighborhood of $K$.

Any simple object $\scrF \in \mush_{\wh{\Lambda}_K}(\wh{\Lambda}_K)$ defines a functor
\[ \overline{\muhom}(\scrF,\cdot) \colon \mush_{\wh{\Lambda}_K}(\wh{\Lambda}_K) \to \Loc(\wh{\Lambda}_K) \simeq \Loc(\Lambda_K) . \]
The last equivalence is induced by the inclusion map $\Lambda_K\to \wh{\Lambda}_K$, which is a homotopy equivalence.
We also have the restriction functor $\operatorname{res} \colon \mush_{L_K}(L_K) \to \mush_{\wh{\Lambda}_K}( \wh{\Lambda}_K)$.

We set $\scrF_K  \coloneqq \frakm_{\wh{\Lambda}_K} (\bfk_K [-1]) \in \mush_{\wh{\Lambda}_K}(\wh{\Lambda}_K)$.
Then, the composition of the functors
\[\xymatrix@C=45pt{
\Sh_{L_K}(\bR^3) \ar[r]^-{\bfm_{L_K}} & \mush_{L_K} (L_K) \ar[r]^-{\operatorname{res}} & \mush_{\wh{\Lambda}_K}(\wh{\Lambda}_K) \ar[r]^-{\overline{\muhom}(\scrF_K,\cdot )} & \Loc(\Lambda_K)
} \]
coincides with $\rest{\mu_K(\cdot )}{\Lambda_K}[1]$, where $\mu_K\colon \Sh(\bR^3) \to \Sh(T^*_K\bR^3)$ is the microlocalization functor~\cite[Section~4.3.]{KS90}.
Indeed, for any $F\in \Sh_{L_K}(\bR^3)$,
\begin{align}\label{muhom-mu_K} 
\rest{ \overline{\muhom}( \scrF_K , \rest{ (\frakm_{L_K} (F)) }{\wh{\Lambda}_K})}{\Lambda_K} = \rest{ \muhom(\bfk_K[-1], F)}{\Lambda_K} = \rest{\mu_K(F)}{\Lambda_K}[1] . 
\end{align}
To compute the last term, we prepare the following proposition. 
Similar statements and arguments for the proofs appear in \cite[Sections~4--5]{Shende}. 

\begin{proposition}\label{prop-muK-conormal}
Let us denote inclusion maps by $K \overset{i}{\rightarrow} \bR^3 \overset{j}{\leftarrow} \bR^3\setminus K$. Then, there are two equivalences of functors:
\begin{align*}
\rest{ \mu_K(i_*(\cdot))}{\Lambda_K} & \simeq (\rest{\pi_{\bR^3}}{\Lambda_K})^* \colon \Loc(K)\to \Loc(\Lambda_K), \\
\rest{ \mu_K(j_!(\cdot)) }{\Lambda_K} & \simeq e^*(\cdot)[-1]\colon \Loc(\bR^3\setminus K)\to \Loc(\Lambda_K).
\end{align*}
\end{proposition}
\begin{proof}
(1) For any $F\in \Loc(K)$, $\CMS(i_*F) \subset T^*_K\bR^3$.
Using \cite[Theorem~6.4.1]{KS90}, one can check that the microsupport of $\mu_K(i_* F)\in \Sh(T^*_K\bR^3)$ is contained in the zero section of $T^*(T^*_K\bR^3)$. 
Hence $\mu_K(i_* F)$ is locally constant on $T^*_K\bR^3$. 
Note also that $\rest{\mu_K(i_* F)}{0_K}\simeq F$ on $K\cong 0_K=T^*_K\bR^3\cap 0_{\bR^3}$ by \cite[Theorem~4.3.2 (iv)]{KS90}. 
Since the projection $\rest{\pi_{\bR^3}}{T^*_K\bR^3}\colon T^*_K\bR^3\to K$ is a homotopy equivalence, 
we obtain
\[\mu_K(i_* F)\simeq (\rest{\pi_{\bR^3}}{T^*_K\bR^3})^*(F),\]
from which the first assertion follows.

(2) For any $F\in \Loc(\bR^3\setminus K)$, $\CMS(j_!F)\setminus 0_{\bR^3} \subset T^*_K\bR^3\setminus 0_K$.
From a microlocal estimate using \cite[Theorem 6.4.1]{KS90} as above, it follows that $\rest{\mu_K(j_!F)}{T^*_K\bR^3\setminus 0_K}$ is locally constant on $T^*_K\bR^3\setminus 0_K$. 

Consider the normalized geodesic flow $\phi=(\phi_t)_{t\in \bR}$ on $T^*\bR^3\setminus 0_{\bR^3}$.
It is a homogeneous Hamiltonian isotopy, and we have the GKS kernel
$K^\phi\in \Sh(\bR^3\times \bR^3\times \bR)$ associated to $\phi$ \cite[Section 3]{GKS}.
Let us fix  $\varepsilon>0$ and put $K^\phi_{-\varepsilon} \coloneqq \rest{K^\phi}{\{t=-\varepsilon\}}\in \Sh(\bR^3\times \bR^3)$.
We note that $ K^\phi_{-\varepsilon} \simeq \bfk_{\{(x_1,x_2)\in \bR^3\times \bR^3\mid |x_1-x_2|\leq \varepsilon\}}$.
If $\varepsilon>0$ is sufficiently small, 
$N_K\coloneqq \{ x'\in \bR^3 \mid \min_{x\in K}|x-x'|< \varepsilon \}$
is a tubular neighborhood of $K$ and we can compute that
\[ \begin{array}{cc} K^\phi_{-\varepsilon}\circ\bfk_K\simeq \bfk_{\overline{N}_K} , &  K^\phi_{-\varepsilon}\circ (j_!F) \simeq F_{\bR^3\setminus \overline{N}_K}, \end{array}
\]
where $\overline{N}_K$ is the closure of $N_K$.
Moreover, by \cite[Theorem 7.2.1]{KS90}, 
\begin{align}\label{contact-transf} 
\begin{split}
(\phi_{-\varepsilon})_*\left( \rest{\muhom (\bfk_K,j_! F)}{U^*\bR^3}\right) & \simeq \rest{ \muhom \left( K^{\phi}_{-\varepsilon}\circ \bfk_K, K^{\phi}_{-\varepsilon}\circ (j_! F)\right)}{U^*\bR^3} \\
& \simeq  \rest{\muhom (\bfk_{\overline{N}_K}, F_{\bR^3\setminus \overline{N}_K})}{U^*\bR^3},
\end{split}
\end{align}
which is supported on $\phi_{\varepsilon}^{-1}(\Lambda_K)$.

By (a version of) Sato's distinguished
triangle \cite[(I.3.5)]{Gui23}, 
\[\xymatrix@C=12pt{
 \cHom(\bfk_{\overline{N}_K}, \bfk_{\bR^3})\otimes F_{\bR^3\setminus \overline{N}_K} \ar[r] & \cHom(\bfk_{\overline{N}_K}, F_{\bR^3\setminus \overline{N}_K}) \ar[r] & 
 \hat{\pi}_* \left( \rest{\muhom (\bfk_{\overline{N}_K}, F_{\bR^3\setminus \overline{N}_K})}{U^*\bR^3} \right) &
}\]
is a (co)fiber sequence in $\Sh (\bR^3)$, 
where $\hat{\pi} \coloneqq \rest{\pi_{\bR^3}}{U^*\bR^3}$.
We have
\[\cHom(\bfk_{\overline{N}_K}, \bfk_{\bR^3})\otimes F_{\bR^3\setminus \overline{N}_K} \simeq \bfk_{N_{K}} \otimes F_{\bR^3\setminus \overline{N}_K}  \simeq 0,\] so the second arrow in the sequence is an isomorphism.
Since $\pi'\coloneqq \rest{\hat{\pi}}{\phi_{\varepsilon}^{-1}(\Lambda_K)} \colon \phi_{\varepsilon}^{-1}(\Lambda_K) \to \partial \overline{N}_K$ is a homeomorphism, we obtain
\begin{align*}
 \rest{ \cHom(\bfk_{\overline{N}_K}, F_{\bR^3\setminus \overline{N}_K})  }{\partial \overline{N}_K} \simeq \pi'_*\left( \rest{\muhom (\bfk_{\overline{N}_K}, F_{\bR^3\setminus \overline{N}_K})}{\phi_{\varepsilon}^{-1}(\Lambda_K)}\right) .
\end{align*}
In addition,
\begin{align*} \cHom(\bfk_{\overline{N}_K}, F_{\bR^3\setminus \overline{N}_K}) & \simeq \cHom(\bfk_{\partial \overline{N}_K}, F_{\bR^3\setminus \overline{N}_K}) \\
& \simeq \cHom(\bfk_{\partial \overline{N}_K}, F) \\
& \simeq \Gamma_{\partial \overline{N}_K} (F) \simeq  F_{\partial \overline{N}_K}[-1].
\end{align*}
Here, the first and the second isomorphisms follow from $\cHom (\bfk_{N_K}, F_{\bR^3\setminus \overline{N}_K})\simeq 0$ and $\cHom(\bfk_{\partial \overline{N}_K}, F_{\overline{N}_K})\simeq 0$ for $F\in \Loc(\bR^3\setminus K)$, both of which are confirmed by local calculations.   
The last isomorphism follows from \cite[Corollary~5.4.11]{KS90}. Therefore,
\begin{align}\label{muhom-phiLambdaK}
i'^*F[-1] \simeq \pi'_*\left( \rest{\muhom (\bfk_{\overline{N}_K}, F_{\bR^3\setminus \overline{N}_K})}{\phi_{\varepsilon}^{-1}(\Lambda_K)}\right) 
\end{align}
for the inclusion map $i'\colon \partial \overline{N}_K\to \bR^3\setminus K$.

Now, we obtain
\begin{align*} 
\rest{\mu_K(j_! F)}{\Lambda_K} & \simeq \rest{\muhom(\bfk_K , j_!F)}{\Lambda_K} \\ 
& \simeq (\phi_{\varepsilon})_* \left( \rest{\muhom (\bfk_{\overline{N}_K}, F_{\bR^3\setminus \overline{N}_K})}{\phi_{\varepsilon}^{-1}(\Lambda_K)} \right) \\
& \simeq (\phi_{\varepsilon}\circ \pi'^{-1})_* \left( \pi'_*   \left( \rest{\muhom (\bfk_{\overline{N}_K}, F_{\bR^3\setminus \overline{N}_K})}{\phi_{\varepsilon}^{-1}(\Lambda_K)} \right) \right) \\
& \simeq ( \phi_{\varepsilon}\circ \pi'^{-1})_* ( i'^* F ) [-1].
\end{align*}
Here, the second and the fourth isomorphisms follow from (\ref{contact-transf}) and (\ref{muhom-phiLambdaK}) respectively.
Note that $\phi_{\varepsilon}\circ \pi'^{-1}\colon \partial \overline{N}_K\to \Lambda_K$ is a homeomorphism and the map $ i'\circ (\phi_{\varepsilon}\circ \pi'^{-1})^{-1}$ agrees with
\[ e\colon \Lambda_K\to  \bR^3\setminus K ,\  (x;\xi)\mapsto x+ \varepsilon \cdot \xi\]
(see (\ref{emb-ei})).
Hence, we have
$( \phi_{\varepsilon}\circ \pi'^{-1})_* (i'^* F ) = e^* F$.
This proves the second assertion.
\end{proof}

\subsection{A functor associated to an exact Lagrangian cobordism}\label{subsec-cobordism-functor}

We return to the geometric setup in Section \ref{subsec-clean}.
We will apply the sheaf theory introduced above and the results by \cite{Li} about exact Lagrangian cobordisms. 
The goal of this subsection is to give a commutative diagram (\ref{diagram-Phi-J}).

Let $K_0$ and $K_1$ be knots in $\bR^3$. For $\varphi \in \Ham_c(T^*\bR^3)$, suppose that $\varphi (T^*_{K_0}\bR^3)$ and $0_{\bR^3}$ intersect cleanly along $K_1\subset 0_{\bR^3} = \bR^3$.
We use the following notation for the inclusion maps:
\[\xymatrix{
\bR^3 \setminus K_0 \ar[r]^-{j_0} & \bR^3 & K_0 \ar[l]_-{i_0} , &
\bR^3 \setminus K_1 \ar[r]^-{j_1} & \bR^3 & K_1 \ar[l]_-{i_1}. 
}\]

The sequence of stable categories 
\begin{align}\label{split-Verdier}
\xymatrix{\Loc(K_0) \ar[r]^-{{i_0}_*} & \Sh_{L_{K_0}}(\bR^3 ) \ar[r]^-{j_0^*} & \Loc (\bR^3\setminus K_0)}
\end{align}
is a split Verdier sequence with adjunctions $i_0^*\dashv {i_0}_*\dashv {i_0}^!$ and ${j_0}_! \dashv j_0^* \dashv {j_0}_*$, that is, a cofiber sequence in the category of stable categories whose morphisms admit both left and right adjoints. 
See \cite[Appendix A]{CDHHLMNN2} for the general theory of split Verdier sequences. 
In particular, for any $F\in \Sh_{L_{K_0}}(\bR^3)$, the counit morphism of ${j_0}_! \dashv j_0^*$ and the unit morphism of  $i_0^*\dashv {i_0}_*$ give a (co)fiber sequence  
\begin{align}\label{semi-decomp}
\xymatrix{  {j_0}_!j_0^* F \ar[r] & F \ar[r] & {i_0}_*i_0^* F }
\end{align}
in $\Sh_{L_{K_0}}(\bR^3)$, such that $j_0^* F \in \Loc(\bR^3\setminus K_0)$ and $i_0^*F  \in \Loc(K_0)$.
If $F_x\simeq 0$ at some $x\in \bR^3\setminus K_0$, then $j_0^*F\simeq 0$ in $\Loc (\bR^3\setminus K_0)$ and it follows that $F\simeq {i_0}_*i_0^*F$ from (\ref{semi-decomp}).

For the exact Lagrangian submanifold $L_{\varphi}  = \varphi(T^*_{K_0}\bR^3) \setminus 0_{\bR^3}$, we have
\[L_{\varphi} \setminus D^*_r \bR^3 = \wh{\Lambda}_{K_0} \setminus D^*_r \bR^3\]
for sufficiently large $r>0$ and, as discussed in Section \ref{subsec-clean}, we may assume that
\[L_{\varphi} \cap D^*_{\varepsilon}\bR^3 = \wh{\Lambda}_{K_1} \cap D^*_{\varepsilon}\bR^3\]
for sufficiently small $\varepsilon>0$.
Let $\wh{\Lambda}'_{K_0}\coloneqq \wh{\Lambda}_K \setminus D^*_r \bR^3$ and $\wh{\Lambda}'_{K_1}\coloneqq \wh{\Lambda}_{K_1} \cap D^*_{\varepsilon}\bR^3$.
For $i=0,1$, the restriction to $\wh{\Lambda}'_{K_i} \subset \wh{\Lambda}_{K_i}$ gives an equivalence
\[ \mush_{\wh{\Lambda}_{K_i}}(\wh{\Lambda}_{K_i}) \overset{\sim}{\to} \mush_{\wh{\Lambda}'_{K_i}}(\wh{\Lambda}'_{K_i}).  \]

For a Weinstein manifold $X$ with contact type boundary $\partial^{\infty}X$,
Li~\cite{Li} constructed a functor associated to an exact Lagrangian cobordism in $\bR\times \partial^{\infty}X$ and observed its properties.
In the present situation, $T^*\bR^3$ is a Weinstein manifold with the Lagrangian skeleton $0_{\bR^3}$.
Via the map $\psi$ of (\ref{psi-symplectization}), $\psi^{-1}(L_{\varphi}) = \psi^{-1}( \varphi(T^*_{K_0}\bR^3 ) \setminus 0_{\bR^3} )$ is an exact Lagrangian cobordism in $\bR\times U^*\bR^3$.
Then, we have a functor described as follows.

\begin{theorem}[{\cite[Theorem~1.1]{Li}}]
There is a fully faithful functor associated to the exact Lagrangian cobordism $L_{\varphi}$
\[\Phi_{L_{\varphi}} \colon 
\mush_{0_{\bR^3}\cup \varphi(T^*_{K_0}\bR^3)}(0_{\bR^3}\cup \varphi(T^*_{K_0}\bR^3)) \hookrightarrow  \mush_{L_{K_0}}(L_{K_0}) 
. \]
\end{theorem}

Strictly speaking, Theorem~1.1 of \cite{Li} is stated for the categories of proper objects. 
The statement used here is proved earlier in the course of the proof of that theorem. 

\begin{remark}\label{rem-connected-intersection}
Let us denote $L=0_{\bR^3}\cup \varphi(T^*_{K_0}\bR^3)$.
For the definition of the category $\mush_{L} (L)$, recall that we need to take a lift of $L$ to a conical Lagrangian $\wh{L}$ as in (\ref{lift-of-Lagrangian}).
We remark that if the clean intersection $\varphi(T^*_{K_0}\bR^3) \cap \bR^3$ is not connected,
a function $f\colon L\to \bR$ which is used to define $\wh{L}$ may not exist.
Under our assumption that $K_1=\varphi(T^*_{K_0}\bR^3) \cap \bR^3$ is a knot, we can take $f\colon L\to \bR$ such that $\rest{f}{0_{\bR^3}}$ is constant to $0$ and $d( \rest{f}{\varphi(T^*_{K_0}\bR^3)}) = \rest{\lambda_X}{ \varphi(T^*_{K_0}\bR^3)}$, and we obtain the lift $\wh{L}$ by (\ref{lift-of-Lagrangian}).
\end{remark}

Since $L_{K_0}\setminus D^*_r\bR^3 = \left( 0_{\bR^3}\cup \varphi(T^*_{K_0}\bR^3) \right) \setminus D^*_r\bR^3 = \wh{\Lambda}'_{K_0}$,  we have the restriction functors
\begin{align*}
\operatorname{res} & \colon \mush_{L_{K_0}}(L_{K_0}) \to \mush_{\wh{\Lambda}'_{K_0}} (\wh{\Lambda}'_{K_0})\simeq \mush_{\wh{\Lambda}_{K_0}}(\wh{\Lambda}_{K_0}), \\
\operatorname{res}' & \colon \mush_{0_{\bR^3}\cup \varphi(T^*_{K_0}\bR^3)} (0_{\bR^3}\cup \varphi(T^*_{K_0}\bR^3)) \to \mush_{\wh{\Lambda}'_{K_0}} (\wh{\Lambda}'_{K_0})\simeq \mush_{\wh{\Lambda}_{K_0}}(\wh{\Lambda}_{K_0}) .
\end{align*}

\begin{lemma}\label{lem-Phi-res}
$\operatorname{res} \circ \Phi_{L_{\varphi}} \simeq \operatorname{res}' $.
\end{lemma}
\begin{proof}
Let us consider $L_{\varphi}$ as the composition of the following two Lagrangian cobordisms:
\begin{itemize} 
\item the Lagrangian cobordism $L_{\varphi}$ from $\Lambda_{K_1}$ to $\Lambda_{K_0}$,
\item the trivial cobordism $\wh{\Lambda}_{K_0}$ from $\Lambda_{K_0}$ to $\Lambda_{K_0}$.
\end{itemize}
Then, by \cite[Theorem 3.6]{Li}, we have
\[ \Phi_{L_{\varphi}} \simeq \Phi_{\wh{\Lambda}_{K_0}} \circ \left( \Phi_{L_{\varphi}}\times \id_{\mush_{\wh{\Lambda}_{K_0}}( \wh{\Lambda}_{K_0})}\right) .\]
From this result, we obtain the following commutative diagram (here we abbreviate $0_{\bR^3}\cup\varphi(T^*_{K_0}\bR^3)$ by $L$):
\[\xymatrix{
\mush_{L}(L) \ar[dd]_-{\Phi_{L_{\varphi}}} \ar[r]_-{\sim} \ar@/^20pt/[rr]^{\operatorname{res}'} &   \mush_{L}(L) \times_{\mush_{ \wh{\Lambda}_{K_0}}( \wh{\Lambda}_{K_0})} \mush_{ \wh{\Lambda}_{K_0}}(\wh{\Lambda}_{K_0}) \ar[d]_-{\Phi_{L_{\varphi}}\times \id} \ar[r]_-{\mathrm{pr}'} &  \mush_{ \wh{\Lambda}_{K_0}}(\wh{\Lambda}_{K_0}) \ar@{=}[d] \\
 & \mush_{L_{K_0}} (L_{K_0}) \times_{\mush_{ \wh{\Lambda}_{K_0}}( \wh{\Lambda}_{K_0})} \mush_{\wh{\Lambda}_{K_0}}(\wh{\Lambda}_{K_0}) \ar[r]_-{\mathrm{pr}} & \mush_{ \wh{\Lambda}_{K_0}}(\wh{\Lambda}_{K_0})  \\
\mush_{L_{K_0}} (L_{K_0}) & \mush_{L_{K_0}} (L_{K_0}) . \ar[u]_-{\vsim} \ar[l]_-{\Phi_{\wh{\Lambda}_{K_0}}} \ar@/_15pt/[ur]_-{\operatorname{res}} & 
}\]
Here, $\mathrm{pr}$ and $\mathrm{pr}'$ are the projections to the second factor, and the fiber products are defined by
\[\xymatrix@R=0pt{  \mush_{L} (L) \ar[r]^-{\mathrm{res}'} & \mush_{ \wh{\Lambda}_{K_0}}( \wh{\Lambda}_{K_0}) & \mush_{\wh{\Lambda}_{K_0}}(\wh{\Lambda}_{K_0}) \ar[l]_-{\id} , \\ \mush_{L_{K_0}} (L_{K_0}) \ar[r]^-{\mathrm{res}} & \mush_{ \wh{\Lambda}_{K_0}}( \wh{\Lambda}_{K_0}) & \mush_{\wh{\Lambda}_{K_0}}(\wh{\Lambda}_{K_0}) \ar[l]_-{\id} . } \]
Moreover, since $\wh{\Lambda}_{K_0}$ is the cobordism coming from the constant Legendrian isotopy of $\Lambda_{K_0}$,
we have by \cite[Theorem 1.7]{Li}
\[\Phi_{\wh{\Lambda}_{K_0}} \simeq \id_{\mush_{L_{K_0}} (L_{K_0}) }. \]
Then, the commutativity of the diagram implies that $\operatorname{res} \circ \Phi_{L_{\varphi}} \simeq \operatorname{res}' $.
\end{proof}

The next lemma can be verified
from the construction of the cobordism functor in \cite[Section 3.1]{Li}.
\begin{lemma}\label{lem-res-0U}
Let $U\subset \bR^3$ be an open subset such that $\varphi(T^*_{K_0}\bR^3) \cap \overline{T^*U} = \emptyset$.
Then, the following diagram commutes:
\begin{align}\label{diagram-0U}
\begin{split}
\xymatrix{  \mush_{0_{\bR^3}\cup \varphi(T^*_{K_0}\bR^3)}(0_{\bR^3}\cup \varphi(T^*_{K_0}\bR^3)) \ar[r]^-{\Phi_{L_{\varphi}}} \ar[d] & \mush_{L_{K_0}}(L_{K_0})\ar[d] \\
\mush_{0_U}(0_U) \ar[r]^-{\id} & \mush_{0_U}(0_U) . }
\end{split}
\end{align}
Here, the vertical functors are defined by the restrictions to $0_U \subset 0_{\bR^3}\cup \varphi(T^*_{K_0}\bR^3)$ and $0_U\subset L_{K_0}=0_{\bR^3} \cup T^*_{K_0}\bR^3$.

\end{lemma}
\begin{proof}
Referring to \cite[Section 3.1]{Li},
we can see that, in our setting, the functor $\Phi_{L_{\varphi}}$ is derived from a contact Hamiltonian flow on $T^{*,\infty}(\bR^3\times \bR)$ that is a lift of the negative Liouville flow of $T^*\bR^3$.
The open subset $T^*U$ is preserved by the negative Liouville flow, and its lift to $T^{*,\infty}(U\times \bR)$ induces the cobordism functor $\Phi_{L}\colon \mush_{0_U}(0_U) \to \mush_{0_U}(0_U)$ taking $L \subset T^*U\setminus 0_U$ as the empty set.
From the construction of the cobordism functor in \cite[Section 3.1]{Li},
the vertical restriction functors of (\ref{diagram-0U}) intertwines $\Phi_{L_{\varphi}}$ and $\Phi_{\emptyset}\simeq \id_{\mush_{0_U}(0_U)} \colon \mush_{0_U}(0_U) \to \mush_{0_U}(0_U)$, and the diagram commutes.
\end{proof}

Let us define a simple object $\scrF_{L_{\varphi}} \in \mush_{L_{\varphi}}(L_{\varphi})$ by
\[  \scrF_{L_{\varphi}} \coloneqq \rest{ \left( K^{\varphi}\circ (\mathfrak{m}_{T^*_{K_0}\bR^3}(\bfk_{K_0}[-1]) ) \right)}{L_{\varphi}}. \]
It is the image of $\mathfrak{m}_{T^*_{K_0}\bR^3}(\bfk_{K_0}[-1]) \in \mush_{T^*_{K_0}\bR^3}(T^*_{K_0}\bR^3) $ under the composite functor 
\[\xymatrix@C=35pt{
\mush_{T^*_{K_0}\bR^3}(T^*_{K_0}\bR^3) \ar[r]^-{K^{\varphi}\circ (\cdot) }  &\mush_{\varphi(T^*_{K_0}\bR^3)}(\varphi(T^*_{K_0}\bR^3)) \ar[r]^-{\rest{(\cdot)}{L_{\varphi}}} & \mush_{L_\varphi}(L_\varphi), 
}\]
where $K^{\varphi}\coloneqq \rest{K^H}{\{t=1\}}$ for the kernel $K^H$ associated to a Hamiltonian isotopy $(\varphi_H^t)_{t\in [0,1]}$ with compact support such that $\varphi_H^1=\varphi$~\cite{GKS}.
We may assume that $D^*_r\bR^3$ contains the support of $(\varphi^t_H)_{t\in [0,1]}$.
Consider the simple object $\mathscr{F}_{K_i}=\mathfrak{m}_{\wh{\Lambda}_{K_i}}(\bfk_{K_i}[-1])\in \mush_{\wh{\Lambda}_{K_i}}(\wh{\Lambda}_{K_i})$ as in Section \ref{subsec-conormal} for $i=0,1$.
Then, the restrictions of $\scrF_{L_{\varphi}}$ on $\wh{\Lambda}'_{K_0} = \wh{\Lambda}_{K_0}\setminus D^*_r\bR^3$ and $\wh{\Lambda}'_{K_1}$ satisfy
\begin{align}\label{property-of-FL}
\begin{array}{cc}
 \rest{\scrF_{L_\varphi}}{\wh{\Lambda}'_{K_0}} \simeq \rest{\scrF_{K_0}}{\wh{\Lambda}'_{K_0}}, & \rest{\scrF_{L_\varphi}}{\wh{\Lambda}'_{K_1}} \simeq \rest{\scrF_{K_1}}{\wh{\Lambda}'_{K_1}} \otimes \ell',
\end{array}\end{align}
for some rank $1$ local system $\ell'\in \Loc(\wh{\Lambda}'_{K_1})$.
Via the homotopy equivalence $\wh{\Lambda}'_{K_1} \simeq \Lambda_{K_1}$, we obtain from $\ell'$ a rank $1$ local system $\ell$ on $\Lambda_{K_1}$.
\begin{remark}\label{rem-ell}
Both $\rest{\scrF_{L_\varphi}}{\wh{\Lambda}'_{K_1}}$ and $\scrF_{K_1}$ are given by the restrictions of simple objects of $\mush_{T^*_{K_1}\bR^3} (T^*_{K_1}\bR^3)$.
Hence, $\ell = (\rest{\pi_{\bR^3}}{\Lambda_{K_1}})^* (\bar{\ell})$ for some local system $\bar{\ell}$ on $K_1$. 
\end{remark}

Using $\ell$ and the embedding $\sigma_1\colon \Lambda_{K_1} \hookrightarrow L_{\varphi}$, we have a functor
\[ \ell\otimes \sigma_1^*(\cdot)\colon  \Loc(L_{\varphi}) \to \Loc(\Lambda_{K_1}) ,\  F\mapsto \ell \otimes \sigma_1^*F. \]

\begin{lemma}
Let us consider the equivalences of categories
\begin{align}\label{three-equiv} \begin{cases} \frakm_{L_{K_1}} & \colon \Sh_{L_{K_1}} (\bR^3) \to \mush_{L_{K_1}}(L_{K_1}) , \\
\overline{\muhom}(\scrF_{K_1}, \cdot) & \colon \mush_{\wh{\Lambda}'_{K_1}}(\wh{\Lambda}'_{K_1})\simeq \mush_{\wh{\Lambda}_{K_1}}(\wh{\Lambda}_{K_1}) \to \Loc(\wh{\Lambda}_{K_1}) \simeq \Loc(\Lambda_{K_1}) ,\\
\overline{\muhom}(\scrF_{L_{\varphi}}, \cdot) & \colon \mush_{L_{\varphi}}(L_{\varphi}) \to \Loc(L_{\varphi}).
\end{cases}
\end{align}
Then, the following diagram commutes:
\[\xymatrix@C=40pt@R=30pt{
\mush_{L_{K_1}}(L_{K_1}) \ar[r] & \mush_{\wh{\Lambda}'_{K_1}}(\wh{\Lambda}'_{K_1})  \ar[d]^-{\overline{\muhom}(\scrF_{K_1}, \cdot)}_-{\vsim} & 
\mush_{L_{\varphi}}(L_{\varphi}) \ar[l] \ar[d]^-{\overline{\muhom}(\scrF_{L_{\varphi}}, \cdot)}_-{\vsim}
\\
\Sh_{L_{K_1}}(\bR^3) \ar[u]^-{\frakm_{L_{K_1}}}_-{\vsim} \ar[r]^-{\rest{\mu_{K_1}(\cdot)}{\Lambda_{K_1}}[1]} & \Loc(\Lambda_{K_1}) & \Loc(L_{\varphi}) \ar[l]_{ \ell \otimes \sigma_1^*(\cdot)} ,
}\]
where the upper horizontal arrows are the restriction functors.
\end{lemma}
\begin{proof}
The left square commutes by (\ref{muhom-mu_K}).
To see that the right square commutes, note that
\[\begin{array}{cc} \Loc(\Lambda_{K_1})\to \mush_{\Lambda_{K_1}}(\Lambda_{K_1}) ,\  F\mapsto \scrF_{K_1}\otimes F , & \Loc(L_{\varphi}) \to \mush_{L_{\varphi}}(L_{\varphi}) ,\  G\mapsto \scrF_{L_{\varphi}}\otimes G \end{array}\]
are the inverse of the vertical functors. Then, for any $G\in \Loc(L_{\varphi})$,
\[   \rest{(\mathscr{F}_{L_{\varphi}} \otimes G)}{\wh{\Lambda}'_{K_1}} = (\mathscr{F}_{K_1}\otimes \ell') \otimes \rest{G}{\wh{\Lambda}'_{K_1}}, \]
and it agrees with $\scrF_{K_1}\otimes (\ell \otimes \sigma_1^*G )$ via the natural equivalence $\mush_{\wh{\Lambda}'_{K_1}}(\wh{\Lambda}'_{K_1})\simeq \mush_{\Lambda_{K_1}}(\Lambda_{K_1})$.
This shows the commutativity of the right square.
\end{proof}

From this lemma, the three functors of (\ref{three-equiv}) induce an equivalence of categories
\begin{align}\label{equiv-Sh*L-mush}
\begin{split}
\mush_{0_{\bR^3}\cup \varphi (T^*_{K_0}\bR^3)}(0_{\bR^3}\cup \varphi (T^*_{K_0}\bR^3))
& \simeq \mush_{L_{K_1}}(L_{K_1}) \times_{\mush_{\wh{\Lambda}'_{K_1}}(\wh{\Lambda}'_{K_1})} \mush_{L_{\varphi}}(L_{\varphi}) \\
&\simto   \Sh_{L_{K_1}}(\bR^3) \times_{\Loc(\Lambda_{K_1})} \Loc (L_{\varphi}),
\end{split}
\end{align}
where the fiber product in the second line is defined from
\begin{align}\label{maps-for-fiber-prod} \xymatrix@C=50pt{ \Sh_{L_{K_1}} (\bR^3) \ar[r]^-{ \rest{\mu_{K_1}(\cdot)}{\Lambda_{K_1}}[1]} & \Loc(\Lambda_{K_1}) & \Loc(L_{\varphi}) \ar[l]_-{\ell \otimes \sigma_1^*(\cdot)}. }
\end{align}
Through (\ref{equiv-Sh*L-mush}) and $\frakm_{L_{K_0}}\colon \Sh_{L_{K_0}}(\bR^3) \simto \mush_{L_{K_0}}(L_{K_0})$,
the functor $\Phi_{L_{\varphi}}$ gives a fully faithful functor
\[ \boldsymbol{\Phi} \colon \Sh_{L_{K_1}}(\bR^3) \times_{\Loc(\Lambda_{K_1})} \Loc (L_{\varphi}) \hookrightarrow \Sh_{L_{K_0}}(\bR^3)  \]
which makes the following diagram commute:
\begin{align}\label{diagram-Phi-bPhi}
\begin{split}\xymatrix{
\Sh_{L_{K_1}}(\bR^3) \times_{\Loc(\Lambda_{K_1})} \Loc (L_{\varphi}) \ar@{^{(}->}[r]^-{\bPhi} & \Sh_{L_{K_0}}(\bR^3) \ar[d]_-{\frakm_{L_{K_0}}}^-{\vsim} \\
\mush_{0_{\bR^3}\cup \varphi (T^*_{K_0}\bR^3)}(0_{\bR^3}\cup \varphi (T^*_{K_0}\bR^3))  \ar@{^{(}->}[r]^-{\Phi_{L_{\varphi}}} \ar[u]^-{(\ref{equiv-Sh*L-mush})}_{\vsim} &  \mush_{L_{K_0}}(L_{K_0}).
}
\end{split}
\end{align}

A proof of the statement in the following remark is given in \cite[Appendix C.1.20]{HeyerMann}. 
It can also be seen from the explicit description of homotopy pullbacks of dg categories in \cite{BB13}, together with the correspondence between dg categories and $\Mod_{\bfk}$-enriched stable categories. 

\begin{remark}\label{rem-fiber-prod}
    We have the following description of the pullbacks of $\Mod_\bfk$-enriched stable categories. 
    For $\Mod_\bfk$-enriched stable categories $\cC_1, \cC_2, \cD$ and exact functors $F_1\colon \cC_1\to \cD, F_2\colon \cC_2\to \cD$, an object of $\cC_1\times_{\cD} \cC_2$ consists of a triple $(c_1,c_2,\alpha)$ where $c_i$ is an object of $\cC_i$ for each $i\in \{1,2\}$ and $\alpha\colon F_1(c_1)\simeq F_2(c_2)$ is an equivalence in $\cD$. 
    We often abbreviate $\alpha$. 
    A morphism from $(c_1,c_2,\alpha)$ to  $(c'_1,c'_2,\alpha')$ consists of a triple $(f_1,f_2,\beta)$ where $f_i\colon c_i\to c'_i$ is a morphism in $\cC_i$ for each $i\in \{1,2\}$ and 
    $\beta$ is a homotopy between $\alpha'\circ F_1(f_1)$ and $F_2(f_2)\circ \alpha$ in $\cD$.
    Equivalently, there is an equivalence
    \[\Hom_{\cC_1\times_{\cD} \cC_2}((c_1,c_2,\alpha), (c_1',c_2',\alpha'))\simeq \Hom_{\cC_1}(c_1,c_1')\times_{\Hom_{\cD}(F_1(c_1), F_2(c_2'))} \Hom_{\cC_2} (c_2, c_2') \]
    where the pullback in the right hand side is computed in the $\infty$-category $\Mod_\bfk$.
    Note that if either $c_2$ or $c_2'$ is a $0$-object, the projection \[
    \Hom_{\cC_1\times_{\cD} \cC_2}((c_1,c_2,\alpha), (c_1',c_2',\alpha'))\to\Hom_{\cC_1}(c_1,c'_1)\] is an equivalence. 
\end{remark}

\begin{lemma}\label{lem-bPhi-sigma-mu}
The following diagram commutes:
\begin{align}\label{diagram-Sh*L-Lambda_0}
\begin{split}
\xymatrix{
& \Loc(\Lambda_{K_0})  & \\
\Sh_{L_{K_1}}(\bR^3) \times_{\Loc(\Lambda_{K_1})} \Loc (L_{\varphi}) \ar[ur]^-{\sigma_0^*\circ \operatorname{pr}} \ar@{^{(}->}[rr]^-{\bPhi}  & &  \Sh_{L_{K_0}}(\bR^3) . \ar[ul]_-{ \rest{\mu_{K_0}(\cdot)}{\Lambda_{K_0}}[1] } 
} \end{split}
\end{align}
Here $\operatorname{pr} \colon \Sh_{L_{K_1}}(\bR^3) \times_{\Loc(\Lambda_{K_1})} \Loc(L_{\varphi}) \to \Loc(L_{\varphi}) ,\  (F,G) \mapsto G$ is the projection.
\end{lemma}
\begin{proof}
Consider the following diagram
\[\xymatrix@C=40pt{
\Sh_{L_{K_1}}(\bR^3) \times_{\Loc(\Lambda_{K_1})} \Loc(L_{\varphi}) \ar[d]^-{(\ref{equiv-Sh*L-mush})}_-{\vsim} \ar[r]^-{\sigma_0^*\circ \operatorname{pr}} &  \Loc(\Lambda_{K_0}) \ar[d]^-{\scrF_{K_0}\otimes (\cdot)}_{\vsim} & \Sh_{L_{K_0}}(\bR^3) \ar[d]^-{\frakm_{L_{K_0}}}_-{\vsim}
\ar[l]_-{\rest{\mu_{K_0}(\cdot)}{\Lambda_{K_0}}[1]} \\
\mush_{0_{\bR^3}\cup \varphi (T^*_{K_0}\bR^3)}(0_{\bR^3}\cup \varphi (T^*_{K_0}\bR^3)) \ar[r]_-{\operatorname{res}'}  &  \mush_{\wh{\Lambda}_{K_0}}(\wh{\Lambda}_{K_0}) & \mush_{L_{K_0}}(L_{K_0}). \ar[l]^-{\operatorname{res}}
}\]
If this diagram commutes, then the commutativity of (\ref{diagram-Sh*L-Lambda_0}) follows immediately from Lemma \ref{lem-Phi-res} and the diagram (\ref{diagram-Phi-bPhi}).

From (\ref{property-of-FL}), $\rest{\scrF_{L_{\varphi}}}{\wh{\Lambda}'_{K_0}} \simeq \rest{\scrF_{K_0}}{\wh{\Lambda}'_{K_0}}$.
It follows that for any $(F,G) \in \Sh_{L_{K_1}}(\bR^3) \times_{\Loc(\Lambda_{K_1})} \Loc(L_{\varphi})$,
\[ (\operatorname{res}'\circ (\ref{equiv-Sh*L-mush})) (F,G) \simeq \rest{( \mathscr{F}_{L_{\varphi}} \otimes G) }{\wh{\Lambda}'_{K_0}} \simeq \scrF_{K_0}\otimes \sigma_0^*G \]
via the natural equivalences $\mush_{\wh{\Lambda}_{K_0}}(\wh{\Lambda}_{K_0}) \simeq \mush_{\wh{\Lambda}'_{K_0}}(\wh{\Lambda}'_{K_0})\simeq \mush_{\Lambda_{K_0}}(\Lambda_{K_0})$.
This shows that the left square commutes.
The right triangle commutes by (\ref{muhom-mu_K}).
\end{proof}

Let us fix $x_0 \in \bR^3$ outside a sufficiently large ball containing $ N_{K_0} \cup N_{K_1} \cup \pi_{\bR^3}(\varphi(T^*_{K_0}\bR^3))$.
Then, $x_0 \in \bR^3 \setminus K_0$.
For $i\in \{0,1\}$, we have a functor
\[   \Sh_{L_{K_i}}(\bR^3) \to \Sh (\{x_0\})  \]
which takes the stalk at $x_0$.
Then, we have a diagram
\begin{align}\label{diagram-Sh*L-stalk}
\begin{split}
\xymatrix{
 \Sh_{L_{K_1}}(\bR^3) \times_{\Loc(\Lambda_{K_1})} \Loc (L_{\varphi})   \ar@{^{(}->}[rr]^-{\bPhi } \ar[dr] & &  \Sh_{L_{K_0}}(\bR^3)  \ar[dl]  \\
& \Sh (\{x_0\}). &  
}\end{split}\end{align}
Lemma \ref{lem-res-0U} implies that this diagram commutes.

For any $\bfk$-vector space $V$, let $V_{\bR^3} \in \Sh_{L_{K_1}}(\bR^3)$ denote a sheaf on $\bR^3$ constant to $V$.
Recall the functors (\ref{maps-for-fiber-prod}) defining the fiber product $ \Sh_{L_{K_1}}(\bR^3)\times_{\Loc(\Lambda_{K_1})} \Loc(L_{\varphi})$. Since
\[ \rest{\mu_{K_1}(V_{\bR^3})}{\Lambda_{K_1}} [1] \simeq 0 \text{ in }\Loc(\Lambda_{K_1}), \]
we have an object $(V_{\bR^3},0)\in \Sh_{L_{K_1}}(\bR^3)\times_{\Loc(\Lambda_{K_1})} \Loc(L_{\varphi})$.
\begin{lemma}\label{lem-PhiL-constant}
$\bPhi(( V_{\bR^3}, 0))\simeq V_{\bR^3}$ 
\end{lemma}
\begin{proof}
Let us abbreviate $F= \bPhi( (V_{\bR^3},0)) \in \Sh_{L_{K_0}}(\bR^3)$.
By Lemma \ref{lem-bPhi-sigma-mu},
\[  \rest{ \mu_{K_0}( F )}{\Lambda_{K_0}} [1]  \simeq \sigma_0^* (0) =0 \text{ in }\Loc(\Lambda_{K_0}).  
\]
Since the functor $\overline{\muhom}(\scrF_{K_0},\cdot)$ is an equivalence, it follows from (\ref{muhom-mu_K}) that
\[\rest{ \frakm_{L_{K_0}} ( F )}{\wh{\Lambda}_{K_0}} \simeq 0 \text{ in }\mush_{\wh{\Lambda}_{K_0}}(\wh{\Lambda}_{K_0})\]
for the microlocalization functor $\frakm_{L_{K_0}}$.
Hence $\CMS(F) \subset 0_{\bR^3}$ and $F\simeq V'_{\bR^3}$ for some chain complex $V'$ of $\bfk$-vector spaces by \cite[Proposition 5.4.5]{KS90}.
By the commutativity of (\ref{diagram-Sh*L-stalk}), $V'$ is isomorphic to $V$. 
\end{proof}

Next, consider the local system $\bar{\ell}$ on $K_1$ in Remark \ref{rem-ell} and a constant sheaf $\bfk_{L_{\varphi}} \in \Loc(L_{\varphi})$. By the functors (\ref{maps-for-fiber-prod}),
\[ \xymatrix@C=55pt{ {i_1}_*\bar{\ell} \ar@{|->}[r]^-{\rest{\mu_{K_1}(\cdot)}{\Lambda_{K_1}} [1]} &  \rest{\mu_{K_1}({i_1}_*\bar{\ell})}{\Lambda_{K_1}} [1] \simeq (\rest{\pi_{\bR^3}}{\Lambda_{K_1}})^*(\bar{\ell}) [1] = \ell [1] & \bfk_{L_{\varphi}}[1] \ar@{|->}[l]_-{\ell \otimes \sigma_1^*(\cdot)} . }\]
Here, we use Proposition \ref{prop-muK-conormal} for $\bar{\ell}\in \Loc(K_1)$.
Thus, we have an object $({i_1}_*\bar{\ell}, \bfk_{L_{\varphi}}[1]) \in \Sh_{L_{K_1}}(\bR^3)\times_{\Loc(\Lambda_{K_1})} \Loc(L_{\varphi})$.

\begin{lemma}\label{lem-F_L-F_K} 
$\ell$ is a trivial local system. 
That is, $\rest{\scrF_{L_\varphi}}{\wh{\Lambda}'_{K_1}} \simeq \scrF_{K_1}$. 
Moreover,
\[\bPhi((\bfk_{K_1}, \bfk_{L_\varphi}[1]))\simeq \bfk_{K_0}\] in $\Sh_{L_{K_0}}(\bR^3)$.
\end{lemma}
\begin{proof}
Let us abbreviate $F= \bPhi ({i_1}_*\bar{\ell}, \bfk_{L_{\varphi}}[1])\in \Sh_{L_{K_0}}(\bR^3)$.
By the commutativity of (\ref{diagram-Sh*L-stalk}), the stalk of $F$ at $x_0\in \bR^3\setminus K_0$ vanishes.
This means that $j_0^* F \simeq 0$ in $\Loc(\bR^3\setminus K_0)$ and $F \simeq {i_0}_*i_0^*F$ by the (co)fiber sequence (\ref{semi-decomp}), where $i_0^*F\in \Loc(K_0)$.

By Lemma \ref{lem-bPhi-sigma-mu},
\[ \rest{ \mu_{K_0}(F)}{\Lambda_{K_0}}[1] \simeq \sigma_0^*\bfk_{L_{\varphi}}[1] =\bfk_{\Lambda_{K_0}}[1] \text{ in }\Loc(\Lambda_{K_0}).  \]
On the other hand, by Proposition \ref{prop-muK-conormal},
\[   \rest{ \mu_{K_0}({i_0}_*i_0^*F)}{\Lambda_{K_0}} \simeq (\rest{\pi_{\bR^3}}{\Lambda_{K_0}})^* (i_0^*F) .\]
It follows that $\bfk_{\Lambda_{K_0}} \simeq  (\rest{\pi_{\bR^3}}{\Lambda_{K_0}})^* (i_0^*F)$ in $\Loc(\Lambda_{K_0})$.
Taking a section $s\colon K_0\to \Lambda_{K_0}$ of $\rest{\pi_{\bR^3}}{\Lambda_{K_0}}$, we obtain
\[\bfk_{K_0} = s^*\bfk_{\Lambda_{K_0}} \simeq (\rest{\pi_{\bR^3}}{\Lambda_{K_0}}\circ s)^* (i_0^*F) = i_0^*F  . \]
Therefore, $i_0^*F$ is a sheaf on $K_0$ constant to $\bfk$, and $F\simeq {i_0}_*i_0^*F \simeq \bfk_{K_0}$.

By Lemma \ref{lem-PhiL-constant}, we have $\bPhi((\bfk_{\bR^3}, 0 ))\simeq \bfk_{\bR^3}$.
From the fully faithfulness of $\bPhi$,
\begin{align*}\Hom ((\bfk_{\bR^3}, 0 ), ({i_1}_*\bar{\ell}, \bfk_{L_\varphi}[1] )) & \simeq \Hom (\bPhi((\bfk_{\bR^3}, 0 )), \bPhi(({i_1}_*\bar{\ell}, \bfk_{L_\varphi}[1] ))) \\
& \simeq \Hom (\bfk_{\bR^3}, \bfk_{K_0}) \\
& \simeq H^*(K_0; \bfk ).
\end{align*}
On the other hand, by the definition of the fiber product $ \Sh_{L_{K_1}}(\bR^3)\times_{\Loc(\Lambda_{K_1})} \Loc(L_{\varphi})$ (see also Remark \ref{rem-fiber-prod}),
\[\Hom ((\bfk_{\bR^3}, 0 ), ({i_1}_*\bar{\ell}, \bfk_{L_\varphi} ))\simeq \Hom (\bfk_{\bR^3}, {i_1}_*\bar{\ell}).\]
The isomorphism $\Hom (\bfk_{\bR^3}, {i_1}_*\bar{\ell}) \simeq H^*(K_0; \bfk )$ implies that $\bar{\ell}$ is a trivial rank $1$ local system on $K_1$.
Hence, $\ell$ is also a trivial local system on $\Lambda_{K_1}$.
\end{proof}

From the construction of the $3$-manifold $M_{\varphi}$,
we naturally have an equivalence of categories
\[\Loc(M_{\varphi}) \cong \Loc(\bR^3\setminus K_1) \times_{\Loc(\Lambda_{K_1})} \Loc(L_{\varphi}),\]
where the fiber product is defined from the embeddings
$ \bR^3\setminus K_1 \overset{e_1}{\leftarrow} \Lambda_{K_1} \overset{\sigma_1}{\rightarrow}L_{\varphi}$.
The following diagram commutes by Proposition \ref{prop-muK-conormal}:
\[ 
\xymatrix@C=50pt{ 
\Loc(\bR^3\setminus K_1) \ar[r]^-{e_1^*} \ar@{^{(}->}[d]_-{{j_1}_!} & \Loc(\Lambda_{K_1}) \ar@{=}[d] & \Loc(L_{\varphi}) \ar[l]_-{\sigma_1^*} \ar@{=}[d]
\\ \Sh_{L_{K_1}} (\bR^3) \ar[r]^-{ \rest{\mu_{K_1}(\cdot)}{\Lambda_{K_1}}[1]} & \Loc(\Lambda_{K_1}) & \Loc(L_{\varphi}) \ar[l]_-{\sigma_1^*},
} \]
where the lower horizontal functors agrees with (\ref{maps-for-fiber-prod}) since $\ell$ is trivial by Lemma \ref{lem-F_L-F_K}.
Thus, the fully faithful functor ${j_1}_!$ and the identity functor on $\Loc(L_{\varphi})$ define a functor
\[J \colon  \Loc (M_{\varphi})  \to \Sh_{L_{K_1}}(\bR^3) \times_{\Loc(\Lambda_{K_1})} \Loc (L_{\varphi}).\]
By the explicit description of the pullbacks of stable categories mentioned in Remark \ref{rem-fiber-prod}, $J$ is fully faithful.

We combine the diagrams (\ref{diagram-Sh*L-Lambda_0}) and (\ref{diagram-Sh*L-stalk}), and the functors $J$ and $j_0^* \colon \Sh_{L_{K_0}}(\bR^3) \to \Loc(\bR^3\setminus K_0)$.
Then, we obtain the following commutative diagram:
\begin{align}\label{diagram-Phi-J}
\begin{split}
\xymatrix{
&  \Loc (\Lambda_{K_0}) &  & \\
  \Loc (M_{\varphi}) \ar@/^10pt/[ur]^-{\sigma_0^*} \ar@{^{(}->}[r]^-{J} \ar[dr] & \Sh_{L_{K_1}}(\bR^3) \times_{\Loc(\Lambda_{K_1})} \Loc (L_{\varphi})  \ar@{^{(}->}[rr]^-{\bPhi }  & &  \Sh_{L_{K_0}}(\bR^3) \ar@/_10pt/[ull]_-{\rest{\mu_{K_0}(\cdot)}{\Lambda_{K_0}}[1]} \ar[dll] \ar[d]^-{j_0^*} \\
&  \Sh (\{x_0\}) & & \Loc (\bR^3 \setminus K_0) , \ar[ll]
}
\end{split}
\end{align}
where all the three arrows heading to $\Sh(\{x_0\})$ are the functors which take the stalk at $x_0$.

Let us also consider objects in $\Loc(\varphi(T^*_{K_0}\bR^3))$. We have a natural equivalence
\[  \Loc (\varphi (T^*_{K_0}\bR^3))\simeq \Loc(K_1) \times_{\Loc(\Lambda_{K_1})}\Loc (L_{\varphi}), \]
where the fiber product is defined from $\rest{\pi_{\bR^3}}{\Lambda_{K_1}}\colon \Lambda_{K_1}\to K_1$ and $\sigma_1\colon \Lambda_{K_1} \to L_{\varphi}$.
Similar to $J$, we  have a fully faithful functor 
\[ I\colon\Loc (\varphi (T^*_{K_0}\bR^3))\simeq \Loc(K_1) \times_{\Loc(\Lambda_{K_1})}\Loc (L_{\varphi})\to  \Sh_{L_{K_1}}(\bR^3) \times_{\Loc(\Lambda_{K_1})} \Loc (L_{\varphi}) \]
induced by ${i_1}_*\colon  \Loc(K_1) \hookrightarrow \Sh_{L_{K_1}}(\bR^3)$ and the identity functor on $\Loc(L_{\varphi})$. 

\begin{lemma}\label{lem-Phi-I}
The essential image of the fully faithful functor $\bPhi \circ I\colon\Loc (\varphi (T^*_{K_0}\bR^3))\to \Sh_{L_{K_0}}(\bR^3)$ coincides with the essential image of ${i_0}_*\colon  \Loc(K_0) \to \Sh_{L_{K_0}}(\bR^3)$.
\end{lemma}
\begin{proof}
Take $F\in \Loc(K_0)$ arbitrarily.
Let $\pi'\colon \varphi(T^*_{K_0}\bR^3) \to K_0$ be the composition of $\rest{\varphi^{-1}}{\varphi(T^*_{K_0}\bR^3)}\colon \varphi(T^*_{K_0}\bR^3) \to T^*_{K_0}\bR^3$ and
$\rest{(\pi_{\bR^3})}{T^*_{K_0}\bR^3} \colon T^*_{K_0}\bR^3 \to K_0$.
Since $\varphi$ fixes every point in $\sigma_0(\Lambda_{K_0})$, $\pi'\circ \sigma_0 = \rest{\pi_{\bR^3}}{\Lambda_{K_0}}$.
Then, we obtain an object $\pi'^*F\in \Loc(\varphi ( T^*_{K_0}\bR^3) )$ such that $\sigma_0^* (\pi'^*F) = (\rest{\pi_{\bR^3}}{\Lambda_{K_0}})^*F $ in $\Loc(\Lambda_{K_0})$.

Consider $\bPhi ( I(\pi'^*F[1]))\in \Sh_{L_{K_0}}(\bR^3)$.
From the commutativity of (\ref{diagram-Sh*L-stalk}), its stalk at $x_0$ vanishes.
Thus, it follows from the (co)fiber sequence (\ref{semi-decomp}) that $\bPhi ( I(\pi'^*F[1])) \simeq {i_0}_*G$ for some $G\in \Loc(K_0)$.

We claim that $G\simeq F$.
Indeed, by Lemma \ref{lem-bPhi-sigma-mu} and Proposition \ref{prop-muK-conormal},
\[  (\sigma_0^*\circ \operatorname{pr}) (I(\pi'^*F[1])) \simeq \mu_{K_0}( {i_0}_*G )[1] \simeq (\rest{\pi_{\bR^3}}{\Lambda_{K_0}})^*G [1]. \]
In addition, from the definition of $I$,
\[(\sigma_0^*\circ \operatorname{pr}) (I(\pi'^*F[1])) = \sigma_0^* (\pi'^*F[1]) = (\rest{\pi_{\bR^3}}{\Lambda_{K_0}})^*F[1] .\]
Therefore, $(\rest{\pi_{\bR^3}}{\Lambda_{K_0}})^*F \simeq (\rest{\pi_{\bR^3}}{\Lambda_{K_0}})^*G$ in $\Loc(\Lambda_{K_0})$.
Choose a section $s\colon K_0\to \Lambda_{K_0}$ of $\rest{\pi_{\bR^3}}{\Lambda_{K_0}}$, then
\[ G\simeq s^*(\rest{\pi_{\bR^3}}{\Lambda_{K_0}})^*G\simeq s^* (\rest{\pi_{\bR^3}}{\Lambda_{K_0}})^*F \simeq F. \]
Thus, we obtain $\bPhi(I(\pi'^*F[1]))\simeq {i_0}_*F$.
\end{proof}

\subsection{\texorpdfstring{A homomorphism $h$ between fundamental groups}{A homomorphism h between fundamental groups}}\label{subsec-hom-h}

From now on, for the reason explained in Remark~\ref{rem-repcat} below,
we work with the homotopy categories of $\infty$-categories. 

\begin{notation}
    Let $\Vect_{\bfk}$ denote the $1$-category of $\bfk$-vector spaces, possibly infinite dimensional.
    
    For an $\infty$-category $\cC$, $\Homotopy (\cC)$ denotes its homotopy category. 
    If $\cC$ is a $\Mod_\bfk$-enriched stable $\infty$-category, then $\Homotopy (C)$ naturally admits a structure of a triangulated category enriched in $\Vect_{\bfk}$. 
    The (co)fiber sequences in $\cC$ correspond to the distinguished triangles in $\Homotopy (\cC)$. 
    For any objects $x,y$ of a $\Mod_\bfk$-enriched stable $\infty$-category $\cC$, we have
    \[ \Hom_{\Homotopy (\cC)}(x,y)\simeq H^0(\Hom_{\cC} (x,y)) \quad \text{in} \: \Vect_{\bfk}.\]
    In addition, an exact functor $\alpha \colon \cC \to \cD$ between stable $\infty$-categories induces an exact functor $\Homotopy (\alpha) \colon \Homotopy(\cC) \to \Homotopy (\cD)$ between the triangulated categories. 
    When no confusion is likely to arise, we will simply write $\alpha$ for $\Homotopy (\alpha)$. 

    Note also that, if $\cC\to \cD \to \cE$ is a split Verdier sequence, then the homotopy category $\Homotopy (\cD)$ admits a semiorthogonal decomposition $\Homotopy (\cD)=\langle \Homotopy (\cE),\Homotopy (\cC)\rangle $ in the classical sense where $\Homotopy (\cE)$ is regarded as a subcategory of $\Homotopy (\cD)$ via the left adjoint of the projection functor $\cD\to \cE$.
\end{notation}

For a manifold $X$, let $\loc (X)$ denote the $1$-category of locally constant sheaves of $\bfk$-vector spaces on $X$.
It is equivalent to the homotopy category of the full subcategory of $\Loc(X)$ consisting of $F \in \Loc(X)$ such that $H^i(F)=0$ for any $i\neq 0$.

We note that $\loc(M_{\varphi})$ is equivalent to the homotopy category of the full subcategory of $\Loc(M_{\varphi})$
consisting of $F\in \Loc(M_{\varphi})$ such that its stalk $F_{x_0}\in \Sh(\{x_0\})$ at $x_0$ is a complex of $\bfk$-vector spaces concentrated in the $0$-th degree, and similar for $\loc(\bR^3\setminus K_0)$.
Then, from the diagram (\ref{diagram-Phi-J}), we have a commutative diagram
\begin{align}\label{diagram-Phi-stalk} 
\begin{split}
\xymatrix{ \loc(M_{\varphi}) \ar[rr]^-{ j_0^* \circ \bPhi \circ J } \ar[rd] & & \loc(\bR^3 \setminus K_0) \ar[ld] \\ & \Vect_{\bfk} , }
\end{split}
\end{align}
where the downward arrows are functors defined by taking the stalk at $x_0$.

For any group $G$, let $\Rep_{\bfk} (G)$ denote the abelian category of representations of $G$ on $\bfk$-vector spaces, possibly infinite dimensional.
We denote the forgetful functor by
\[F_G\colon \Rep_{\bfk}(G) \to \Vect_{\bfk}. \]
In addition, let $\Rep^{\fin}_{\bfk}(G)$ denote the full subcategory of $\Rep_{\bfk}(G)$ consisting of $\rho \in \Rep_{\bfk}(G)$ such that $F_G(\rho)$ has finite dimension.
For any $V\in \Vect_{\bfk}$, $\GL(V)$ denote the group of linear automorphisms of $V$.

In this section, let us abbreviate
\[\begin{array}{cc} G_0 \coloneqq \pi_1(\bR^3 \setminus K_0 ,x_0) , & G_1 \coloneqq \pi_1(M_{\varphi},x_0). \end{array}\]
Considering the monodromy of locally constant sheaves, we  have natural equivalences of abelian categories
\[\begin{array}{cc}
\mon^0_{x_0}  \colon \loc(\bR^3 \setminus K_0) \to \Rep_{\bfk}(G_0) , & 
\mon^1_{x_0}  \colon \loc (M_{\varphi}) \to \Rep_{\bfk}(G_1).
\end{array}\]
More precisely,
for any locally constant sheaf $\ell \in \loc(\bR^3\setminus K_0)$ and any $g\in G_0$, $\mon^0_{x_0}(\ell) (g) \in \GL (\ell_{x_0})$ is the monodromy for $\ell$ along a loop based at $x_0$ representing $g$, and $\mon^1_{x_0}$ is defined in a similar way. 
On the other hand, by taking the universal cover $\widetilde{M}_0$ of $\bR^3\setminus K_0$, the inverse functor of $\mon^0_{x_0}$ is defined by
\[ (\mon^0_{x_0})^{-1} \colon \Rep_{\bfk} (G_0) \to \loc(\bR^3\setminus K_0) ,\  \rho \mapsto \ell^0_{\rho} \coloneqq \widetilde{M}_0 \times_{\rho} V, \]
where $V= F_{G_0}(\rho)$ and $\widetilde{M}_0 \times_{\rho} V$ is the locally constant sheaf associated to $\rho$.
Similarly, we define
\[ (\mon^1_{x_0})^{-1} \colon \Rep_{\bfk} (G_1) \to  \loc(M_{\varphi}) ,\ \rho \mapsto \ell^1_{\rho} \coloneqq \widetilde{M}_{\varphi}\times_{\rho}V,\]
where $\widetilde{M}_{\varphi}$ is the universal cover of $M_{\varphi}$ and $V=F_{G_1}(\rho)$.

Let us abbreviate $ j_0^* \circ \bPhi \circ J $ by $\Psi$. Then, from (\ref{diagram-Phi-stalk}), we have a commutative diagram
\[\xymatrix{ \Rep_{\bfk}(G_1) \ar[rr]^-{ \mon^0_{x_0} \circ \Psi \circ (\mon^1_{x_0})^{-1}} \ar[rd]_{F_{G_1}} & & \Rep_{\bfk}(G_0) \ar[ld]^-{F_{G_0}} \\ & \Vect_{\bfk} . }
\]

\begin{remark}\label{rem-repcat}
    The representation category $\Rep_{\bfk}(G)$ also admits an $\infty$-categorical definition for any group $G$.  
    From the viewpoint of the $\infty$-category theory, it is natural to define a representation category $\Rep_{\bfk}^{\infty}(G)$ to be the functor category $\Fun (\mathrm{B}G, \Mod _\bfk)$. 
    We may consider the $\Mod_\bfk$-enriched full subcategory $\Rep_{\bfk}^{\infty}(G)^{\heartsuit}$ consisting of those functors whose values are complexes concentrated in degree $0$. 
    The homotopy category of $\Rep_{\bfk}^{\infty}(G)^{\heartsuit}$ is equivalent to the $1$-category $\Rep_{\bfk}(G)$ that we use. 
    However, the functors $\mon_{x_0}^0, \mon_{x_1}^1$ do not lift to enriched $\infty$-functors from $\Loc (\bR^3\setminus K_i)$ to $\Rep_{\bfk}^{\infty}(G_i)$ for $i\in \{0,1\}$. 
    On the other hand, each of the functors $(\mon_{x_0}^0)^{-1}, (\mon_{x_1}^1)^{-1}$ can be defined as enriched $\infty$-functor via a natural morphism $\bR^3\setminus K_i\to \mathrm{B}G_i$ and a natural identification $\Loc (\bR^3\setminus K_i)\simeq \Fun (\bR^3\setminus K_i, \Mod _\bfk)$ for each $i\in \{0,1\}$.  
    These $\infty$-functors induce equivalences of the homotopy categories, but they are not equivalences of the $\Mod_{\bfk}$-enriched $\infty$-categories. 
\end{remark}

\begin{proposition}\label{prop-h-mon}
There exist a group homomorphism
\[ h \colon G_0 \to G_1\]
and a $1$-dimensional representation $\kappa \colon G_0 \to \bfk^*$ of $G_0$
such that $\mon^0_{x_0} \circ \Psi \circ (\mon^1_{x_0})^{-1}$ agrees with 
\[ h^*(\cdot)\otimes \kappa\colon \Rep_{\bfk}(G_1) \to  \Rep_{\bfk} (G_0) ,\  \rho \mapsto (h^*\rho)\otimes \kappa. \]
\end{proposition}
\begin{proof}
We use the following fact for a general group $G$, called Tannaka reconstruction:
Let $\End(F_G)$ denote the set of natural transformations from $F_{G}$ to itself.
Its element is written as $(\theta_{\rho})_{\rho \in \Rep_{\bfk} (G)}$, where $\theta_{\rho} \colon F_G(\rho) \to F_G(\rho)$ is a $\bfk$-linear map.
Then, it has a natural structure of $\bfk[G]$-algebra with unit $(\id_{F_G(\rho)})_{\rho \in \Rep_{\bfk}(G)}$, and the $\bfk$-algebra map
\[ \bfk[G] \to \End(F_G) ,\  g \mapsto (\rho(g))_{\rho \in \Rep_{\bfk}(G)}\]
is an isomorphism.

Let us abbreviate 
$\Psi' = \mon^0_{x_0} \circ \Psi \circ (\mon^1_{x_0})^{-1}$.
Since $F_{G_0}\circ \Psi' \simeq F_{G_1}$, $\Psi'$ induces a unital ring homomorphism
\[ \hat{h} \colon \End(F_{G_0}) \to \End (F_{G_1}) ,\  (\theta_{\rho_0})_{\rho_0 \in \Rep_{\bfk} (G_0)} \mapsto (\theta_{\Psi'(\rho_1)})_{\rho_1 \in \Rep_{\bfk} (G_1) }. \]
Using the above isomorphism $\End(F_{G_i}) \cong \bfk [G_i]$ for $i=0,1$, we obtain a ring homomorphism
$ \hat{h} \colon \bfk[G_0] \to \bfk[G_1]$.
Since $\Rep_{\bfk}(G_i)$ can be viewed to the category of left $\bfk[G_i]$-modules,
$\hat{h}$ induces a functor
\[ \hat{h}^*\colon \Rep_{\bfk}(G_1) \to \Rep_{\bfk} (G_0) \]
and this agrees with $\Psi'$.
To see this, note that for $g\in G_0$, $\hat{h}(g)\in \bfk[G_1]$ corresponds to $(\Psi'(\rho_1)(g))_{\rho_1\in \Rep_{\bfk}(G_1)}\in \End(F_{G_1})$. 

Let $\bfk[G_i]^{\times}$ denote the group of units of $\bfk[G_i]$ for $i=0,1$.
It contains $G_i$ as a subgroup and $\hat{h}$ is restricted to a group homomorphism
\[ \hat{h} \colon G_0 \to \bfk[G_1]^{\times}. \]
Since $G_1$ is left-orderable by Lemma \ref{lem-LO}, $\bfk[G_1]$ has only trivial units, that is,
\[\bfk[G_1]^{\times} = \{k \cdot g' \mid k\in \bfk^*,\ g' \in G_1 \}.\]
For the proof, see \cite[Proposition 6]{RZ}, for example.
For every $g\in G_0$, we can uniquely determine $\kappa(g) \in \bfk^*$ and $h(g)\in G_1$ such that $\hat{h} (g) = \kappa(g) \cdot h(g)$.
Then, both $h\colon G_0\to G_1$ and $\kappa \colon G_0 \to \bfk^*$ are group homomorphisms.
For any $\rho \in \Rep_{\bfk}(G_1)$, $g\in G_0$ and $v\in F_{G_1}(\rho)$,
\[ ((\hat{h}^*\rho)(g)) (v)  = \kappa (g) \cdot ((h^*\rho)(g))(v). \]
Therefore, $\Psi'(\rho)=\hat{h}^*\rho = (h^*\rho) \otimes \kappa$.
\end{proof}

\begin{lemma}
The representation $\kappa \colon G_0 \to \bfk^*$ of Proposition \ref{prop-h-mon} is trivial.
\end{lemma}
\begin{proof}
By the functors (\ref{maps-for-fiber-prod}), 
\[ \xymatrix@C=55pt{ {j_1}_!\bfk_{\bR^3\setminus K_1} \ar@{|->}[r]^-{\rest{\mu_{K_1}(\cdot)}{\Lambda_{K_1}} [1]} &  \rest{\mu_{K_1}({j_1}_!\bfk_{\bR^3\setminus K_1})}{\Lambda_{K_1}} [1] \simeq e_1^*\bfk_{\bR^3\setminus K_1}    & \bfk_{L_{\varphi}} \ar@{|->}[l]_-{\sigma_1^*} . }\]
Here, we use Proposition \ref{prop-muK-conormal} for $\bfk_{\bR^3\setminus K_1}\in \Loc(\bR^3\setminus K_1)$.
Thus, we have an object $({j_1}_!\bfk_{\bR^3\setminus K_1}, \bfk_{L_\varphi}) \in  \Sh_{L_{K_1}}(\bR^3) \times_{\Loc(\Lambda_{K_1})} \Loc (L_{\varphi})$.
Moreover, involving the objects considered in Lemma \ref{lem-PhiL-constant} and Lemma \ref{lem-F_L-F_K}, we have a distinguished triangle
\[\xymatrix{
({j_1}_!\bfk_{\bR^3\setminus K_1}, \bfk_{L_\varphi}) \ar[r] &  ( \bfk_{\bR^3}, 0) \ar[r] & ({i_1}_*\bfk_{K_1}, \bfk_{L_\varphi}[1]) \ar[r]^-{+1} & 
}\]
in $\Homotopy (\Sh_{L_{K_1}}(\bR^3) \times_{\Loc(\Lambda_{K_1})} \Loc (L_{\varphi}))$.
Let us abbreviate $F= \bPhi(({j_1}_!\bfk_{\bR^3\setminus K_1}, \bfk_{L_\varphi})) \in \Homotopy (\Sh_{L_{K_0}}(\bR^3))$. Then, Lemma \ref{lem-PhiL-constant} and Lemma \ref{lem-F_L-F_K} show that there is a distinguished triangle in $\Homotopy (\Sh_{L_{K_0}}(\bR^3))$
\[ \xymatrix{ F \ar[r] & \bfk_{\bR^3} \ar[r] & \bfk_{K_0} \ar[r]^-{+1}& .  }\]
In particular, $j_0^*F\simeq \bfk_{\bR^3\setminus K_0}$ in $\Homotopy (\Loc(\bR^3\setminus K_0))$.

Let $\rho_{\mathrm{triv}} \colon G_1\to \bfk^*$ denote the $1$-dimensional trivial representation.
Since $h^*(\rho_{\mathrm{triv}})\in \Rep_{\bfk}(G_0)$ is trivial, Proposition \ref{prop-h-mon} shows that
\[  \Psi(\ell^1_{\rho_{\mathrm{triv}}}) \simeq \ell^0_{\kappa} . \]
In addition, by the definition of the functor $J$, $J(\ell^1_{\rho_{\mathrm{triv}}}) \simeq ({j_1}_!\bfk_{\bR^3\setminus K_1}, \bfk_{L_\varphi}) $.
Therefore,
\[j_0^* F = (j_0^* \circ \bPhi \circ J)(\ell^1_{\rho_{\mathrm{triv}}}) \simeq \ell^0_{\kappa},\]
and thus $\ell^0_{\kappa}$ is a trivial rank $1$ local system on $\bR^3\setminus K_0$.
Hence $\kappa$ is a trivial representation of $G_0$.
\end{proof}

By this lemma, $\mon^0_{x_0} \circ \Psi \circ (\mon^1_{x_0})^{-1}$ agrees with the pullback functor \[h^* \colon \Rep_{\bfk}(G_1) \to \Rep_{\bfk}(G_0)\]
for the group homomorphism $h$ of Proposition \ref{prop-h-mon}.
It means that for any $\rho \in \Rep_{\bfk}(G_1)$,
\begin{align}\label{isom-Psi-ell1}  \Psi(\ell^1_{\rho}) \simeq \ell^0_{h^*\rho} . \end{align}

\begin{lemma}\label{lem-dist-tirangle}
Fix any $\rho \in \Rep_{\bfk} (G_1)$.
Then, a sheaf $F \coloneqq (\bPhi \circ J) (\ell^1_{\rho})\in \Homotopy (\Sh_{L_{K_0}}(\bR^3))$ satisfies
\[j_0^* F \simeq \ell^0_{ h^*\rho} \]
in $\loc(\bR^3\setminus K_0)$. 
Moreover, there is a distinguished triangle in $\Homotopy (\Loc(\Lambda_{K_0}))$ 
\begin{align}\label{dist-tri-ell_rho}
\xymatrix{  e_0^*\ell^0_{h^*\rho} \ar[r] & \sigma_0^*\ell^1_{\rho} \ar[r] & \pi_0^*i_0^* F [1] \ar[r]^-{+1} &  ,  
}\end{align}
where $\pi_0 \coloneqq \rest{\pi_{\bR^3}}{\Lambda_{K_0}} \colon \Lambda_{K_0} \to K_0$.
\end{lemma}
\begin{proof}
For the sheaf $F\in \Homotopy (\Sh_{L_{K_0}}(\bR^3))$ defined as in the assertion,
\[j_0^*F = j_0^*\circ \bPhi \circ J (\ell^1_{\rho}) = \Psi( \ell^1_{\rho})\]
and it follows from the diagram (\ref{diagram-Phi-stalk}) that the stalk $(j_0^*F)_{x_0}$ at $x_0$ is isomorphic to $(\ell^1_{\rho})_{x_0} = F_{G_1}(\rho)$ in $\Vect_{\bfk}$.
The monodromy for $j_0^*F\in \loc(\bR^3\setminus K_0)$ along loops based at $x_0$ is given by
\[ \mon^0_{x_0}(j_0^*F) = (\mon^0_{x_0}\circ \Psi \circ (\mon^1_{x_0})^{-1}) (\rho) = h^*\rho  \in \Rep_{\bfk}(G_0). \]
Here, recall that $G_0 = \pi_1(\bR^3\setminus K_0,x_0)$.
Therefore, $j_0^*F \simeq \ell^0_{h^*\rho}$ in $\loc(\bR^3\setminus K_0)$.

Next, we consider the distinguished triangle in $\Homotopy (\Sh_{L_{K_0}}(\bR^3))$ induced by (\ref{semi-decomp})
\[ \xymatrix{
{j_0}_!j_0^* F \ar[r] & F \ar[r] & {i_0}_*i_0^* F \ar[r]^-{+1} & .
}\]
Note that ${j_0}_!j_0^* F\simeq {j_0}_! \ell^0_{h^*\rho}$.
To this distinguished triangle, let us apply the functor $\rest{\mu_K(\cdot)}{\Lambda_{K_0}}[1] \colon \Homotopy (\Sh_{L_{K_0}}(\bR^3)) \to \Homotopy (\Loc(\Lambda_{K_0}))$.
Proposition \ref{prop-muK-conormal} shows that
\[\begin{array}{cc}
\rest{\mu_K( {j_0}_! \ell^0_{h^*\rho})}{\Lambda_{K_0}}[1] = e_0^*\ell^0_{h^*\rho} , &
\rest{\mu_K( {i_0}_*i_0^*F )}{\Lambda_{K_0}}[1] = \pi_0^*i_0^*F[1] .
\end{array}\]
In addition, from the diagram (\ref{diagram-Phi-J}), we have
\[\rest{\mu_{K_0}(F)}{\Lambda_{K_0}} [1]=  \rest{(\mu_{K_0}\circ\bPhi \circ J)(\ell^1_{\rho})}{\Lambda_{K_0}}[1] \simeq \sigma_0^* \ell^1_{\rho}.\]
Therefore, we obtain a distinguished triangle (\ref{dist-tri-ell_rho}) in $\Homotopy (\Loc (\Lambda_{K_0}))$.
\end{proof}

For the base point $y_0$ of $\Lambda_{K_0}$ fixed in Section \ref{subsec-clean},
we choose paths
\begin{align}\label{paths-connecting-basept}
\begin{array}{cc} \gamma_0 \colon [0,1] \to \bR^3\setminus K_0, & \gamma_1 \colon [0,1] \to M_{\varphi}, \end{array}
\end{align}
such that $\gamma_0(0)=\gamma_1(0)=x_0$, $\gamma_0(1)=e_0(y_0)$ and $\gamma_1(1)=\sigma_0(y_0)$.
They define group homomorphisms
\begin{align*}
(e_0)_*' \colon & \pi_1(\Lambda_{K_0} , y_0) \to G_0 = \pi_1(\bR^3 \setminus K_0,x_0) ,\  [\gamma] \mapsto [\gamma_0 \cdot (e_0\circ \gamma) \cdot \gamma_0^{-1}], \\
(\sigma_0)_*' \colon & \pi_1(\Lambda_{K_0} , y_0) \to G_1 = \pi_1(M_{\varphi},x_0) ,\  [\gamma] \mapsto [\gamma_1 \cdot (\sigma_0\circ \gamma) \cdot \gamma_1^{-1}] .
\end{align*}
Let us set $m'_0 \coloneqq (e_0)_*'(m_0)\in G_0$ and $l'_0\coloneqq (\sigma_0)_*'(l_0)\in G_1$.

\begin{lemma}\label{lem-rho-i}
Let $\rho \in \Rep^{\fin}_{\bfk} (G_1)$.
If $(h^*\rho)(m'_0)\in \GL (F_{G_1}(\rho))$ does not have $1$ as an eigenvalue,
then
\[i_0^* (\bPhi \circ J(\ell^1_{\rho})) \simeq 0 \]
in $\Homotopy (\Loc(K_0))$.
\end{lemma}

\begin{proof}
For any $\rho \in \Rep_{\bfk}^{\fin}(G_1)$ satisfying the condition in the assertion, we set
\[ F \coloneqq (\bPhi \circ J)(\ell^1_{\rho}) \in \Homotopy (\Sh_{L_{K_0}}(\bR^3)).\]
By Lemma \ref{lem-dist-tirangle},
we have a distinguished triangle (\ref{dist-tri-ell_rho}) in $\Homotopy (\Loc(\Lambda_{K_0}))$
\[
\xymatrix{  e_0^*\ell^0_{h^*\rho} \ar[r] & \sigma_0^*\ell^1_{\rho} \ar[r] & \pi_0^*i_0^* F [1] \ar[r]^-{+1} &  . 
}\]
If we take a section $s\colon K_0\to \Lambda_{K_0}$, then
\[s^* (\pi_0^*i_0^*F ) = i_0^*F = i_0^* (\bPhi \circ J(\ell^1_{\rho})).\]
Therefore, it suffices to show that the morphism
$e_0^*\ell^0_{h^*\rho} \to \sigma_0^* \ell^1_{\rho}$ in the above distinguished triangle  is a quasi-isomorphism, which implies that $\pi_0^*i_0^*F\simeq 0$.

As complexes of sheaves, both $\ell^0_{h^*\rho}$ and $\ell^1_{\rho}$ are concentrated in degree $0$.
By taking the cohomolgy, we obtain an exact sequence in $\loc(\Lambda_{K_0})$
\[  \xymatrix{
\ell'\coloneqq H^{0}(\pi_0^*i_0^*F) \ar[r]^-{f} & e_0^*\ell^0_{h^*\rho} \ar[r]^-{g} & \sigma_0^* \ell^1_{\rho} \ar[r] & H^1(\pi_0^*i_0^*F).
} \]
To prove that $g$ is an isomorphism,
let us observe the monodromy along $m_0\in \pi_1(\Lambda_{K_0},y_0)$
for the locally constant sheaves in this exact sequence.
\begin{itemize}
\item For $\ell'$, the monodromy along $m_0$ is given by the identity map on $\ell'_{y_0}$.
Indeed, $\ell' = \pi_0^* (H^0(i_0^*F))$ (since $\pi_0^*$ is an exact functor) for $i_0^*F\in \Homotopy (\Sh(K_0))$,
and $(\pi_0)_*(m_0)\in \pi_1(K_0,\pi_0(y_0))$ is represented by a constant loop in $K_0$.
\item For $e_0^* \ell^0_{h^*\rho}$, let $U_{m_0} \in \GL ( (\ell^0_{h^*\rho})_{e_0(y_0)} )$ be the linear automorphism given by the monodromy along $m_0$.
Via the parallel transport for $\ell^0_{h^*\rho}$ along $\gamma_0^{-1}$ from $e_0(y_0)$ to $x_0$,
$U_{m_0}$ coincides with the linear automorphism $ (h^*\rho)( m'_0 ) \in \GL( F_{G_1}(\rho) )$.
Here, note that $(\ell^0_{h^*\rho})_{x_0} \cong (\ell^1_{\rho})_{x_0} = F_{G_1}(\rho)$
\end{itemize}
The morphism $f\colon \ell' \to e_0^* \ell^0_{h^*\rho}$ induces a linear map $f_{y_0} \colon \ell'_{y_0} \to (\ell^0_{h^*\rho})_{e_0(y_0)}$.
From the above observation, $\Image f_{y_0}$ is contained in a linear subspace
\[  \{ v\in  (\ell^0_{h^*\rho})_{e_0(y_0)} \mid U_{m_0}(v) =v \}. \]
Via the parallel transport for $\ell^0_{h^*\rho}$ along $\gamma_0^{-1}$, this is isomorphic to
\[\{v \in F_{G_1}(\rho) \mid ((h^*\rho)(m'_0)) (v) = v \},\]
which is the zero vector space from the assumption on $\rho$. Therefore, $\Image f_{y_0}=0$.

The morphism $g\colon e_0^*\ell^0_{h^*\rho}\to \sigma_0^*\ell^1_{\rho}$ induces a linear map between the stalks
\[g_{y_0} \colon (\ell^0_{h^*\rho})_{e_0(y_0)} \to (\ell^1_{\rho})_{\sigma_0(y_0)}.\]
Since $\Ker g_{y_0} = \Image f_{y_0}=0$,  $g_{y_0}$ is injective.
Moreover, since both of the stalks have the same finite dimension as $F_{G_1}(\rho)$, $g_{y_0}$ must be an isomorphism.
As $e_0^*\ell^0_{h^*\rho}$ and $\sigma_0^*\ell^1_{\rho}$ are locally constant sheaves, we can conclude that $g$ is an isomorphism.
\end{proof}

Using the semiorthogonal decomposition of $\Homotopy (\Sh_{L_{K_0}}(\bR^3))$ induced by the split Verdier sequence (\ref{split-Verdier}), we obtain the following result from Lemma \ref{lem-rho-i}.

\begin{proposition}\label{prop-rho-J-Lambda}
Let $\rho \in \Rep^{\fin}_{\bfk} (G_1)$. If $(h^*\rho)(m'_0) \in \GL (F_{G_1}(\rho))$ does not have $1$ as an eigenvalue,
then
\[(\bPhi \circ J)(\ell^1_{\rho}) \simeq {j_0}_!\ell^0_{h^*\rho} \]
in $\Homotopy (\Sh_{L_{K_0}}(\bR^3))$.
Moreover, the two representations of $\pi_1(\Lambda_{K_0},y_0)$ 
\[ (h^*\rho) \circ (e_0)_*' ,\ \rho \circ (\sigma_0)_*' \colon \pi_1(\Lambda_{K_0},y_0) \to \GL(F_{G_1}(\rho)) \]
are equivalent in $\Rep_{\bfk}^{\fin}(\pi_1(\Lambda_{K_0},y_0))$.
\end{proposition}

\begin{proof}
From Lemma \ref{lem-rho-i}, we have ${i_0}_*i_0^* (\bPhi\circ J)(\ell^1_{\rho}) \simeq 0$.
From the distinguished triangle induced by (\ref{semi-decomp}), we obtain
\[ (\bPhi \circ J)(\ell^1_{\rho}) \simeq {j_0}_!j_0^*( \bPhi \circ J(\ell^1_{\rho}) ) = {j_0}_! \Psi (\ell^1_{\rho}) \simeq {j_0}_! \ell^0_{h^*\rho} \]
in $\Homotopy (\Sh_{L_{K_0}}(\bR^3))$.
Here, the last isomorphism follows from (\ref{isom-Psi-ell1}).

Via the isomorphism $(\ell^0_{h^*\rho})_{e_0(y_0)}\cong F_{G_0}(h^*\rho) = F_{G_1}(\rho)$ obtained by the parallel transport along $\gamma_0$, the monodromy for $e_0^*\ell^0_{h^*\rho}$  along any $[\gamma] \in \pi_1(\Lambda_{K_0},y_0)$ is given by
\[(h^*\rho) ([\gamma_0\cdot (e_0\circ \gamma)\cdot \gamma_0^{-1}]) = (h^*\rho) ( (e_0)'_*([\gamma]) ) . \] 
Likewise, via the isomorphism $(\ell^1_{\rho})_{\sigma_0(y_0)} \cong F_{G_1}(\rho)$ obtained by the parallel transport along $\gamma_1$,  the monodromy for $\sigma_0^*\ell^1_{\rho}$ along $[\gamma] \in \pi_1(\Lambda_{K_0},y_0)$ is given by
\[ \rho ([\gamma_1\cdot (\sigma_0\circ \gamma)\cdot \gamma_1^{-1}]) = \rho ( (\sigma_0)'_*([\gamma]) ) .  \]
As we have seen in the proof of Lemma \ref{lem-rho-i}, there is an isomorphism $ e_0^*\ell^0_{h^*\rho} \simeq \sigma_0^*\ell^1_{\rho}$ in $\loc (\Lambda_{K_0})$.
Comparing their monodromy along any $[\gamma]\in \pi_1(\Lambda_{K_0},y_0)$, we obtain
\[(h^*\rho) \circ (e_0)_*' \simeq \rho \circ (\sigma_0)_*'\]
in $\Rep^{\fin}_{\bfk} (\pi_1(\Lambda_{K_0},y_0))$.
\end{proof}

\subsection{\texorpdfstring{Surjectivity of $h$}{Surjectivity of h}}\label{subsec-surj}

Consider the Hurewicz homomorphisms
\begin{align*}
p_0& \colon G_0 = \pi_1(\bR^3\setminus K_0,x_0) \to H_1(\bR^3\setminus K_0;\bZ)\cong \bZ , \\
p_1 &\colon G_1 = \pi_1(M_{\varphi},x_0) \to H_1(M_{\varphi};\bZ)\cong \bZ, 
\end{align*}
such that
$p_0$ maps $m'_0 = (e_0)'_*(m_0)\in G_0$ to $1\in \bZ$ and
$p_1$ maps $(\sigma_0)'_*(m_0)\in G_1$ to $1\in \bZ$.

\begin{lemma}\label{lem-phm-nonzero}
$(p_1\circ h)( m'_0) \neq 0$.
\end{lemma}
\begin{proof}
Assume that $(p_1\circ h)( m'_0) = 0$. 
Let $V \coloneqq \bfk[t,t^{-1}]$ and define a representation $\rho \colon G_1 \to \GL(V) \in \Rep_{\bfk} (G_1)$ such that $\rho(g)(\xi) = t^{p_1(g)}\cdot \xi$ for every $\xi \in V=\bfk[t,t^{-1}]$.
(Note that since $\rho$ is infinite dimensional, we cannot apply Proposition \ref{prop-rho-J-Lambda}.)

Let $F\coloneqq \bPhi(J(\ell^1_{\rho} ))\in \Homotopy (\Sh_{L_{K_0}}(\bR^3))$.
By (\ref{isom-Psi-ell1}), $j_0^* F\simeq \ell^0_{h^*\rho}$.
Since $h^*\rho \colon G_0\to \GL(V)$ is a trivial representation of $G_0$ from our assumption, we have
\begin{align}\label{eqn-jF-V}
{j_0}_! j_0^* F\simeq V_{\bR^3\setminus K_0}.
\end{align}

Consider distinguished triangles in $\Homotopy (\Sh_{L_{K_0}}(\bR^3))$ from (\ref{semi-decomp})
\[
\xymatrix@R=10pt{  V_{\bR^3\setminus K_0}\ar[r] &  V_{\bR^3}\ar[r]^-{f} & V_{K_0}\ar[r]^-{+1} &  , \\ 
 {j_0}_! j_0^* F\ar[r] &  F \ar[r]^-{g} &   {i_0}_*i_0^* F\ar[r]^-{+1} & . 
}
\]
Replacing $\bfk$ with $V$ in Lemma \ref{lem-F_L-F_K}, one can show that $ \bPhi(V_{K_1},V_{L_{\varphi}}[1])\simeq V_{K_0}$. Moreover, $ \bPhi(V_{\bR^3},0)\simeq V_{\bR^3}$ by Lemma \ref{lem-PhiL-constant}.
Since $\bPhi$ is fully faithful, there exists $f'\in \Hom ((V_{\bR^3},0), (V_{K_1},V_{L_{\varphi}}[1]))$ such that $\bPhi(f')=f$ and
\begin{align}\label{Cone-f} \bPhi( \Cone (f') ) \simeq \Cone (\bPhi(f')) \simeq V_{\bR^3\setminus K_0} [1].  
\end{align}
Let us denote $\Cone(f')[-1] = (W,\ell)\in \Homotopy (\Sh_{L_{K_1}}(\bR^3)\times_{\Loc(\Lambda_{K_1})}\Loc(L_{\varphi}))$.
Then, from the distinguished triangle
\[ \xymatrix{ (W,\ell) \ar[r] & (V_{\bR^3},0) \ar[r]^-{f'} &  (V_{K_1},V_{L_{\varphi}}[1]) \ar[r]^-{+1} & }  \]
and $j_1^* V_{K_1}\simeq 0$, we have
\[ j_1^*W\simeq j_1^*V_{\bR^3} = V_{\bR^3\setminus K_0}  \text{ in }\Homotopy (\Loc(\bR^3\setminus K_1)).\]

Next, by Lemma \ref{lem-Phi-I}, for the object $i_0^*F \in \Homotopy (\Loc(K_0))$, there exists an object $(G,\ell' )\in \Homotopy (\Loc(K_1)\times_{\Loc(\Lambda_{K_1})} \Loc(L_{\varphi})) $ such that
\[\bPhi ( I((G,\ell')) ) \simeq {i_0}_*i_0^*F \text{ in }\Homotopy (\Sh_{L_{K_0}}(\bR^3)) . \]
Recall that $F= \bPhi (J(\ell^1_{\rho}))$ by definition.
Since $\bPhi$ is fully faithful, there exists $g'\in \Hom ( J(\ell^1_{\rho}), I((G,\ell')) )$ such that $\bPhi (g') =g$ and
\begin{align}\label{Cone-g}  \bPhi (\Cone (g')) \simeq \Cone (\bPhi(g'))  \simeq {j_0}_!j_0^* F [1] . 
\end{align}
Let us denote $\Cone(g')[-1] = (W',\ell'')\in \Homotopy (\Sh_{L_{K_1}}(\bR^3)\times_{\Loc(\Lambda_{K_1})}\Loc(L_{\varphi}))$.
Note that the $\Homotopy (\Sh_{L_{K_1}}(\bR^3))$-components of $J(\ell^1_{\rho})$ and $I((G,\ell'))$ are ${j_1}_!(\rest{\ell^1_{\rho}}{\bR^3\setminus K_1})$ and ${i_1}_*G$ respectively. (See the definitions of $J$ and $I$. We identify $\bR^3\setminus K_1\cong \bR^3\setminus \overline{N}_{K_1}\subset M_{\varphi}$.)
Then, from the distinguished triangle
\[ \xymatrix{ (W',\ell'') \ar[r] & J(\ell^1_{\rho}) \ar[r]^-{g'} &  I((G,\ell')) \ar[r]^-{+1} & }  \]
and $j_1^* ({i_1}_*G) \simeq 0$, we have
\[ j_1^*W' \simeq j_1^* ({j_1}_!(\rest{\ell^1_{\rho}}{\bR^3\setminus K_1})) \simeq \rest{\ell^1_{\rho}}{\bR^3\setminus K_1}  \text{ in } \Homotopy (\Loc(\bR^3\setminus K_1)).\]

Since $V_{\bR^3\setminus K_0}\simeq {j_0}_!j_0^* F$ by (\ref{eqn-jF-V}), it follows from (\ref{Cone-f}) and (\ref{Cone-g}) that
\[ \bPhi (\Cone (f')) \simeq \bPhi (\Cone (g')).\]
From the fully faithfulness of $\bPhi$, $\Cone (f') \simeq \Cone (g')$.
Using the above notations, $(W,\ell)\simeq (W',\ell'')$, and we obtain $j_1^*W\simeq j_1^*W'$ in $\Homotopy (\Loc(\bR^3\setminus K_1))$.
However, this means that $V_{\bR^3\setminus K_0}\simeq \rest{\ell^1_{\rho}}{\bR^3\setminus K_1}$ in $\Homotopy (\Loc(\bR^3\setminus K_1))$, which contradicts to the non-triviality of the representation $\rho\circ j'_*\colon \pi_1(\bR^3\setminus K_1,x_0)\to \GL(V)$, where $j'\colon \bR^3\setminus K_1 \cong \bR^3\setminus \overline{N}_{K_1}\to M_{\varphi}$ is the inclusion map.
\end{proof}

We note that the functor $h^*\colon \Rep_{\bfk}(G_1) \to \Rep_{\bfk}(G_0),\  \rho \mapsto \rho\circ h$ has a left adjoint
\[  h_* \colon  \Rep_{\bfk}(G_0) \to \Rep_{\bfk}(G_1)  \]
such that for any representation $\rho \colon G_0\to \GL(V)$ of $G_0$, $(h_*\rho)(g) \in \GL ( \bfk[G_1] \otimes_{\bfk[G_0]} V)$ is determined by
\[(h_*\rho)(g) (\xi \otimes v)= (g\cdot \xi)\otimes v\]
for any $g\in G_1$, $\xi \in \bfk[G_1]$ and $v\in V$.
Here, $\bfk[G_1]$ is regarded as a right $\bfk[G_0]$-module by the ring homomorphism $\bfk[G_0]\to \bfk[G_1]$ induced by $h$.
In particular, the dimension of $h_*\rho$ is
\[ \dim_{\bfk}(h_*\rho) = \dim_{\bfk} ( \bfk[G_1]\otimes_{\bfk[G_0]} V ) = \#(G_1/ h(G_0)  ) \cdot  \dim V . \]

\begin{proposition}\label{prop-h-surj}
The group homomorphism $h \colon \pi_1(\bR^3\setminus K_0,x_0) \to \pi_1(M_{\varphi},x_0)$ is surjective.
\end{proposition}
\begin{proof}
    Consider any representation $\rho_1 \in \Rep^{\fin}_{\bfk} (G_1)$ such that $(h^*\rho_1)(m'_0)$ does not have $1$ as an eigenvalue.
    By Lemma \ref{lem-rho-i} and the distinguished triangle induced by (\ref{semi-decomp}), the counit morphism $ {j_0}_!j_0^* (\bPhi \circ J(\ell^1_{\rho_1}) ) \to \bPhi \circ J(\ell^1_{\rho_1})$ is an isomorphism.
    Moreover, using the fully faithful functor $\bPhi\circ J$, we obtain the following isomorphisms
    for any representation $\rho_2 \in \Rep_{\bfk} (G_1)$ which is not necessarily finite dimensional:
    \begin{align*}
        \Hom (\rho_1,\rho_2) 
        & \simeq \Hom (\bPhi \circ J(\ell^1_{\rho_1}), \bPhi \circ J(\ell^1_{\rho_2})) \\
        & \simeq \Hom ( {j_0}_!j_0^* \bPhi \circ J(\ell^1_{\rho_1}), \bPhi \circ J(\ell^1_{\rho_2})) \\
        & \simeq \Hom ( j_0^* \bPhi \circ J(\ell^1_{\rho_1}), j_0^*\bPhi\circ J(\ell^1_{\rho_2})) \\
        & \simeq \Hom (h^*\rho_1, h^*\rho_2) \\
        & \simeq \Hom (h_*h^*\rho_1, \rho_2). 
    \end{align*}
    For the fourth isomorphism, recall that $ j_0^* \bPhi\circ J(\ell^1_{\rho}) = \Psi (\ell^1_{\rho}) \simeq \ell^1_{h^*\rho}$.
    This isomorphism $\Hom (\rho_1,\rho_2) \overset{\sim}{\to} \Hom (h_* h^*\rho_1, \rho_2)$ coincides with the morphism obtained by pre-composition of the counit morphism $h_* h^*\rho_1\to \rho_1$.

    Take a $1$-dimensional representation $\rho\colon G_1\to \bfk^*=\bC^*$ with $\rho(m_1)^k\neq 1$ for every $k\in\bZ\setminus \{0\}$. 
    By Lemma \ref{lem-phm-nonzero}, this $\rho$ satisfies the assumption for $\rho_1$. 
    By taking $(\rho, h_*h^*\rho )$ as $(\rho_1, \rho_2)$ above, we obtain 
    \[\Hom (\rho, h_*h^*\rho)\simeq \Hom (h_*h^*\rho,h_*h^*\rho).\] 
    Let us focus on the identity map $\id_{h_*h^*\rho} \colon h_*h^*\rho\to h_*h^*\rho$. This factors as
    \[h_*h^*\rho\to \rho \to h_*h^*\rho\]
    by the surjectivity of $\Hom (\rho, h_*h^*\rho)\to \Hom (h_*h^*\rho,h_*h^*\rho)$.
    In particular, the morphism $\rho \to h_*h^*\rho$ gives a linear surjective map between the representation spaces.
    Since $\rho$ and $h^*\rho$ are $1$-dimensional,
    \[  \# (G_1/ h(G_0)) = \# (G_1/ h (G_0) )\cdot \dim_{\bC} (h^*\rho) =\dim_{\bC} (h_*h^*\rho)  \leq \dim_{\bC}(\rho) =1 .\]
    This shows that $h$ is surjective. 
\end{proof}

In the proof of the above proposition, we used the existence of an element $\xi \in \bfk^*=\bC^*$ such that $\xi^k \neq 1$ for any $k\in \bZ\setminus \{0\}$.
Recalling $\bfk=\bC$, we can strengthen the latter assertion of Proposition \ref{prop-rho-J-Lambda}.
\begin{proposition}\label{prop-GL_n-rep}
For any $\rho \in \Rep^{\fin}_{\bC} (G_1)$,
$(h^*\rho) \circ (e_0)'_*$ is equivalent to $\rho \circ (\sigma_0)_*'$ in $\Rep^{\fin}_{\bC}(\pi_1(\Lambda_{K_0},y_0))$.
\end{proposition}
\begin{proof}
The homomorphism $p_1\circ h \colon G_0\to \bZ$ factors through $\bar{h} \colon \bZ \to \bZ$ so that the square in the following diagram commutes (i.e. $p_1\circ h = \bar{h} \circ p_0$):
\[\xymatrix{
\pi_1(\Lambda_{K_0}) \ar[r]^-{(e_0)'_*} \ar@/^18pt/[rr]^{(\sigma_0)'_*} & 
G_0 \ar[d]_-{p_0} \ar[r]^-{h} & G_1 \ar[d]_-{p_1} \ar[r]^-{\rho} & \GL_n(\bC) \\
 & \bZ \ar[r]^-{\bar{h}} & \bZ . &
}\]
Since $p_1\circ h$ is surjective, $\bar{h}$ is also surjective.
Let $ c \coloneqq \bar{h} (1) \in \{1, -1\}$.
In addition, recall that $p_0 \colon G_0  \to  \bZ$
maps $m'_0 = (e_0)'_*(m_0) \in G_0$ to $1\in \bZ$.

We want to show that for any $n\in\bZ_{\geq 1}$ and any representation $\rho \colon G_1 \to \GL_n(\bC)$,
\begin{align}\label{eq-h-rho}
(h^*\rho)\circ (e_0)'_*\simeq \rho \circ (\sigma_0)_*' \text{ in } \Rep^{\fin}_{\bC}(\pi_1(\Lambda_{K_0},y_0)). 
\end{align}
The proof is divided into three steps.

(Step.1) We prove (\ref{eq-h-rho}) for any $1$-dimensional representation $\rho \colon G_1 \to \bC^*$.
In this case, $\rho$ factors through $p_1$ and there exists a group homomorphism $\bar{\rho} \colon \bZ \to \bC^*$ such that
$\rho = \bar{\rho}\circ p_1$.
If $\bar{\rho}(1)=1$, then $\rho(g)=1$ for any $g\in G_1$, so (\ref{eq-h-rho}) is obvious.
If $\bar{\rho}(1) \neq 1$, then
\[(h^*\rho)(m'_0) = (\bar{\rho} \circ p_1\circ h ) (m'_0) = (\bar{\rho} \circ \bar{h} \circ p_0)(m'_0)= \bar{\rho} (c) \neq 1 \]
since $c\in \{\pm 1\}$.
Hence, $\rho$ satisfies the condition of Proposition \ref{prop-rho-J-Lambda}, and (\ref{eq-h-rho}) follows from Proposition \ref{prop-rho-J-Lambda}.

(Step.2) We prove (\ref{eq-h-rho}) for any $\SL_n(\bC)$-representation $\rho \colon G_1 \to \SL_n(\bC)$.
In this case, choose $\xi \in \bC^*$ satisfying the two conditions:
\begin{enumerate}
\item $\xi^{-c}$ is not an eigenvalue of the matrix $(h^*\rho) (m'_0) \in \SL_n(\bC)$.
\item $\xi^k \neq 1$ for every $k\in \bZ\setminus \{0\}$.
\end{enumerate}
Then, we define
\[ \tilde{\rho} \colon G_1 \to \GL_n(\bC) ,\  g \mapsto \xi^{p_1(g)}\cdot \rho(g). \]
This is a finite dimensional representation of $G_1$ and the matrix
\[(h^*\tilde{\rho})(m'_0) = \xi^{p_1\circ h(m'_0)} \cdot (h^*\rho)(m'_0) = \xi^{c} \cdot (h^*\rho)(m'_0) \]
does not have $1$ as an eigenvalue by the first condition on $\xi$.
Therefore, by Proposition \ref{prop-rho-J-Lambda},
\[  (h^*\tilde{\rho})\circ (e_0)'_* \simeq \tilde{\rho} \circ (\sigma_0)'_* \]
in $\Rep^{\fin}_{\bC}(\pi_1(\Lambda_{K_0},y_0))$.
These two representations are written as
\begin{align*}
 (h^*\tilde{\rho})\circ (e_0)'_* & \colon  g \mapsto \xi^{q_0 (g)} \cdot \left( (h^*\rho)\circ (e_0)'_*\right) (g) , \\
 \tilde{\rho} \circ (\sigma_0)'_*& \colon  g\mapsto \xi^{q_1(g)} \cdot \left( \rho \circ (\sigma_0)'_*\right) (g),
\end{align*}
where $q_0,q_1 \colon \pi_1(\Lambda_{K_0},y_0) \to \bZ$ are given by $q_0\coloneqq p_1\circ h \circ (e_0)'_*$ and $q_1 \coloneqq p_1\circ (\sigma_0)'_*$.
The determinant of $(h^*\tilde{\rho})\circ (e_0)'_*(g)$ and $\tilde{\rho} \circ (\sigma_0)'_*(g)$ are the same, so we obtain
\[ (\xi^{q_0(g)})^n = (\xi^{q_1(g)})^n \]
for every $g\in \pi_1(\Lambda_{K_0},y_0)$.
From the second condition on $\xi$, it follows that $q_0 = q_1$.
Hence, $(h^*\rho)\circ (e_0)'_* \simeq \rho \circ (\sigma_0)'_* $ in $\Rep^{\fin}_{\bC} (\pi_1(\Lambda_{K_0},y_0))$.

(Step.3)
We arbitrarily take a representation $\rho \colon G_1 \to \GL_n(\bC)$.
The $1$-dimensional representation $G_1 \to \bC^* ,\  g \mapsto \det \rho(g)$ factors through $p_1$
and there exists a group homomorphism $\nu \colon \bZ \to \bC^*$ such that $\nu\circ p_1 (g) = \det (\rho(g))$ for every $g\in G_1$.
We choose $\zeta \in \bC^*$ such that $\zeta^n = \nu(1)$ and consider two representations
\[\begin{array}{cc} 
\rho_1  \colon G_1 \to \bC^* ,\  g \mapsto \zeta^{p_1(g)} , &
\rho_2  \colon G_1 \to \SL_n(\bC) ,\ g\mapsto \zeta^{-p_1(g)} \cdot \rho(g).
\end{array}\]
Here, note that $\det \rho(g) = (\nu(1) )^{p_1(g)} = \zeta^{n\cdot p_1(g)}$.
For any $g\in G_1$, $\rho(g) = \rho_1(g) \cdot \rho_2(g)$ and moreover, from the above two steps,
\[\begin{array}{cc} (h^*\rho_1)\circ (e_0)_*'= \rho_1 \circ (\sigma_0)_*',&  (h^*\rho_2)\circ (e_0)_*' \simeq \rho_2 \circ (\sigma_0)_*' , \end{array}\]
in $\Rep^{\fin}_{\bC}(\pi_1(\Lambda_{K_0},y_0))$.
This shows that (\ref{eq-h-rho}) holds for $\rho$.
\end{proof}

Using Proposition \ref{prop-GL_n-rep} for a $1$-dimensional representation, we can show the following.

\begin{proposition}\label{prop-b=0}
The integer $b$ of Proposition \ref{prop-H1L} is equal to $0$.
\end{proposition}
\begin{proof}
For $i\in \{0,1\}$, note that $(e_i)_*(l_0) =0 $ in $ H_1(\bR^3\setminus K_i;\bZ)$ since this homology class comes from the longitude of $K_i$ with respect to the Seifert framing.
Then, we have
\[\begin{array}{cc} p_0( (e_0)'_*(l_0))=0 ,& p_1((\sigma_0)'_*(l_0))=b .\end{array}\]
The second equation can be checked by
\[ (\sigma_0)_*(l_0) = a\cdot (\sigma_1)_*(l_1) + b \cdot (\sigma_0)_*(m_0) \text{ in }H_1(L_{\varphi};\bZ)  \]
from Proposition \ref{prop-H1L} and $(\sigma_1)_*(l_1) =(e_1)_* (l_1)=0$ in $H_1(M_{\varphi};\bZ)\cong H_1(\bR^3\setminus K_1;\bZ)$.

Consider a specific $1$-dimensional representation
$\rho \colon G_1 \to \bC^* ,\  g \mapsto 2^{p_1(g)}$.
Using the homomorphism $\bar{h}\colon \bZ \to \bZ$ in the proof of Proposition \ref{prop-GL_n-rep},
\begin{align*} ((h^* \rho) \circ (e_0)_*') (l_0) &= 2^{(\bar{h}\circ p_0 \circ (e_0)_*' )(l_0)}= 2^{\bar{h}(0)}=1, \\
(\rho \circ (\sigma_0)_*') (l_0) &= 2^b.
\end{align*}
By Proposition \ref{prop-GL_n-rep}, $(h^*\rho)\circ (e_0)'_*= \rho\circ (\sigma_0)'_*$.
Therefore, $1=2^b$ and thus, $b=0$.
\end{proof}

\begin{remark}
$a\in \{\pm 1\}$ and $b \in \bZ$ are from Proposition \ref{prop-H1L}, and this proposition is proved by using Floer cohomology group for Lagrangian submanifolds in $T^*\bR^3$. See \cite[Appendix A]{O25}.
It would be interesting if one can show that $b=0$ by only using Floer cohomology or more elementary tools.
\end{remark}

\section{\texorpdfstring{Constraints on the knot types of $(K_0,K_1)$}{Constraints on the knot types of (K0,K1)}}\label{sec-constraint}

We continue to consider the setup in Section \ref{subsec-cobordism-functor}.
Knots $K_0$ and $K_1$ in $\bR^3$ are oriented and equipped with the Seifert framing.
For $\varphi \in \Ham_c(T^*\bR^3)$, we suppose that $\varphi (T^*_{K_0}\bR^3)$ and $0_{\bR^3}$ intersect cleanly along $K_1$.

\subsection{Epimorphism between knot groups}\label{subsec-epi-knots}

Let us define a homomorphism $r_1\colon \pi_1(L_{\varphi},\sigma_1(y_1))\to \pi_1(\Lambda_{K_1},y_1)$ by the composition
\[ \xymatrix{ \pi_1(L_{\varphi},\sigma_1(y_1)) \ar[r] & H_1(L_{\varphi};\bZ) \ar[r]^-{(\sigma_1)_*^{-1}} & H_1(\Lambda_{K_1};\bZ) \cong \pi_1(\Lambda_{K_1},y_1). }\]
Here, the first map and the last isomorphism are the Hurewicz homomorphisms, and the middle map is the inverse of $(\sigma_1)_*$, which is an isomorphism on the first homology group by Proposition \ref{prop-H1L}.
Then, $r_1$ is the retraction of the homomorphism
$(\sigma_1)_* \colon \pi_1( \Lambda_{K_1} ,y_1) \to \pi_1(L_{\varphi},\sigma_1(y_1))$ induced by the embedding $\sigma_1 \colon \Lambda_{K_1} \to L_{\varphi}$, that is, $r_1\circ (\sigma_1)_*= \id_{\pi_1( \Lambda_{K_1} ,y_1) }$.

Let us consider the composite map
\begin{align}\label{hom-r1-sigma0}
\xymatrix{ \pi_1(\Lambda_{K_0},y_0) \ar[r]^-{(\sigma_0)_*} &\pi_1(L_{\varphi},\sigma_0(y_0)) \cong \pi_1(L_{\varphi},\sigma_1(y_1))  \ar[r]^-{r_1} & \pi_1(\Lambda_{K_1},y_1) ,}
\end{align}
where the middle isomorphism is determined by choosing a path in $L_{\varphi}$ from $\sigma_0(y_0)$ to $\sigma_1(y_1)$.
Using $a\in\{\pm 1\}$ and $b\in \bZ$ of Proposition \ref{prop-H1L}, (\ref{hom-r1-sigma0}) maps  
$ m_0$ to $( m_1)^a$ and $l_0$ to $(l_1)^a ( m_1 )^b$.
By Proposition \ref{prop-b=0}, we know that $b=0$.
Hereafter, by changing the orientation of $K_1$ if necessary, we assume that $a=1$.
Therefore, (\ref{hom-r1-sigma0}) is given by
\begin{align}\label{r1-sigma0-rel}
(\ref{hom-r1-sigma0}) \colon \pi_1(\Lambda_{K_0},y_0) \to \pi_1(\Lambda_{K_1},y_1) ,\  m_0 \mapsto  m_1,\ l_0 \mapsto  l_1.
\end{align} 

From the construction (\ref{glue-M_phi}) of the manifold $M_{\varphi}$ and the Seifert--van Kampen theorem,
there exist a unique group homomorphism
\[\bar{r}_1 \colon \pi_1(M_{\varphi} , y'_1) \to \pi_1 (\bR^3\setminus K_1 , e_1(y_1)) \]
such that the following diagram commutes:
\begin{align}\label{diagram-van-kampen} 
\begin{split}
\xymatrix{
\pi_1(L_{\varphi}, \sigma_1(y_1)) \ar[r]^-{i} \ar@/^18pt/[rr]^{(e_1)_* \circ r_1} & \pi_1(M_{\varphi} , y'_1) \ar@{-->}[r]^-{\bar{r}_1}  & \pi_1 (\bR^3\setminus K_1 , e_1(y_1))\\
    \pi_1(\Lambda_{K_1}, y_1) \ar[r]^-{(e_1)_*} \ar[u]^-{(\sigma_1)_*}  & \pi_1(\bR^3\setminus N_{K_1}, e_1(y_1)) .  \ar[u]_-{i'} \ar[ur]_-{i''}^-{\cong} & 
}\end{split}
\end{align}
Here, the homomorphisms $i$, $i'$ and $i''$ are induced by the inclusion maps.
We remark that $\bar{r}_1$ is a surjection.

We choose a path from $e_0(y_0)$ to $x_0$ in $\bR^3\setminus K_0$ and a path from $y'_1$ to $x_0$ in $M_{\varphi}$ to fix isomorphisms
\[\begin{array}{cc}
\pi_1(\bR^3\setminus K_0, e_0(y_0)) \cong \pi_1(\bR^3\setminus K_0, x_0) = G_0 , &  \pi_1(M_{\varphi},y'_1) \cong \pi_1(M_{\varphi},x_0) = G_1 .
\end{array}\]
Then, we obtain a surjective group homomorphism
\[k\colon \pi_1(\bR^3\setminus K_0, e_0(y_0)) \to \pi_1(\bR^3 \setminus K_1 ,e_1(y_1)) \]
by the composition
\begin{align}\label{hom-r_1-h}
\xymatrix{
\pi_1(\bR^3\setminus K_0, e_0(y_0)) \cong G_0 \ar@{->>}[r]^-{h} & G_1 \cong \pi_1(M_{\varphi},y'_1) \ar@{->>}[r]^-{\bar{r}_1} & \pi_1(\bR^3 \setminus K_1 ,e_1(y_1)) .
}\end{align}

\begin{remark}\label{rem-epi-knot}
Given knots $K$ and $K'$, let us write $K \geq K'$ if there exists a surjective group homomorphism $\pi_1(\bR^3\setminus K) \to \pi_1(\bR^3\setminus K')$.
(In this section, following the preceding studies, let us simply call it an \textit{epimorphism} between knot groups.)
Then, for the knots $K_0$ and $K_1$ in this section, we have $K_0 \geq K_1$.
There have been a number of studies on epimorphisms between knot groups.
For example, the following results are known.
\begin{itemize}
\item Let $\Delta_K\in \bZ[t]$ denote the Alexander polynomial of a knot $K$.
It is well-known that $\Delta_{K}$ is divisible by $\Delta_{K'}$ if $K\geq K'$.
\item For any knot $K$, $\pi_1(\bR^3\setminus K)$ maps onto at most finitely many knot groups (up to isomorphisms) \cite[Corollary 1.2]{AgLiu}.
\item When $K$ and $K'$ are $2$-bridge knots, \cite[Theorem 1.1]{ORS} gave a sufficient condition to be $K \geq K'$, and a necessary and sufficient condition was given by \cite[Theorem 8.1]{ALSS}.
\end{itemize}
We will discuss in Remark \ref{rem-epi-peripheral} some refinements of the condition $K \geq K'$.
\end{remark}

The epimorphism $k$ has the following property.

\begin{proposition}\label{prop-k-rep}
For $j\in \{0,1\}$, let us abbreviate the meridian $(e_j)_*(m_j)$ and the longitude $(e_j)_*(l_j)$ of $K_j$ by $m_j$ and $l_j$ respectively.
Then, for any representation
\[\rho \colon  \pi_1(\bR^3\setminus K_1, e_1(y_1)) \to \GL_n(\bC),\] there exists $g \in \GL_n(\bC)$ such that
\[\begin{array}{cc} (\rho \circ k) (m_0) = g\cdot \rho(m_1) \cdot g^{-1}, & (\rho\circ k) (l_0) = g\cdot \rho(l_1)\cdot g^{-1}.
\end{array}\]
\end{proposition}
\begin{proof}
It suffices to study representations of the fundamental groups up to equivalence, so we do not need to care about the choice of base point.
For any representation $\rho \colon  \pi_1(\bR^3\setminus K_1) \to \GL_n(\bC)$, we consider the following diagram:
\[\xymatrix{
\pi_1(\Lambda_{K_0}) \ar[d]_-{(e_0)_*} \ar[rr]^-{   r_1\circ (\sigma_0)_* } \ar[rd]^-{(\sigma_0)_*} & & \pi_1(\Lambda_{K_1}) \ar[d]_-{(e_1)_*}  & \\
\pi_1(\bR^3\setminus K_0) \ar@/_18pt/[rr]^-{k}  \ar[r]^-{h} & \pi_1(M_{\varphi}) \ar[r]^-{\bar{r}_1} & \pi_1(\bR^3\setminus K_1) \ar[r]^-{\rho} & \GL_n(\bC)
}\]
From the definitions of $\bar{r}_1$ and $k$, we have
\[\begin{array}{cc} (e_1)_*\circ (r_1\circ (\sigma_0)_*) = \bar{r}_1 \circ (\sigma_0)_* , & \bar{r}_1\circ h =k. \end{array}\]
We apply Proposition \ref{prop-GL_n-rep} to the finite dimensional representation $\rho \circ \bar{r}_1$ of $\pi_1(M_{\varphi})$, then
\[  (\rho\circ \bar{r}_1)\circ h\circ (e_0)_* \simeq (\rho \circ \bar{r}_1) \circ (\sigma_0)_* \]
in $\Rep^{\fin}_{\bC} (\pi_1(\Lambda_{K_0}))$.
Therefore,
\[\rho \circ k \circ (e_0)_* \simeq \rho \circ (e_1)_* \circ (r_1\circ (\sigma_0)_*) . \]
in $\Rep^{\fin}_{\bC} (\pi_1(\Lambda_{K_0}))$.
The map $ r_1\circ (\sigma_0)_* $ agrees with (\ref{hom-r1-sigma0}), so it maps $m_0$ to $m_1$ and $l_0$ to $l_1$.
Take $g\in \GL_n(\bC)$ such that $\rho \circ k \circ (e_0)_* (x) = g\cdot (\rho \circ (e_1)_* \circ (r_1\circ (\sigma_0)_*)(x))\cdot g^{-1}$ for every $x\in \pi_1(\Lambda_{K_0})$.
Then the assertion of the proposition holds.
\end{proof}

Consider $T^*_{K_0}\bR^3$ as a $3$-manifold diffeomorphic to $S^1\times \bR^2$.
In this section,
we say that $\varphi^{-1}(K_1)$ is a \textit{trivial knot} in $T^*_{K_0}\bR^3$ if it is isotopic to $K_0 = T^*_{K_0}\bR^3 \cap 0_{\bR^3}$ in $T^*_{K_0}\bR^3$.
In such case, $\psi^{-1}(L_{\varphi})$ is diffeomorphic to $\bR\times \Lambda_{K_0}$ and it is an exact Lagrangian concordance in $\bR \times U^*\bR^3$ from $\Lambda_{K_1}$ to $\Lambda_{K_0}$.

\begin{theorem}\label{thm-unknot}
If $K_0$ is the unknot in $\bR^3$, then $K_1$ is also the unknot in $\bR^3$.
Moreover, $\varphi^{-1}(K_1)$ is a trivial knot in $T^*_{K_0}\bR^3$.
\end{theorem}
\begin{proof}
From the existence of an epimorphism (\ref{hom-r_1-h}), if $K_0$ is the unknot, then $\pi_1(\bR^3\setminus K_1)\cong \bZ$.
By the loop theorem, $K_1$ must be the unknot in $\bR^3$.

To prove the latter assertion,
let us use the $3$-manifold $M'_{\varphi}$ defined in the proof of Lemma \ref{lem-LO}.
As we have seen there, if $K_1$ is the unknot, then $M'_{\varphi}$ is a $3$-manifold obtained from $T^*_{K_0}\bR^3\cong K_0 \times \bR^2$ by a Dehn surgery along $\varphi^{-1}(K_1)$.
From the existence of a surjective homomorphism
$h\colon \bZ\cong \pi_1(\bR^3\setminus K_0)\to \pi_1(M'_{\varphi})$ and $H_1(M'_{\varphi};\bZ)\cong \bZ$, we have $\pi_1(M'_{\varphi}) \cong \bZ$.
This implies that the embedded torus $\sigma_0(\Lambda_{K_0})$ is not incompressible since $(\sigma_0)_* \colon \pi_1(\Lambda_{K_0}) \to \pi_1(M'_{\varphi})$ is not injective.
Therefore, it follows from Theorem \ref{thm-Gabai} that $M'_{\varphi} \cong S^1 \times \bR^2$ and moreover, $\varphi^{-1}(K)$ is isotopic to the trivial knot $K_0\times \{0\}$.
\end{proof}

\begin{theorem}
Suppose that $K_1$ has the same knot type as $K_0$ in $\bR^3$.
Then, $\varphi^{-1}(K_1)$ is a trivial knot in $T^*_{K_0}\bR^3$.
\end{theorem}
\begin{proof}
By Theorem \ref{thm-unknot}, we may assume that $K_1$ is not the unknot in $\bR^3$.
By \cite[Theorem 1.1]{Hem}, it is known that for any knot $K$ in $\bR^3$, $\pi_1(\bR^3 \setminus K)$ is finitely generated and residually finite, and it implies that $\pi_1(\bR^3\setminus K)$ is a Hopfian group.
Therefore, if $K_0$ is isotopic to $K_1$, the epimorphism (\ref{hom-r_1-h}) from $\pi_1(\bR^3\setminus K_0)$ to $\pi_1(\bR^3\setminus K_1)$ must be an isomorphism.
From the construction of (\ref{hom-r_1-h}), the surjection $h$ is also an isomorphism, and thus $\bar{r}_1$ must be an isomorphism.

We refer to the diagram (\ref{diagram-van-kampen}).
Since $K_1$ is not the unknot from our assumption, $e_1(\Lambda_{K_1}) = \partial N_{K_1}$ is an incompressible surface in $\bR^3\setminus K_1$.
This means that $(e_1)_*$ is injective.
The map $(\sigma_1)_*$ is also injective since it has a retraction $r_1$.
From a general property of amalgamated product \cite[Chapter IV, Theorem 2.6]{LynSch}, both of the maps
\[ \begin{array}{cc}  i \colon  \pi_1(L_{\varphi})\to \pi_1(M_{\varphi}) , & i' \colon \pi_1(\bR^3\setminus N_{K_1})\to \pi_1(M_{\varphi}), \end{array}\]
in the diagram (\ref{diagram-van-kampen}) are injective.
Since $\bar{r}_1$ is an isomorphism, $(e_1)_*\circ r_1 =\bar{r}_1\circ i$ is also injective, which in turn implies that $r_1$ is injective.
Since $r_1$ is a retraction of $(\sigma_1)_*$, it follows that $(\sigma_1)_* \colon \pi_1(\Lambda_{K_1},y_1) \to \pi_1 (L_{\varphi},y_1)$ is an isomorphism, too.

Consider a $3$-manifold $M$ obtained from a non-trivial surgery in $T^*_{K_0}\bR^3$ along $\varphi^{-1}(K_1)$.
We note that
\[\varphi^{-1}\circ \sigma_1\colon \Lambda_{K_1} \to  T^*\bR^3\setminus \varphi^{-1}(K_1) = \varphi^{-1}(L_{\varphi})\]
is an embedding onto the boundary of a tubular neighborhood of $\varphi^{-1}(K_1)$, and it induces an isomorphism on the fundamental group.
This implies that $\pi_1(M) \cong \bZ$.
In particular, $H_1(M;\bZ)\cong \bZ$ is torsion free and $\sigma_0(\Lambda_{K_0})$ is not an incompressible surface.
We apply Theorem \ref{thm-Gabai}, then $M\cong \bR^2\times S^1$ and $\varphi^{-1}(K_1)$ is isotopic to the trivial knot in $T^*_{K_0}\bR^3$.
\end{proof}

Let $T_{(p,q)}$ denote the $(p,q)$-torus knot in $\bR^3$ for integers $p,q$ which are coprime.
In the following, when we say that a knot is the $(p,q)$-torus knot, we always assume that $1 < p < |q|$.

For an epimorphism from $\pi_1(\bR^3\setminus T_{(p,q)})$ to a knot group, the following is known.

\begin{proposition}[{\cite[Proposition 2.4]{SW2}}]\label{prop-epi-torus}
Let $K'$ be a knot in $\bR^3$ and $K$ be the $(p,q)$-torus knot.
If there exists an epimorphism $\pi_1(\bR^3\setminus K) \to \pi_1(\bR^3 \setminus K')$,
then $K'$ is either the unknot or the $(p',q')$-torus knot with $p' \mid p$ and $q'\mid q$, or $p' \mid q$ and $q'\mid p$. 
\end{proposition}

\begin{remark}
This fact is deduced from a classical result which characterizes torus knots.
Let us outline the proof of \cite[Proposition 2.4]{SW2}.
For any knot $K$, let $Z_K$ denote the center of the group $\pi_1(\bR^3\setminus K)$.
\begin{enumerate}
\item Let $f \colon \pi_1(\bR^3\setminus T_{(p,q)}) \to \pi_1(\bR^3\setminus K')$ be an epimorphism.
Assume that $Z_{K'}=\{1\}$. Then, $f$ factors through an epimorphism
\[(\bZ/ p\bZ)* (\bZ/q\bZ) \cong  \pi_1(\bR^3\setminus T_{(p,q)}) / Z_{T_{(p,q)}} \to \pi_1(\bR^3\setminus K').\]
However, since $\pi_1(\bR^3\setminus K')$ is locally indicable \cite[Lemma 2]{HS}, there is no non-trivial element which has a finite order, so we obtain a contradiction. 
\item 
$Z_{K'}\neq \{1\}$ if and only if $K'$ is either the unknot or a torus knot \cite{BZ}.
See also \cite[Theorem 6.1]{BZH}.
\item Suppose that $K'$ is the $(p',q')$-torus knot. Then, the epimorphism $f \colon \pi_1(\bR^3\setminus T_{(p,q)}) \to \pi_1(\bR^3\setminus T_{(p',q')})$ induces a surjective homomorphism $(\bZ/ p\bZ)* (\bZ/q\bZ) \to (\bZ/ p'\bZ)* (\bZ/q'\bZ)$.
By an algebraic argument on torsion elements of $(\bZ/ p'\bZ)* (\bZ/q'\bZ)$, we can deduce that $p'\mid p$ and $q' \mid q$, or $p'\mid q$ and $q' \mid p$.
\end{enumerate}
\end{remark}

\begin{theorem}\label{thm-trefoil}
If $K_0$ is the $(2,q)$-torus knot for an odd number $q$, then $K_1$ is also the $(2,q)$-torus knot.
\end{theorem}
\begin{proof}
Suppose that $K_0=T_{(2,q)}$. Then, Proposition \ref{prop-epi-torus} shows that $K_1 $ is $ T_{(2,q')}$ with $q'|q$ or the unknot.
By \cite[Theorem 1.5]{O23}, $K_1 \neq T_{(2,q')}$ for $q'\neq q$.
Moreover, by \cite[Theorem 1.1]{O25}, $K_1$ is not the unknot.
Therefore, $K_1 = T_{(2,q)}$.
\end{proof}

A more general case where $K_0$ is a $(p,q)$-torus knot will be discussed in Theorem \ref{cor-torus}.

Let us also consider the figure-eight knot $4_1$.
\begin{theorem}\label{thm-figure-8}
If $K_0$ is the figure-eight knot, $K_1$ is also the figure-eight knot.
\end{theorem}
\begin{proof}
Suppose that $K_0=4_1$.
By \cite[Theorem 2.1]{KitanoSuzukiII}, $4_1$ is minimal, that is, if there exists an epimorphism $\pi_1(\bR^3\setminus 4_1)\to \pi_1(\bR^3\setminus K')$ for a knot $K'$, then $K'$ is either $4_1$ or the unknot.
(Precisely, \cite[Theorem 2.1]{KitanoSuzukiII} states the minimality of $4_1$ among prime knots.
We note that $4_1 \geq K_1\#K_2$ implies $4_1\geq K_1$ and $4_1 \geq K_2$.)
By \cite[Theorem 1.1]{O25}, $K_1$ is not the unknot. Therefore, $K_1=4_1$.
\end{proof}

\subsection{\texorpdfstring{$A$-polynomials}{A-polynomials}}

We review the definition of an enhancement of $A$-polynomial from \cite[Section 2]{NiZ}.
See also \cite[Section 10.1]{BS23}.

We prepare several notations.
Let $G$ be a finitely generated group.
For any $\SL_2(\bC)$-representation $\rho$ of $G$, its \textit{character} is a map defined by
\[\chi_{\rho}\colon G \to \bC ,\  g \mapsto \tr (\rho(g)).\]
We denote the set of characters of $\SL_2(\bC)$-representations of $G$ by
\[ X(G) \coloneqq \{\chi_{\rho} \colon G \to \bC  \mid \rho \text{ is a }\SL_2(\bC)\text{-representation of }G \}. \]
It is identified with an affine algebraic set. See \cite[Section 1.4, Corollary 1.4.5]{CS}. $X(G)$ is called the \textit{character variety} of $G$.
When $G= \pi_1(M,x)$ for a connected manifold $M$ with a base point $x$, we denote $X(M) \coloneqq X(\pi_1(M,x))$.
\begin{remark}
For any $x_1,x_2\in M$, there is a canonical isomorphism $X(\pi_1(M,x_1)) \cong X(\pi_1(M,x_2))$ of algebraic sets.
Thus, when discussing $X(M)$ and its elements, we will omit writing the base point $x$ from $\pi_1(M,x)$.
\end{remark}

Let $K$ be an oriented knot in $\bR^3$ equipped with the Seifert framing.
We take a tubular neighborhood $N_K$ of $K$ and denote
\[M_K\coloneqq \bR^3\setminus N_K.\]
For $\Lambda_K$, 
let $e \colon \Lambda_K \to  M_K$ be an embedding onto $\partial N_K$ defined as (\ref{emb-ei}).
The generators $\{m,l\}$ of $\pi_1(\Lambda_K)$ is determined such that $e_*(m)$ and $e_*(l)$ are the meridian and the longitude respectively of the framed knot $K$.
The embedding $e \colon \Lambda_K \to M_K$ defines a regular map
\[e^* \colon X(M_K) \to X(\Lambda_K) ,\  \chi_{\rho} \mapsto \chi_{\rho \circ e_*} =\chi_{\rho} \circ e_* .\]

We also consider a map
\[ t \colon \bC^* \times \bC^* \to X(\Lambda_K)\]
such that for any $(\mu,\lambda)\in \bC^*\times \bC^*$, $t(\mu,\lambda) \colon \pi_1(\Lambda_K) \to \bC $ is the character of the representation
\[ \pi_1(\Lambda_K) \to \SL_2(\bC) ,\  m \mapsto \begin{pmatrix} \mu & 0 \\ 0 & \mu^{-1} \end{pmatrix},\ l \mapsto \begin{pmatrix} \lambda & 0 \\ 0 & \lambda^{-1} \end{pmatrix}.\]
Let $\scrX_K$ be the set of irreducible components $X$ of the algebraic set $X(M_K)$ such that $\dim_{\bC} \overline{e^*(X)}=1$.
We have an element $X_{K}^{\mathrm{red}}\in \scrX_K$, which we call the trivial component, consisting of characters of reducible representations of $\pi_1(M_K)$.
Using the commutator subgroup $[\pi_1(M_K),\pi_1(M_K)]$ of $\pi_1(M_K)$, 
\begin{align}\label{X-red-commutator}
X^{\mathrm{red}}_K = \{ \chi \in X(M_K) \mid \chi (c) =2 \text{ for every } c\in [\pi_1(M_K),\pi_1(M_K)] \}.
\end{align}
For the proof, see \cite[Corollary 1.2.2]{CS}.

Now, consider a subset of $\bC^*\times \bC^*$
\begin{align}\label{tilde-V-K}
\tilde{V}_K\coloneqq  t^{-1} \left( \bigcup_{X_K^{\mathrm{red}} \neq X \in \scrX_K} \overline{e^* (X)} \right) .
\end{align}
We define $\tilde{A}_K\in \bC[\mu,\lambda]$ as a polynomial with no repeated factors such that its zero locus agrees with the Zariski closure of $\tilde{V}_K$ in $\bC^2$.
It is determined up to multiplication by non-zero constants.

\begin{remark}
$\scrX_K= \{X_K^{\mathrm{red}}\}$ if and only if $K$ is the unknot \cite{DG}, and in this case, we define $\tilde{A}_K\coloneqq 1$.
For any knot $K$,
its A-polynomial $A_{K} \in \bZ[\mu,\lambda]$ was defined by \cite[Section 2.3]{CCGLS} as a polynomial with no repeated factors such that its zero locus agrees with the Zariski closure of $ t^{-1} \left( \bigcup_{X \in \scrX_K} \overline{e^* (X)} \right)$ in $\bC^2$.
Since $t^{-1}\left( \overline{e_*(X^{\mathrm{red}}_{K})}\right) = \{\lambda=1\}$, $A_K$ is the least common multiple of $\tilde{A}_K$ and $\lambda-1$. See \cite[Remark 10.1]{BS23}.
\end{remark}

\begin{remark}[Independence on the orientation of $K$]\label{rem-indep-ori-K}
Since $t(\mu,\lambda) = t(\mu^{-1},\lambda^{-1})$, the set $\tilde{V}_K$ is preserved by the map $(\bC^*)^2 \to (\bC^*)^2 ,\  (\mu,\lambda) \mapsto (\mu^{-1},\lambda^{-1})$.
On the other hand, there is a natural bijection $\tilde{V}_{K} \to \tilde{V}_{-K} ,\  (\mu,\lambda) \mapsto (\mu^{-1},\lambda^{-1})$, where
$-K$ is the knot $K$ with the opposite orientation.
Therefore, $\tilde{A}_{K} = \tilde{A}_{-K}$.
\end{remark}

We want to deduce a relation between $\tilde{A}_{K_0}$ and $\tilde{A}_{K_1}$ from Proposition \ref{prop-k-rep}.

For $i\in \{0,1\}$, we have
the embedding $e_i \colon \Lambda_{K_i} \to M_{K_i}$ and the generators $\{m_i,l_i\}$ of $\pi_1(\Lambda_{K_i})$.
The map (\ref{hom-r_1-h}) gives a surjective group homomorphism
\[k  \colon \pi_1(M_{K_0}) \to \pi_1(M_{K_1}).\]
This defines a map $k^* \colon X(M_{K_1}) \to X(M_{K_0}) ,\  \chi_{\rho} \mapsto \chi_{\rho\circ k}$.
The next lemma follows immediately from Proposition \ref{prop-k-rep}.

\begin{lemma}\label{lem-Ce=ek}
Let $C^* \colon X(\Lambda_{K_1}) \to X(\Lambda_{K_0}) ,\  \chi \mapsto \chi \circ C$
be the map defined from an isomorphism
\[ C \colon \pi_1(\Lambda_{K_0}) \to \pi_1(\Lambda_{K_1}) ,\  m_0 \mapsto m_1 , \ l_0 \mapsto  l_1 . \]
Then, the following diagram commutes:
\[\xymatrix{
X(\Lambda_{K_1}) \ar[r]^-{C^*} & X(\Lambda_{K_0}) \\
X(M_{K_1}) \ar[r]^-{k^*} \ar[u]_-{e_1^*} & X(M_{K_0}) \ar[u]_-{e_0^*}.
}\]
\end{lemma}

The regular map $k^* \colon X(M_{K_1}) \to X(M_{K_0})$ has the following property.
\begin{lemma}\label{lem-X-hat}
For every $X \in \scrX_{K_1}$, there exists $\hat{X} \in \scrX_{K_0}$ such that $k^*(X)\subset \hat{X}$.
Moreover, $k^*(X)\subset X^{\mathrm{red}}_{K_0}$ if and only if $X=X^{\mathrm{red}}_{K_1}$.
\end{lemma}
\begin{proof}
For any $X\in \scrX_{K_1}$, $k^*(X)$ is an irreducible subset of $X(M_{K_0})$.
Let $\hat{X}$ be an irreducible component of $X(M_{K_0})$ which contains $k^*(X)$.
By \cite[Lemma 2.1]{DG}, $\dim_{\bC} \overline{e_0^*(\hat{X})}$ is either $0$ or $1$.
Since $\dim_{\bC} \overline{C^*(e_1^*(X))}= \dim_{\bC} \overline{e_1^*(X)} =1$ and
\[C^*(e_1^*(X)) = e_0^*(k^*(X)) \subset e_0^*(\hat{X})\] by Lemma \ref{lem-Ce=ek}, it follows that $\dim_{\bC} \overline{e_0^*(\hat{X})}=1$ and thus $\hat{X}\in \scrX_{K_0}$. 

If we take $X= X^{\mathrm{red}}_{K_1}$, it is obvious that
\begin{align*} 
k^*(X^{\mathrm{red}}_{K_1}) & = \{ \chi_{\rho \circ k} \mid \rho \colon \pi_1(M_{K_1}) \to \SL_2(\bC) \text{ is a reducible representation}\} \\
& \subset X^{\mathrm{red}}_{K_0}.
\end{align*}
Conversely, suppose that $k^*(X) \subset X^{\mathrm{red}}_{K_0}$ for $X\in \scrX_{K_1}$.
Take any character $\chi\in X$.
Let us use the description (\ref{X-red-commutator}) for $X^{\mathrm{red}}_{K_0}$ and $X^{\mathrm{red}}_{K_1}$.
Then, $k^*(\chi) = \chi \circ k \in X^{\mathrm{red}}_{K_0}$ 
is constant to $2$ on the commutator subgroup $[\pi_1(M_{K_0}), \pi_1(M_{K_0})]$.
Since $k\colon \pi_1(M_{K_0}) \to \pi_1(M_{K_1})$ is a surjection, the image $k([\pi_1(M_{K_0}), \pi_1(M_{K_0})])$ is a normal subgroup of $\pi_1(M_{K_1})$ which contains $[\pi_1(M_{K_1}), \pi_1(M_{K_1})]$.
Therefore, $\chi$ must be constant to $2$ on $[\pi_1(M_{K_1}), \pi_1(M_{K_1})]$, and this means that $\chi_{} \in X^{\mathrm{red}}_{K_1}$. Therefore, $X=X^{\mathrm{red}}_{K_1}$.
\end{proof}
We remark that the surjectivity of $k$ is essential for the latter part of this proof.

Using the above two lemmata, we can deduce a relation between $\tilde{A}_{K_0}$ and $\tilde{A}_{K_1}$.
\begin{theorem}\label{thm-divide-A}
$\tilde{A}_{K_0}$ is divisible by $\tilde{A}_{K_1}$.
\end{theorem}
\begin{proof}
For any $X \in \scrX_{K_1} \setminus \{X^{\mathrm{red}}_{K_1} \}$, we have $\hat{X} \in \scrX_{K_0}\setminus \{X^{\mathrm{red}}_{K_0} \}$ from Lemma \ref{lem-X-hat} such that $k^*(X) \subset \hat{X}$, and Lemma \ref{lem-Ce=ek} shows that
\[ C^* (e_1^*(X)) = (e_0^*\circ k^*)(X) \subset e_0^* (\hat{X}).\]
Therefore, $C^* \colon X(\Lambda_{K_1}) \to X(\Lambda_{K_0})$ is restricted to a map
\[ C^* \colon \bigcup_{X \in \scrX_{K_1} \setminus \{X^{\mathrm{red}}_{K_1} \}} \overline{ e_1^*(X) } \to \bigcup_{X \in \scrX_{K_0} \setminus \{X^{\mathrm{red}}_{K_0} \}} \overline{ e_0^*(X) } . \]
Obviously, $C^*$ commutes with the maps $t \colon \bC^*\times \bC^* \to X(\Lambda_{K_i})$ for $i=0,1$.
From the definition (\ref{tilde-V-K}) of $\tilde{V}_{K_i}$ for $i=0,1$, we obtain $\tilde{V}_{K_1} \subset \tilde{V}_{K_0}$ in $\bC^*\times \bC^*$.
This implies that the polynomial $\tilde{A}_{K_1}$ divides the polynomial $\tilde{A}_{K_0}$.
\end{proof}

Using the divisibility of (an enhanced version of) A-polynomials, we can deduce the following result.

\begin{corollary}\label{cor-torus}
If $K_0$ is the $(p,q)$-torus knot, then $K_1$ is either  the $(p,q)$-torus knot or the unknot.
\end{corollary}
\begin{proof}
Suppose that $K_0= T_{(p,q)}$ and $K_1$ is not the unknot.
Then, Proposition \ref{prop-epi-torus} shows that $K_1=T_{(p',q')}$ with $p'\mid p$ and $q' \mid q$, or $p'\mid q$ and $q' \mid p$.
Moreover, Theorem \ref{thm-divide-A} shows that $\tilde{A}_{T_{(p,q)}}$ is divided by $\tilde{A}_{T_{(p',q')}}$.
By \cite[Proposition 10.3]{BS23},
$\tilde{A}_{T_{(p,q)}}$ divides $\mu^{2pq}\lambda^{2} -1$, and by \cite[Proposition ]{CCGLS}, $\tilde{A}_{T_{(p',q')}}$ is divided by $\mu^{p'q'}\lambda +1$. In summary,
\[ (\mu^{p'q'}\lambda+1) |  \tilde{A}_{T_{(p',q')}} | \tilde{A}_{T_{(p,q)}} | (\mu^{2pq}\lambda^{2} -1 ). \]
Hence, $\mu^{2pq}\lambda^2-1$ is divided by $\mu^{p'q'}\lambda+1$, and this shows that $pq=p'q'$.
Since we know that $p'\mid p$ and $q' \mid q$, or $p'\mid q$ and $q' \mid p$, it follows that $T_{(p',q')} = T_{(p,q)}$.
\end{proof}

\begin{remark}
\cite[Theorem 10.14]{BS23} can be used to give an alternative proof of Theorem \ref{thm-trefoil} when $(p,q)=(2,3)$, without using \cite[Proposition 2.4]{SW2} and \cite[Theorem 1.5]{O23}.
For any polynomial $A = \sum_{i=0}^m\sum_{j=0}^{m'} a_{i,j} \lambda^i \mu^j \in \bC[\lambda,\mu]$, its Newton polygon $\mathcal{N}_A$ is defined as the convex hull in $\bR\times \bR$ of the set
$ \{ (i,j) \mid a_{i,j}\neq 0 \}$.
It is well-known that for $A,B \in \bC[\lambda,\mu]$,
\begin{align}\label{Minkowski-sum} 
\mathcal{N}_{AB} = \{ x+y \in \bR^2 \mid x\in \mathcal{N}_A,\ y\in \mathcal{N}_{B} \}.
\end{align}
For $r\in \bQ$, a non-trivial knot $K$ is called $r$\textit{-thin} if $\mathcal{N}_{\tilde{A}_K}$ is contained in a line segment of slope $r$ \cite[Definition 10.5]{BS23}.
By \cite[Example 10.6]{BS23}, $T_{(2,3)}$ is $6$-thin.
Suppose that $K_0=T_{(2,3)}$. Then, $K_1$ is a non-trivial knot by \cite[Theorem 1.1]{O25}.
Since $\tilde{A}_{K_1}$ divides $\tilde{A}_{T_{(2,3)}}$, it follows from (\ref{Minkowski-sum}) that $K_1$ is also $6$-thin.
Then, by \cite[Theorem 10.14]{BS23}, we can conclude that $K_1=T_{(2,3)}$.
\end{remark}

\subsection{KCH representations and augmentation variety}\label{subsec-KCH-rep}

Let $K$ be an oriented knot in $\bR^3$.

\begin{definition}[{\cite[Definition 3.2]{cornwell-1}}]
A representation $\rho \colon \pi_1(\bR^3\setminus K) \to \GL_n(\bC)$ for $n\in \bZ_{\geq 1}$ is called a \textit{KCH representation} if there exists  $\mu\in \bC^*\setminus \{1\}$ such that $\rho(m)$ is conjugate to the diagonal matrix
$\mathrm{diag}(\mu,1,\dots ,1)$.
\end{definition}

As explained in Remark \ref{rem-KCH} below, KCH representations are closely related to the knot contact homology of $K$.

For any KCH representation $\rho$, $\rho(m)$ commutes with $\rho(l)$.
Thus, if we take $\mu\in \bC^*\setminus \{1\}$ and $v\in \bC^n\setminus \{0\}$ such that $\rho(m)(v) = \mu \cdot v$, then $\rho(l)(v) = \lambda \cdot v$ for some $\lambda \in \bC^*$.
Let us write the pair $(\lambda, \mu)$ by $(\lambda_{\rho}, \mu_{\rho})$.
Let $\operatorname{KCH}^{\mathrm{irr}}_n(K)$ denote the set of $n$-dimensional irreducible KCH representations of $\pi_1(\bR^3\setminus K)$.
(For each $K$, there are only finitely many $n\in \bZ_{\geq 1}$ such that $\operatorname{KCH}^{\mathrm{irr}}_n(K)\neq \emptyset$. See \cite[Theorem 1]{cornwell-1}.)
For every $n\in \bZ_{\geq 1}$, we define
\begin{align}\label{KCH-Un} 
\tilde{U}^n_K\coloneqq \{ (x,y)\in (\bC^*)^2 \mid \text{there exists }\rho \in \operatorname{KCH}^{\mathrm{irr}}_n(K) \text{ such that }x=\lambda_{\rho} \text{ and }y=\mu_{\rho} \}. 
\end{align}
If we omit the the condition that $\rho$ is irreducible, it coincides with a subset $U^n_K \subset (\bC^*)^2$ considered in \cite[Section 3.2]{cornwell-1}.

For any $\rho \in \operatorname{KCH}^{\mathrm{irr}}_n(K)$, a representation $\rho' ,\  g\mapsto {^{t} (\rho(g))^{-1}}$ also belongs to $\operatorname{KCH}^{\mathrm{irr}}_n(K)$.
Therefore, $\tilde{U}^n_K$ is preserved by the map $(\bC^*)^2 \to (\bC^*)^2 \colon (x,y) \to (x^{-1}, y^{-1})$.
This implies that $\tilde{U}^n_K$ is independent of the orientation of $K$.

By using the epimorphism $k\colon \pi_1(\bR^3\setminus K_0) \to \pi_1(\bR^3\setminus K_1)$, we can deduce a relation between $\tilde{U}^n_{K_0}$ and $\tilde{U}^n_{K_1}$.

\begin{theorem}\label{thm-KCH}
For every $n\in \bZ_{\geq 1}$, $\tilde{U}^n_{K_1}\subset \tilde{U}^n_{K_0}$.
\end{theorem}
\begin{proof}
Take any $\rho \in \operatorname{KCH}^{\mathrm{irr}}_n(K_1)$.
The representation $\rho \circ k \colon \pi_1(\bR^3\setminus K_0) \to \GL_n(\bC)$ is irreducible since $k$ is surjective.
Moreover, Proposition \ref{prop-k-rep} shows that $\rho\circ k$ satisfies the condition to be a KCH representation with
$(\lambda_{\rho\circ k},\mu_{\rho\circ k}) = (
\lambda_{\rho},\mu_{\rho})$. 
 This shows that $\tilde{U}^n_{K_1}\subset \tilde{U}^n_{K_0}$.
\end{proof}

\begin{remark}\label{rem-KCH}
Ng introduced in  \cite[Definition 5.4]{Ng} a subset $(\bC^*)^2$, called the \textit{augmentation variety} of $K$.
Here, let us refer to the convention of \cite[Definition 5.1]{Ng-intro} and denote it by $V_K\subset (\bC^*)^2$.
It is defined from the knot contact homology of $K$ with coefficients in $\bZ[\lambda^{\pm},\mu^{\pm}]$.
It follows from \cite[Theorem 1.2]{cornwell-2} that
$V_K = \{y =1 \} \cup \bigcup_{n=1}^{\infty} \tilde{U}^n_K $.
Theorem \ref{thm-KCH} shows that $V_{K_1} \subset V_{K_0}$, and this relation is a refinement of \cite[Theorem 1.1]{O23}.
\end{remark}

\section{Existence of peripheral-pair-preserving surjection}\label{subsec-peripheral}

We go back and continue the discussion in Section \ref{sec-sheaf}.

\noindent
\textbf{Convention.}
In this section, for $i\in \{0,1\}$, let us identify $\Lambda_{K_i}$ with $\partial N_{K_i}$ via the map $e_i$.
Then, $m_i,l_i\in \pi_1(\partial N_{K_i},y_0)$ denote the longitude and the meridian of $K_i$ respectively.
The map $\sigma_i$ is considered as an embedding $\sigma_i\colon \partial N_{K_i} \to  L_{\varphi}$.

\begin{notation}
In this section, we will discuss fundamental groups of compact $3$-manifolds with toroidal boundaries.
For this reason, instead of $\bR^3\setminus K_0$ and $M_{\varphi}$ introduced in Section \ref{subsec-clean}, 
we will often consider compact $3$-manifolds $S^3\setminus N_{K_0}$ and
\[M'_{\varphi} \coloneqq \left( S^3 \setminus N_{K_1} \right) \cup_{\sigma_1} \left( L_{\varphi} \cap \psi([a_1,a_0]\times U^*\bR^3 ) \right) .  \]
The boundary of $S^3\setminus N_{K_0}$ is $\partial N_{K_0}$, and the boundary of $M'_{\varphi}$ is $\sigma_0(\partial N_{K_0})=\psi(\{a_0\}\times \Lambda_{K_0})$.
As we have seen in the proof of Lemma \ref{lem-LO}, $M'_{\varphi}$ is irreducible.
We denote
\[\begin{array}{cc} \Gp\coloneqq \pi_1 (S^3\setminus N_{K_0}, y_0), &  \Gphip\coloneqq \pi_1 (M'_\varphi, \sigma_0(y_0)).\end{array}\]
Note that $\Gp$ and $\Gphip$ are naturally isomorphic to $\pi_1(\bR^3\setminus K_0, y_0)$ and $\pi_1(M_{\varphi},\sigma_0(y_0))$ respectively.
Moreover, we identify $G_0$ (resp. $G_1$) and $\Gp$ (resp. $\Gphip$) by the path $\gamma_0$ (resp. $\gamma_1$) in (\ref{paths-connecting-basept}) connecting the base points.
Then, the surjection $h$ is rewritten as
\[ h\colon \Gp  \twoheadrightarrow  \Gphip . \]
Lastly, let us denote
\[\begin{array}{ll}
\emphip \coloneqq (\sigma_0)_*(m_0), & \ellphip \coloneqq (\sigma_0)_*(l_0) \in \Gphip .
\end{array}\]
\end{notation}

Using the above notation, Proposition \ref{prop-GL_n-rep} about $h\colon \Gp \to G'_{\varphi}$ is restated as follows: For any representation $\rho \colon \Gphip\to \GL_n(\bC)$, there exists $A\in \GL_n(\bC)$ such that
\[ \begin{array}{cc}
A\cdot \rho (h(\emp)) \cdot A^{-1} = \rho(\emphip), & A\cdot \rho( h(\ellp)) \cdot A^{-1} = \rho(\ellphip) . \end{array} \]

The aim of this section is to prove the following theorem. 

\begin{theorem}\label{thm-ml-preserve}
    There exists an element $g\in \Gphip$ such that
    \[\begin{array}{cc}
    g\cdot h(\emp) \cdot g^{-1}=\emphip, & g \cdot h(\ellp) \cdot g^{-1}=\ellphip. \end{array}\] 
    In particular, there exists a surjective homomorphism $\tilde{h}\colon \Gp\to \Gphip$ with $\tilde{h}(\emp)=\emphip$ and $\tilde{h}(\ellp)=\ellphip$.
\end{theorem}

This is a consequence of Proposition \ref{prop-GL_n-rep} and results about the fundamental groups of $3$-manifolds.

The following statement is a direct consequence of Theorem \ref{thm-ml-preserve}.

\begin{corollary}\label{cor-peripheral-preserve}
    There exists a surjective group homomorphism
    \[\tilde{k} \colon \pi_1(\bR^3\setminus K_0) \to \pi_1(\bR^3 \setminus K_1)\]
    such that by changing the orientation of $K_1$ if necessary, $\tilde{k}(l_0)=l_1$ and $\tilde{k}(m_0)=m_1$.
\end{corollary}
\begin{proof}
    Take a surjection $\tilde{h}\colon \Gp\to \Gphip$ in Theorem \ref{thm-ml-preserve} and define $\tilde{k}\coloneqq \bar{r}_1 \circ \tilde{h}$, where $\bar{r}_1 \colon \Gphip \to \pi_1 (\bR^3\setminus K_1) $ is defined in Subsection \ref{subsec-epi-knots}. The equations $\tilde{k}(l_0)=l_1$ and $\tilde{k}(m_0)=m_1$ are obtained from the relation (\ref{r1-sigma0-rel}).
\end{proof}

\begin{remark}\label{rem-epi-peripheral}
Let $K$ and $K'$ be oriented knots in $S^3$ equipped with the Seifert framing.
Let $G_K \coloneqq \pi_1(S^3\setminus K,y)$, where $y\in \partial N_{K}$, and let $l$ and $m$ be the longitude and the meridian respectively.
In the same way, we put $l',m'\in G_{K'}$ for the knot $K'$.
In Remark \ref{rem-epi-knot}, we discussed a condition: $K\geq K'$ if there is an epimorphism $G_K\to G_{K'}$.
We can consider its four refinements:
    \begin{enumerate}
        \item $K \ge_m K'$ if there is an epimorphism $k\colon G_{K}\to G_{K'} $ with $k(m)=m'$.
        \item $K \ge_p K'$ if there is an epimorphism $k\colon G_{K}\to G_{K'} $ with $k (\pi_1 (\partial N_K) )\subset \pi_1 (\partial N_{K'})$.
        \item $K \ge_{ml} K' $ if there is an epimorphism $k\colon G_{K}\to G_{K'} $ with $k(m)=m'$ and $k(l)=l'$.
        \item $K \ge_1 K'$ if there is a degree $1$ map $(S^3\setminus N_K, \partial N_K)\to (S^3\setminus N_{K'}, \partial N_{K'})$.
    \end{enumerate}
    The implications between these conditions are 
    \[K \ge_1 K' \Leftrightarrow K \ge_{ml} \pm K' \Rightarrow K \ge_p K' \Rightarrow K \ge_m \pm K' \Rightarrow K \ge K'.\]
    They can be checked as follows (the authors could not find in literature a summary of the relations of these conditions):
    \begin{itemize}
\item     $K \ge_{ml} \pm K' \Rightarrow K \ge_p K' $ and $K \ge_m \pm K'\Rightarrow K \ge K'$ are trivial. 
   \item $K \ge_{ml} \pm K' \Rightarrow K\geq_1K'$:
   We use the fact that $S^3\setminus N_{K}$ and $S^3\setminus N_{K'}$ are aspherical.
   If $K \ge_{ml} \pm K' $, the epimorphism $k$ associates a map $F$ which makes the following diagram commute up to homotopy:
    \[\xymatrix@R=15pt{ S^3\setminus N_{K} =   K(G_K,1) \ar[r]^-{F} & K (G_{K'},1) = S^3\setminus N_{K'} \\
    \partial N_{K} = K(\bZ^2,1) \ar@<6.5ex>[u] \ar[r]^-{\sim} & K(\bZ^2,1) = \partial N_{K'}. \ar@<-5.5ex>[u] }\]
    Here, the vertical arrows are the inclusion maps.
    We can arrange $F$ by a homotopy on a collar neighborhood of $\partial N_K$ so that the diagram strictly commutes.
    The degree of $F$ is $1$,
    since $F$ induces an isomorphism
    \[ H_3(S^3\setminus N_K, \partial N_K) \cong H_2(\partial N_K) \to H_2(\partial N_{K'})\cong H_3(S^3\setminus N_{K'}, \partial N_{K'})\]
    which preserves the fundamental classes. Therefore, $K\geq_1 K'$ holds.
    \item $K \ge_p K' \Rightarrow K\ge_m\pm K'$: If $K \ge_p K'$, changing the orientation of $K'$ if necessary $k(m)=m'(l')^a$ for some $a\in \bZ$. Let $\la\la m'(l')^{a} \ra\ra\subset G_{K'}$ be the smallest normal subgroup containing $m'(l')^{ a}$, and define $\la\la m \ra\ra\subset G_K$ similarly. Then, $k$ induces a surjection
    \[\{1\}= G_K/\la\la  m \ra \ra \to G_{K'}/\la \la m'(l')^{ a} \ra \ra,\]
    so $G_{K'}/\la \la m'(l')^{ a} \ra \ra$ is trivial.
    Now, the Property P, proved by \cite{KM}, shows that $a=0$.
    \item $K \ge_1 K' \Rightarrow K \ge_{ml} \pm K'$: If $K\geq_1 K'$, changing the orientation of $K'$ if necessary, the degree $1$ map induces an epimorphism $k\colon G_K\to G_{K'}$ such that $k(m)=m'(l')^a$ for some $a\in \bZ$ and $k(l)=l'$.
    By using the Property P as above, we can show that $a=0$.
    \end{itemize}
    Finally, let us discuss whether the opposite implications hold:
    \begin{itemize}
    \item If $K'$ is prime, $K \ge_m \pm K'$ implies $K \ge_p K'$ \cite[Corollary 2.3]{SW2}. 
    For this implication, the primeness condition of $K'$ is needed:
     For any pair of knots $K_1, K_2$ and the mirror $\overline{K_2}$ of $K_2$, there is an isomorphism $G_{K_1\#K_2}\to G_{K_1\# (-\overline{K_2})}$ preserving the meridian.
     In particular, $K_1\# K_2\ge_m K_1\# (-\overline{K_2})$ holds.
     However, if there is an epimorphism $k\colon G_{K_1\#K_2}\to G_{K_1\# (-\overline{K_2})}$ preserving the peripheral subgroups, then $k$ must be an isomorphism from the Hopfian property of knot groups, and we obtain that $K_1\# K_2 $ is isotopic to $ K_1\# (-\overline{K_2})$ up to mirror by \cite[Proposition 3.2]{SW}. This cannot happen when, for example, $K_1=K_2=3_1$.

    \item Infinitely many pairs of prime knots $(K, K')$ with $K \ge K'$ and $K\not\ge_m K'$ are constructed in \cite[Theorem 1.1]{ChaSuzuki}. 

    \item Some examples of the pairs of knots $(K, K')$ with $K \ge_p K'$ and $K\not\ge_{ml} K'$ are known by \cite[Proposition B.3.]{BKN}.
    \end{itemize}
\end{remark}

\subsection{\texorpdfstring{Preliminaries for $3$-manifold groups}{Preliminaries for 3-manifold groups}}

In this subsection, we recall some results on the fundamental groups of $3$-manifolds, which we will use to prove Theorem \ref{thm-ml-preserve}. 

The following Theorem \ref{thm-conj-sep} is a highly non-trivial result, which is obtained after the works on the fundamental groups of the hyperbolic 3-manifolds by Wise and Agol. 

\begin{definition}
    A group $G$ is said to be \textit{conjugacy separable}, if for every two elements $g_1, g_2\in G$ which are not conjugate in $G$, there exist a finite group $H$ and a homomorphism $p\colon G\to H$ such that $p(g_1)$ and $p(g_2)$ are not conjugate in $H$. 
\end{definition}

\begin{theorem}[{\cite[Theorem 1.3.]{HWZ13}}]\label{thm-conj-sep}
    The fundamental group $\pi_1(M)$ of any compact orientable 3-manifold $M$ is conjugacy separable.  
\end{theorem}

The remaining results we need are relatively basic facts. 

\begin{lemma}\label{lem-solidtorus}
   Let $N$ be a compact irreducible $3$-manifold with boundary. 
   Let $T$ be a connected component of $\partial N$ and assume that $T$ is homeomorphic to the torus $S^1\times S^1$. 
   If $\pi_1 (T)\to \pi_1 (N)$ is not injective, then $N$ is homeomorphic to $S^1\times D^2$. 
\end{lemma}
For the proof, see \cite[Lemma 13]{HarpeWeber}, for example.

A subgroup $H$ of a group $G$ is said to be \textit{malnormal} if $H\cap (g H g^{-1})$ is trivial for any $g\in G\setminus H$.

\begin{proposition}[{\cite[Theorem 3.]{HarpeWeber}, \cite[Theorem 4.5.]{Friedl}}]\label{prop-boundary-hyperbolic}
    Let $N$ be a compact orientable irreducible $3$-manifold with toroidal boundary. 
    Let $T$ be a connected component of $\partial N$ and $N_0$ be the JSJ component of $N$ containing $T$. 
    If $N_0$ is hyperbolic, then $\pi_1(T)$ is a malnormal subgroup of $\pi_1 (N)$. 
\end{proposition}

From the JSJ decomposition $\{N_1,\dots ,N_r\}$ of a $3$-manifold $N$, a \textit{graph of groups} is defined as in \cite[Section 2]{Friedl} and \cite[Section 3.1]{Cerocchi}.
The underlying graph consists of the set of vertices $V=\{v_1,\dots ,v_r\}$ and the set of edges $E$ with a map $E\to V\times V,\  e\mapsto (s(e),t(e))$ such that for each JSJ-torus $S \subset \partial N_{k_1}\cap \partial N_{k_2}$, an edge $e_S\in E$ and its inverse $\bar{e}_S$ with $s(e_S) = v_{k_1}$ and $t(e_S)=v_{k_2}$ are assigned.
Additionally, we associate the group $\pi_1(N_k)$ to each vertex $v_k$ and the group $\pi_1(S)$ to $e_S$ and $\bar{e}_S$, together with homomorphisms $\pi_1(S) \to \pi_1(N_{k_i})$ for $i=1,2$ induced by the inclusion maps when $s(e_S)=v_{k_1}$ and $t(e_S)=v_{k_2}$.

Each element of $\pi_1(N)$ is presented by a `loop' in the graph of groups \cite[Section 2.2]{Friedl}.
For the proof of Proposition \ref{prop-boundary-hyperbolic}, Friedl studied loops for $g$ and $gxg^{-1}$, where $g\in \pi_1(N)\setminus \pi_1(T)$ and $x\in \pi_1(T)$.

Similar to Proposition \ref{prop-boundary-hyperbolic}, the following result for the case where $N_0$ is Seifert fibered holds.

\begin{proposition}[{\cite[Theorem 1.1.(i)]{Cerocchi}}]\label{prop-boundary-Seifert}
    Let $N$ be a compact orientable irreducible $3$-manifold with troidal boundary. 
    Assume that $N$ is not homeomorphic to $S^1\times D^2$ or a $[0,1]$-bundle over the torus or the Klein bottle.
    Let $T$ be a connected component of $\partial N$ and $N_0$ be the JSJ component of $N$ containing $T$. 
    Assume that the $N_0$ is Seifert-fibered. 
    Then the Seifert fibered structure of $N_0$ is unique (see Remark \ref{rem-Seifert-unique} below).
    
    Let $c\in \pi_1 (N_0)$ be the element representing the fiber of the Seifert fibration of $N_0$ (which is well-defined up to inversion $c\leftrightarrow c^{-1}$). 
    Let $g\in \pi_1(N)$. 
    \begin{enumerate}
        \item If $g\in \pi_1 (N_0)\setminus \pi_1 (T)$, then $\pi_1 (T)\cap \left( g \cdot \pi_1 (T) \cdot g^{-1} \right) =\langle c\rangle$.
        \item If $g\in \pi_1(N)\setminus \pi_1 (N_0)$, then $\pi_1 (T)\cap \left( g\cdot \pi_1 (T)\cdot g^{-1}\right)$ is trivial. 
    \end{enumerate}
\end{proposition}

For the proof, Cerocchi used the Bass-Serre tree, on which $\pi_1(N)$ acts, associated to the peripheral extension of the graph of groups \cite[Section 3.1]{Cerocchi}.
Take $h\in \pi_1(T) \cap \left( g\cdot \pi_1(T)\cdot g^{-1}\right) $.
Proposition \ref{prop-boundary-Seifert} is proved by studying geodesics connecting two leaves (cosets) $1\cdot \pi_1(T), g\cdot \pi_1(T) \in \pi_1(N)/\pi_1(T)$ in a subtree fixed by the action of $h$.
See \cite[Proposition 3.6]{Cerocchi} for details.

\begin{remark}\label{rem-Seifert-unique}
Note that when the total space is orientable, a $[0,1]$-bundle over the torus is homeomorphic to $S^1\times S^1 \times [0,1]$, and a $[0,1]$-bundle over the Klein bottle is a compact $3$-manifold with a single boundary component.
From the definition of $N_0$, if $N$ is not homeomorphic to $S^1\times D^2$ or a $[0,1]$-bundle over the torus or the Klein bottle, then $N_0$ is homeomorphic to none of them.
Moreover, in this case, a Seifert fibered structure of $N_0$ (if it exists) is unique up to isotopy. See \cite[Theorem 2.3]{hatcher}.
Let us also note that $H_2(N;\bZ/2\bZ)\cong \bZ/2\bZ$ if $N$ is a $[0,1]$-bundle over the torus or the Klein bottle. 
\end{remark}

\subsection{Proof of Theorem \ref{thm-ml-preserve}}

\begin{lemma}\label{lem-conj-single}
    For any $f \in \pi_1 (\partial N_{K_0}, y_0 )$, 
    $h(f)$ is conjugate to $(\sigma_0)_*(f)$ in $\Gphip$. 
\end{lemma}
\begin{proof}
    Fix any $f \in \pi_1 (\partial N_{K_0}, y_0 )$  and put $f_\varphi \coloneqq (\sigma_0)_*(f)\in \Gphip$.
    By Proposition \ref{prop-GL_n-rep}, for any representation $\rho \colon \Gphip \to \GL_n(\bC )$, 
    $\rho (h(f))$ is conjugate to $\rho (f_\varphi)$ in $\GL_n (\bC)$, 
    and hence $\tr (\rho (h(f)))=\tr (\rho (f_\varphi))$. 

    Choose a finite group $H$ and a morphism $p\colon G\to H$ arbitrarily. 
    Let $\mathrm{Irr}_H =\{ \rho_1,\dots \rho_k\}$ be the set of the finite dimensional irreducible complex representations of $H$. 
    For any $\rho_i\in \mathrm{Irr}_H$, we have $\tr (\rho_i \circ p (h(f)))=\tr (\rho_i \circ p (f_\varphi))$.  
    Since the character functions $\{ \tr (\rho_i(\mathchar`- )) \}_i$ form a basis for the vector space of the class functions on $H$, it follows that $p (h(f))$ and $p (f_\varphi)$ are conjugate in $H$. 
    
    Since $\Gphip$ is conjugacy separable by Theorem \ref{thm-conj-sep}, we can conclude that
    $h(f)$ and $f_\varphi$ are conjugate in $\Gphip$. 
\end{proof}

Lemma \ref{lem-conj-single} implies that for any $f \in \pi_1 (\partial N_{K_0}, y_0 )$, there is an element $ g_f\in \Gphip$ satisfying $g_f \cdot h(f) \cdot g_f^{-1}=(\sigma_0)_*(f)$.
For each pair $(c,d)\in \bZ^2$, consider the case $f=m_0^cl_0^d\in \pi_1 (\partial N_{K_0}, y_0 )$ and fix an element $g_{c,d}\in \Gphip$ satisfying $g_{c,d}\cdot  h(\emp^c\ellp^d) \cdot g_{c,d}^{-1}=\emphip^c\ellphip^d$.
Then, we define
\[
    h_{c,d}\colon \Gp\to \Gphip ,\  f \mapsto g_{c,d} \cdot  h(f) \cdot g_{c,d}^{-1}.
\]
By definition, $h_{c,d}$ is a surjection satisfying $h_{c,d} (\emp^c\ellp^d)=\emphip^c\ellphip^d$. 

The existence of $h_{1,0}$ gives a constraint on the topology of $M_\varphi$. (This observation is not needed for the proof of Theorem \ref{thm-ml-preserve}.)
\begin{lemma}
    There exists a knot $K_\varphi\subset \bR^3$ such that $M_\varphi$ is homeomorphic to $\bR^3\setminus K_\varphi$. 
    In particular, $\Gphip$ is isomorphic to a knot group. 
\end{lemma}
\begin{proof}
    Firstly, note the following basic fact: 
    Let $N$ be a compact $3$-manifold and $T\subset \partial N$ be a troidal boundary component of $N$. 
    Choose a base point $y\in T$ and a simple closed curve $\gamma\colon [0,1]\to T$ satisfying $\gamma(0)=\gamma (1)=y$.
    Let $\gamma$ also denote the elements of $\pi_1(T,y)$ and $\pi_1(N,y)$ represented by the path $\gamma$. 
    Let $N_\gamma$ be the $3$-manifold $N\cup_f (S^1\times D^2)$ obtained by gluing $N$ and $S^1\times D^2$ by a homeomorphism $f\colon T\to \partial (S^1\times D^2)$ such that $f(\gamma)$ bounds a disk in $S^1\times D^2$. 
    By Seifert--van Kampen theorem, $\pi_1(N_\gamma ,y)$ is naturally isomorphic to $\pi_1 (N,y)/\langle\langle \gamma \rangle\rangle$, where $\langle\langle \gamma \rangle\rangle$ is the smallest normal subgroup of $\pi_1 (N,y)$ that contains $\gamma$. 

    The surjection $h_{1,0}\colon \Gp\to \Gphip$ induces a surjection $\Gp/\langle\langle \emp \rangle\rangle\to \Gphip/\langle\langle \emphip \rangle\rangle$. 
    By the above observation, $\Gp/\langle\langle \emp\rangle\rangle = \pi_1 ( \left(S^3\setminus N_{K_1}\right) \cup N_{K_1}) = \pi_1(S^3)$ is trivial, and
    thus $\Gphip/\langle\langle \emphip \rangle\rangle$ is also a trivial group. 
    Again by the above observation, $\Gphip/\langle\langle \emphip \rangle\rangle$ is isomorphic to the fundamental group of
    \[(M_\varphi')_{\emphip}\coloneqq M_\varphi'\cup_f (S^1\times D^2),\]
    where $f\colon \sigma_0(\Lambda_{K_0})\to \partial (S^1\times D^2)$ is a homeomorphism satisfying $f_*(\emphip)= 1 \in \pi_1 (S^1\times D^2)$. 
    The closed $3$-manifold $(M_\varphi')_{\emphip}$ is homeomorphic to $S^3$ by the affirmative resolution of the Poincar\'e conjecture. Through a homeomorphism $(M_\varphi')_{\emphip}\cong S^3$, $S^1\times D^2\subset (M_\varphi')_{\emphip}$ corresponds to a tubular neighborhood of some knot $K_\varphi$ in $S^3$. 
    Hence $M_\varphi$ is homeomorphic to the knot complement $\bR^3\setminus K_\varphi$.
\end{proof}

\begin{lemma}
    If $(\sigma_0)_*\colon \pi_1(\Lambda_{K_0},y_0)\to \Gphip$ is not injective, $M'_{\varphi}$ is homeomorphic to $S^1\times D^2$.
    Moreover, $h(\ellp)=\ellphip$ and $h(\emp)=\emphip$.
\end{lemma}

\begin{proof}
Suppose that  $(\sigma_0)_*$ is not injective.
By Lemma \ref{lem-solidtorus}, $M'_{\varphi} \cong S^1\times D^2$. Hence $\Gphip \cong\bZ$.
Since $\Gphip$ is abelian, it follows from Lemma \ref{lem-conj-single} that $h(\ellp)=\ellphip$ and $h(\emp)=\emphip$.
\end{proof}

From now on, we assume that $(\sigma_0)_*\colon \pi_1(\partial N_{K_0},y_0)\to \Gphip$ is injective. 
Denote its image by
\[H \coloneqq  (\sigma_0)_*(\pi_1(\partial N_{K_0},y_0))\subset \Gphip.\] 

Consider the JSJ decomposition of $M_\varphi'$ and let $M_0\subset M_\varphi'$ be the JSJ component containing $\partial M'_{\varphi} =\sigma_0 (\partial N_{K_0})$. 
Note that both $H\to \pi_1(M_0,\sigma_0(y_0))$ and $\pi_1(M_0,\sigma_0(y_0))\to \Gphip$ are injective. 
We divide into cases according to whether $M_0$ is hyperbolic or Seifert-fibered. 

Let us discuss the case where $M_0$ is hyperbolic.
Note that by definition, $h_{1,0}(\emp)=\emphip$.

\begin{lemma}\label{lem-hyp-ml}
    If $M_0$ is hyperbolic, $h_{1,0}(\ellp)=\ellphip$. 
\end{lemma}
\begin{proof}
    By $h_{1,0}(\emp)=\emphip$, we have
    \[h_{1,0}(\ellp) \cdot \emphip \cdot  h_{1,0}(\ellp)^{-1}=h_{1,0}(\ellp\emp\ellp^{-1})=h_{1,0}(\emp)=\emphip .\]
    This means that $H\cap \left( h_{1,0}(\ellp) \cdot H \cdot h_{1,0}(\ellp)^{-1}\right) \neq \{1\}$.
    Since $H=(\sigma_0)_*(\pi_1(\partial N_{K_0},y_0))$ is a malnormal subgroup of $\Gphip$ by Proposition \ref{prop-boundary-hyperbolic},
    $h_{1,0}(\ellp)\in H$. 
    
    By the definition of $h_{1,0}$, $h_{1,0}(\ellp)$ is conjugate to $h(\ellp)$ in $\Gphip$. 
    In addition, by Lemma \ref{lem-conj-single}, $h(\ellp)$ is conjugate to $\ellphip$ in $\Gphip$. 
    Hence $h_{1,0}(\ellp)$ and $\ellphip$ are conjugate in $\Gphip$.
    Let us choose an element $g \in \Gphip$ satisfying $g \cdot h_{1,0}(\ellp) \cdot g^{-1} =\ellphip$. 
    Since $\ellphip\in H\cap g H g^{-1}$, again by the malnormality of $H$ in $\Gphip$, it follows that $g \in H$. 
    Since $H$ is abelian, $\ellphip=g \cdot h_{1,0}(\ellp) \cdot g^{-1}= h_{1,0}(\ellp)$. 
\end{proof}

It remains to discuss the case where $M_0$ is Seifert-fibered. 
Note that $M_\varphi'$ with the toroidal boundary $\sigma_0(\partial N_{K_0})$ satisfies the assumption of Proposition \ref{prop-boundary-Seifert} since $H\to \Gphip$ is injective and $H_2(M_\varphi';\bZ/2\bZ)=0$.
(Indeed, $H^1(M'_{\varphi}, \partial M'_{\varphi};\bZ/2\bZ)=0$ since the map $\bZ/2\bZ \cong H^1(M'_{\varphi};\bZ/2\bZ) \to H^1(\partial M'_{\varphi} ;\bZ/2\bZ) \cong (\bZ/2\bZ)^{\oplus 2}$ induced by the inclusion is injective.)

Let $c\in H\subset \pi_1(M_0,\sigma_0(y_0))$ be the element represented by the regular fiber of the Seifert fibration of $M_0$.
Since $c$ is a primitive element of $H \cong\bZ^2$, there is at most one $d'\in \bZ$ such that $\emphip\ellphip^{d'} \in \langle c \rangle$. 
Therefore, we can choose $d\in \bZ$ so that $\emphip\ellphip^d, \emphip\ellphip^{d+1} \notin \langle c \rangle $. 

\begin{lemma}\label{lem-Seifert-ml}
    $h_{1,d}(\emp)=\emphip$ and $h_{1,d}(\ellp)=\ellphip$.
\end{lemma}
\begin{proof}
    By the definition of $h_{1,d}$, $h_{1,d}(\emp\ellp^d)=\emphip\ellphip^d$. 
    Since $h_{1,d}$ is a group homomorphism, it suffices to show $h_{1,d}(\emp\ellp^{d+1})=\emphip\ellphip^{d+1}$. 
    Since $H$ is abelian, 
    \begin{align*} h_{1,d}(\emp\ellp^{d+1}) \cdot \emphip\ellphip^d \cdot h_{1,d}(\emp\ellp^{d+1})^{-1} & =h_{1,d}(\emp\ellp^{d+1}\emp\ellp^d(\emp\ellp^{d+1})^{-1}) \\
    & =h_{1,d}(\emp\ellp^d) \\ &=\emphip\ellphip^d.
    \end{align*}
    Hence $\emphip\ellphip^d\in H\cap \left( h_{1,d}(\emp\ellp^{d+1}) \cdot H \cdot h_{1,d}(\emp\ellp^{d+1})^{-1}\right)$ and this implies 
    \[H\cap \left( h_{1,d}(\emp\ellp^{d+1}) \cdot H \cdot  h_{1,d}(\emp\ellp^{d+1})^{-1} \right) \not\subset \langle c \rangle  \]
    by the choice of $d$. 
    Hence, $h_{1,d}(\emp\ellp^{d+1})\in H$ by Proposition \ref{prop-boundary-Seifert}.  
    By the definition of $h_{1,d}$ and Lemma \ref{lem-conj-single}, $h_{1,d}(\emp\ellp^{d+1})$ and $\emphip\ellphip^{d+1}$ are conjugate in $\Gphip$. 
    Let us choose an element $g \in \Gphip$ satisfying $g \cdot h_{1,d}(\emp\ellp^{d+1}) \cdot g^{-1}=\emphip\ellphip^{d+1}$. 
    Then, $\emphip\ellphip^{d+1}\in H\cap gHg^{-1}$ holds. This implies $H\cap gHg^{-1} \not\subset \langle c \rangle$ again by the choice of $d$.  
    Hence $g$ is an element of $H$ by Proposition \ref{prop-boundary-Seifert}. 
    Since $H$ is abelian, $ h_{1,d}(\emp\ellp^{d+1})=g \cdot h_{1,d}(\emp\ellp^{d+1}) \cdot g^{-1}=\emphip\ellphip^{d+1}$.
\end{proof}

\begin{proof}[Proof of  Theorem \ref{thm-ml-preserve}]
The above three lemmata prove that Theorem \ref{thm-ml-preserve} holds if we take $g\in G'_{\varphi}$ as follows:
If $(\sigma_0)_*\colon \pi_1(\Lambda_{K_0})\to G'_{\varphi}$ is not injective, we take $g=1\in G'_{\varphi}\cong\bZ$.
If $(\sigma_0)_*$ is injective and $N_0$ is hyperbolic, we take $g=g_{1,0}$.
If $(\sigma_0)_*$ is injective and $N_0$ is Seifert fibered, we take $g=g_{1,d}$, where $d\in \bZ$ is chosen as above.
\end{proof}

\printbibliography

@article {GKS,
    AUTHOR = {Guillermou, St{\'e}phane and Kashiwara, Masaki and Schapira,
              Pierre},
     TITLE = {Sheaf quantization of {H}amiltonian isotopies and applications
              to nondisplaceability problems},
   JOURNAL = {Duke Math. J.},
  FJOURNAL = {Duke Mathematical Journal},
    VOLUME = {161},
      YEAR = {2012},
    NUMBER = {2},
     PAGES = {201--245},
      ISSN = {0012-7094},
     CODEN = {DUMJAO},
   MRCLASS = {53D35 (18F30)},
  MRNUMBER = {2876930},
MRREVIEWER = {Corrado Marastoni},
       DOI = {10.1215/00127094-1507367},
}

@book {KS90,
    AUTHOR = {Kashiwara, Masaki and Schapira, Pierre},
     TITLE = {Sheaves on manifolds},
    SERIES = {Grundlehren der Mathematischen Wissenschaften},
    VOLUME = {292},
 PUBLISHER = {Springer-Verlag, Berlin},
      YEAR = {1990},
     PAGES = {x+512},
      ISBN = {3-540-51861-4},
   MRCLASS = {58G07 (18F20 32C38 35A27)},
  MRNUMBER = {1299726 (95g:58222)},
}

@article {Gui23,
    AUTHOR = {Guillermou, St\'{e}phane},
     TITLE = {Sheaves and symplectic geometry of cotangent bundles},
   JOURNAL = {Ast\'{e}risque},
  FJOURNAL = {Ast\'{e}risque},
    NUMBER = {440},
      YEAR = {2023},
     PAGES = {x+274},
      ISSN = {0303-1179,2492-5926},
      ISBN = {978-2-85629-972-2},
   MRCLASS = {53D12 (18F20 35A27)},
  MRNUMBER = {4612528},
       DOI = {10.24033/ast.1199},
}

@article {NRSSZ15,
    AUTHOR = {Ng, Lenhard and Rutherford, Dan and Shende, Vivek and Sivek,
              Steven and Zaslow, Eric},
     TITLE = {Augmentations are sheaves},
   JOURNAL = {Geom. Topol.},
  FJOURNAL = {Geometry \& Topology},
    VOLUME = {24},
      YEAR = {2020},
    NUMBER = {5},
     PAGES = {2149--2286},
      ISSN = {1465-3060,1364-0380},
   MRCLASS = {53D42 (53D37)},
  MRNUMBER = {4194293},
MRREVIEWER = {Alexander\ Fel\cprime shtyn},
       DOI = {10.2140/gt.2020.24.2149},
       URL = {https://doi.org/10.2140/gt.2020.24.2149},
}

@article {robalo2018lemma,
    AUTHOR = {Robalo, Marco and Schapira, Pierre},
     TITLE = {A lemma for microlocal sheaf theory in the
              {$\infty$}-categorical setting},
   JOURNAL = {Publ. Res. Inst. Math. Sci.},
  FJOURNAL = {Publications of the Research Institute for Mathematical
              Sciences},
    VOLUME = {54},
      YEAR = {2018},
    NUMBER = {2},
     PAGES = {379--391},
      ISSN = {0034-5318,1663-4926},
   MRCLASS = {55U35 (18F99 32C38 35A27 55P42)},
  MRNUMBER = {3784874},
MRREVIEWER = {Birgit\ Richter},
       DOI = {10.4171/PRIMS/54-2-5},
       URL = {https://doi.org/10.4171/PRIMS/54-2-5},
}

@misc{LurieHA,
    title={Higher Algebra},
    author={Lurie, Jacob},
    URL = {https://www.math.ias.edu/~lurie/papers/HA.pdf},
    options={url=true}
}

@misc{NadlerShende20,
      title={Sheaf quantization in Weinstein symplectic manifolds}, 
      author={David Nadler and Vivek Shende},
      year={2020},
      eprint={2007.10154},
      archivePrefix={arXiv},
      primaryClass={math.SG},
}

@article {GPS24,
    AUTHOR = {Ganatra, Sheel and Pardon, John and Shende, Vivek},
     TITLE = {Microlocal {M}orse theory of wrapped {F}ukaya categories},
   JOURNAL = {Ann. of Math. (2)},
  FJOURNAL = {Annals of Mathematics. Second Series},
    VOLUME = {199},
      YEAR = {2024},
    NUMBER = {3},
     PAGES = {943--1042},
      ISSN = {0003-486X,1939-8980},
   MRCLASS = {53D37 (14J33 53D40)},
  MRNUMBER = {4740209},
       DOI = {10.4007/annals.2024.199.3.1},
}

@misc{O25,
      title={Topological constraints on clean Lagrangian intersections from $\mathbb{Q}$-valued augmentations}, 
      author={Yukihiro Okamoto},
      year={2025},
      eprint={2505.00330},
      archivePrefix={arXiv},
      primaryClass={math.SG},
      url={https://arxiv.org/abs/2505.00330}, 
}

@article {BRW,
    AUTHOR = {Boyer, Steven and Rolfsen, Dale and Wiest, Bert},
     TITLE = {Orderable 3-manifold groups},
   JOURNAL = {Ann. Inst. Fourier (Grenoble)},
  FJOURNAL = {Universit\'e{} de Grenoble. Annales de l'Institut Fourier},
    VOLUME = {55},
      YEAR = {2005},
    NUMBER = {1},
     PAGES = {243--288},
      ISSN = {0373-0956,1777-5310},
   MRCLASS = {57M05 (06F15 20F34 20F60 57M50)},
  MRNUMBER = {2141698},
MRREVIEWER = {Stephen\ P.\ Humphries},
       DOI = {10.5802/aif.2098},
       URL = {https://doi.org/10.5802/aif.2098},
}

@article {Li,
    AUTHOR = {Li, Wenyuan},
     TITLE = {Lagrangian cobordism functor in microlocal sheaf theory {I}},
   JOURNAL = {J. Topol.},
  FJOURNAL = {Journal of Topology},
    VOLUME = {16},
      YEAR = {2023},
    NUMBER = {3},
     PAGES = {1113--1166},
      ISSN = {1753-8416,1753-8424},
   MRCLASS = {53D42 (53D10 53D35)},
  MRNUMBER = {4638002},
MRREVIEWER = {Johan\ Asplund},
       DOI = {10.1112/topo.12310},
       URL = {https://doi.org/10.1112/topo.12310},
}

@article {Gab,
    AUTHOR = {Gabai, David},
     TITLE = {Surgery on knots in solid tori},
   JOURNAL = {Topology},
  FJOURNAL = {Topology. An International Journal of Mathematics},
    VOLUME = {28},
      YEAR = {1989},
    NUMBER = {1},
     PAGES = {1--6},
      ISSN = {0040-9383},
   MRCLASS = {57M25 (57R30)},
  MRNUMBER = {991095},
MRREVIEWER = {Cameron\ McA.\ Gordon},
       DOI = {10.1016/0040-9383(89)90028-1},
       URL = {https://doi.org/10.1016/0040-9383(89)90028-1},
}

@article {NiZ,
    AUTHOR = {Ni, Yi and Zhang, Xingru},
     TITLE = {Detection of knots and a cabling formula for
              {$A$}-polynomials},
   JOURNAL = {Algebr. Geom. Topol.},
  FJOURNAL = {Algebraic \& Geometric Topology},
    VOLUME = {17},
      YEAR = {2017},
    NUMBER = {1},
     PAGES = {65--109},
      ISSN = {1472-2747,1472-2739},
   MRCLASS = {57M25},
  MRNUMBER = {3604373},
MRREVIEWER = {Sangbum\ Cho},
       DOI = {10.2140/agt.2017.17.65},
       URL = {https://doi.org/10.2140/agt.2017.17.65},
}

@article {BS23,
    AUTHOR = {Baldwin, John A. and Sivek, Steven},
     TITLE = {Instantons and {L}-space surgeries},
   JOURNAL = {J. Eur. Math. Soc. (JEMS)},
  FJOURNAL = {Journal of the European Mathematical Society (JEMS)},
    VOLUME = {25},
      YEAR = {2023},
    NUMBER = {10},
     PAGES = {4033--4122},
      ISSN = {1435-9855,1435-9863},
   MRCLASS = {57R58 (57K10 57K31)},
  MRNUMBER = {4634689},
MRREVIEWER = {Nikolai\ N.\ Saveliev},
       DOI = {10.4171/jems/1280},
       URL = {https://doi.org/10.4171/jems/1280},
}

@misc{O23,
      title={On knot types of clean Lagrangian intersections in $T^*\mathbb{R}^3$}, 
      author={Yukihiro Okamoto},
      year={2025},
      eprint={2305.09912},
      archivePrefix={arXiv},
      primaryClass={math.SG},
      url={https://arxiv.org/abs/2305.09912}, 
}

@article {DG,
    AUTHOR = {Dunfield, Nathan M. and Garoufalidis, Stavros},
     TITLE = {Non-triviality of the {$A$}-polynomial for knots in {$S^3$}},
   JOURNAL = {Algebr. Geom. Topol.},
  FJOURNAL = {Algebraic \& Geometric Topology},
    VOLUME = {4},
      YEAR = {2004},
     PAGES = {1145--1153},
      ISSN = {1472-2747,1472-2739},
   MRCLASS = {57M25 (57M27 57M50)},
  MRNUMBER = {2113900},
MRREVIEWER = {Darren\ D.\ Long},
       DOI = {10.2140/agt.2004.4.1145},
       URL = {https://doi.org/10.2140/agt.2004.4.1145},
}

@article {CS,
    AUTHOR = {Culler, Marc and Shalen, Peter B.},
     TITLE = {Varieties of group representations and splittings of
              {$3$}-manifolds},
   JOURNAL = {Ann. of Math. (2)},
  FJOURNAL = {Annals of Mathematics. Second Series},
    VOLUME = {117},
      YEAR = {1983},
    NUMBER = {1},
     PAGES = {109--146},
      ISSN = {0003-486X},
   MRCLASS = {57N10},
  MRNUMBER = {683804},
MRREVIEWER = {G.\ Peter\ Scott},
       DOI = {10.2307/2006973},
       URL = {https://doi.org/10.2307/2006973},
}

@article {CCGLS,
    AUTHOR = {Cooper, D. and Culler, M. and Gillet, H. and Long, D. D. and
              Shalen, P. B.},
     TITLE = {Plane curves associated to character varieties of
              {$3$}-manifolds},
   JOURNAL = {Invent. Math.},
  FJOURNAL = {Inventiones Mathematicae},
    VOLUME = {118},
      YEAR = {1994},
    NUMBER = {1},
     PAGES = {47--84},
      ISSN = {0020-9910,1432-1297},
   MRCLASS = {57N10 (57M25)},
  MRNUMBER = {1288467},
MRREVIEWER = {Serge\ L.\ Tabachnikov},
       DOI = {10.1007/BF01231526},
       URL = {https://doi.org/10.1007/BF01231526},
}

@article {KSW,
    AUTHOR = {Kitano, Teruaki and Suzuki, Masaaki and Wada, Masaaki},
     TITLE = {Twisted {A}lexander polynomials and surjectivity of a group
              homomorphism},
   JOURNAL = {Algebr. Geom. Topol.},
  FJOURNAL = {Algebraic \& Geometric Topology},
    VOLUME = {5},
      YEAR = {2005},
     PAGES = {1315--1324},
      ISSN = {1472-2747,1472-2739},
   MRCLASS = {57M25 (57M05)},
  MRNUMBER = {2171811},
MRREVIEWER = {Mark\ Brittenham},
       DOI = {10.2140/agt.2005.5.1315},
       URL = {https://doi.org/10.2140/agt.2005.5.1315},
}

@article {AgLiu,
    AUTHOR = {Agol, Ian and Liu, Yi},
     TITLE = {Presentation length and {S}imon's conjecture},
   JOURNAL = {J. Amer. Math. Soc.},
  FJOURNAL = {Journal of the American Mathematical Society},
    VOLUME = {25},
      YEAR = {2012},
    NUMBER = {1},
     PAGES = {151--187},
      ISSN = {0894-0347,1088-6834},
   MRCLASS = {57M25 (57N10)},
  MRNUMBER = {2833481},
MRREVIEWER = {Bruno\ P.\ Zimmermann},
       DOI = {10.1090/S0894-0347-2011-00711-X},
       URL = {https://doi.org/10.1090/S0894-0347-2011-00711-X},
}

@incollection {ORS,
    AUTHOR = {Ohtsuki, Tomotada and Riley, Robert and Sakuma, Makoto},
     TITLE = {Epimorphisms between 2-bridge link groups},
 BOOKTITLE = {The {Z}ieschang {G}edenkschrift},
    SERIES = {Geom. Topol. Monogr.},
    VOLUME = {14},
     PAGES = {417--450},
 PUBLISHER = {Geom. Topol. Publ., Coventry},
      YEAR = {2008},
   MRCLASS = {57M25 (57M05 57M50)},
  MRNUMBER = {2484712},
MRREVIEWER = {Se-Goo\ Kim},
       DOI = {10.2140/gtm.2008.14.417},
       URL = {https://doi.org/10.2140/gtm.2008.14.417},
}

@article {ALSS,
    AUTHOR = {Aimi, Shunsuke and Lee, Donghi and Sakai, Shunsuke and Sakuma,
              Makoto},
     TITLE = {Classification of parabolic generating pairs of {K}leinian
              groups with two parabolic generators},
   JOURNAL = {Rend. Istit. Mat. Univ. Trieste},
  FJOURNAL = {Rendiconti dell'Istituto di Matematica dell'Universit\`a{} di
              Trieste. An International Journal of Mathematics},
    VOLUME = {52},
      YEAR = {2020},
     PAGES = {477--511},
      ISSN = {0049-4704,2464-8728},
   MRCLASS = {57M50 (20F36 57K10)},
  MRNUMBER = {4207648},
MRREVIEWER = {Jingyin\ Huang},
       DOI = {10.1007/s11004-019-09820-w},
       URL = {https://doi.org/10.1007/s11004-019-09820-w},
}

@article {SW,
    AUTHOR = {Silver, Daniel S. and Whitten, Wilbur},
     TITLE = {Knot group epimorphisms},
   JOURNAL = {J. Knot Theory Ramifications},
  FJOURNAL = {Journal of Knot Theory and its Ramifications},
    VOLUME = {15},
      YEAR = {2006},
    NUMBER = {2},
     PAGES = {153--166},
      ISSN = {0218-2165,1793-6527},
   MRCLASS = {57M25 (37B10)},
  MRNUMBER = {2207903},
MRREVIEWER = {John\ G.\ Ratcliffe},
       DOI = {10.1142/S0218216506004373},
       URL = {https://doi.org/10.1142/S0218216506004373},
}

@misc{SW2,
      title={Knot Group Epimorphisms, II}, 
      author={Daniel S. Silver and Wilbur Whitten},
      year={2008},
      eprint={0806.3223},
      archivePrefix={arXiv},
      primaryClass={math.GT},
      url={https://arxiv.org/abs/0806.3223}, 
}

@article {BZ,
    AUTHOR = {Burde, Gerhard and Zieschang, Heiner},
     TITLE = {Eine {K}ennzeichnung der {T}orusknoten},
   JOURNAL = {Math. Ann.},
  FJOURNAL = {Mathematische Annalen},
    VOLUME = {167},
      YEAR = {1966},
     PAGES = {169--176},
      ISSN = {0025-5831,1432-1807},
   MRCLASS = {55.20},
  MRNUMBER = {210113},
MRREVIEWER = {R.\ Bott},
       DOI = {10.1007/BF01362170},
       URL = {https://doi.org/10.1007/BF01362170},
}

@book {BZH,
    AUTHOR = {Burde, Gerhard and Zieschang, Heiner and Heusener, Michael},
     TITLE = {Knots},
    SERIES = {De Gruyter Studies in Mathematics},
    VOLUME = {5},
   EDITION = {extended},
 PUBLISHER = {De Gruyter, Berlin},
      YEAR = {2014},
     PAGES = {xiv+417},
      ISBN = {978-3-11-027074-7},
      note = {eISBN:978-3-11-027078-5},
   MRCLASS = {57-01 (57M25)},
  MRNUMBER = {3156509},
MRREVIEWER = {Swatee\ Naik},
}

@incollection {Hem,
    AUTHOR = {Hempel, John},
     TITLE = {Residual finiteness for {$3$}-manifolds},
 BOOKTITLE = {Combinatorial group theory and topology ({A}lta, {U}tah,
              1984)},
    SERIES = {Ann. of Math. Stud.},
    VOLUME = {111},
     PAGES = {379--396},
 PUBLISHER = {Princeton Univ. Press, Princeton, NJ},
      YEAR = {1987},
      ISBN = {0-691-08409-2},
      note = {pbk ISBN:0-691-08410-6}, 
   MRCLASS = {57M05 (20E26 20F34 57N10)},
  MRNUMBER = {895623},
}

@article {RZ,
    AUTHOR = {Rolfsen, Dale and Zhu, Jun},
     TITLE = {Braids, orderings and zero divisors},
   JOURNAL = {J. Knot Theory Ramifications},
  FJOURNAL = {Journal of Knot Theory and its Ramifications},
    VOLUME = {7},
      YEAR = {1998},
    NUMBER = {6},
     PAGES = {837--841},
      ISSN = {0218-2165,1793-6527},
   MRCLASS = {20F36 (16S34 57M25)},
  MRNUMBER = {1643939},
MRREVIEWER = {Vitaly\ A.\ Roman\cprime kov},
       DOI = {10.1142/S0218216598000425},
       URL = {https://doi.org/10.1142/S0218216598000425},
}

@article {HS,
    AUTHOR = {Howie, James and Short, Hamish},
     TITLE = {The band-sum problem},
   JOURNAL = {J. London Math. Soc. (2)},
  FJOURNAL = {Journal of the London Mathematical Society. Second Series},
    VOLUME = {31},
      YEAR = {1985},
    NUMBER = {3},
     PAGES = {571--576},
      ISSN = {0024-6107,1469-7750},
   MRCLASS = {57M25},
  MRNUMBER = {812788},
MRREVIEWER = {Martin\ Scharlemann},
       DOI = {10.1112/jlms/s2-31.3.571},
       URL = {https://doi.org/10.1112/jlms/s2-31.3.571},
}

@article {GP,
    AUTHOR = {Ganatra, Sheel and Pomerleano, Daniel},
     TITLE = {A log {PSS} morphism with applications to {L}agrangian
              embeddings},
   JOURNAL = {J. Topol.},
  FJOURNAL = {Journal of Topology},
    VOLUME = {14},
      YEAR = {2021},
    NUMBER = {1},
     PAGES = {291--368},
      ISSN = {1753-8416,1753-8424},
   MRCLASS = {53D40 (14N35 53D12)},
  MRNUMBER = {4235013},
MRREVIEWER = {Johan\ Asplund},
       DOI = {10.1112/topo.12183},
       URL = {https://doi.org/10.1112/topo.12183},
}

@article {AL,
    AUTHOR = {Asplund, Johan and Li, Yin},
     TITLE = {Persistence of unknottedness of clean {L}agrangian
              intersections},
   JOURNAL = {J. Topol.},
  FJOURNAL = {Journal of Topology},
    VOLUME = {18},
      YEAR = {2025},
    NUMBER = {4},
     PAGES = {Paper No. e70053},
      ISSN = {1753-8416,1753-8424},
   MRCLASS = {57R17 (53D37)},
  MRNUMBER = {5006905},
       DOI = {10.1112/topo.70053},
       URL = {https://doi.org/10.1112/topo.70053},
}

@article {SmW,
    AUTHOR = {Smith, Ivan and Wemyss, Michael},
     TITLE = {Double bubble plumbings and two-curve flops},
   JOURNAL = {Selecta Math. (N.S.)},
  FJOURNAL = {Selecta Mathematica. New Series},
    VOLUME = {29},
      YEAR = {2023},
    NUMBER = {2},
     PAGES = {Paper No. 29, 62},
      ISSN = {1022-1824,1420-9020},
   MRCLASS = {53D37 (14J30 14J33 16S38 53D45)},
  MRNUMBER = {4565163},
       DOI = {10.1007/s00029-023-00828-z},
       URL = {https://doi.org/10.1007/s00029-023-00828-z},
}

@article {Shende,
    AUTHOR = {Shende, Vivek},
     TITLE = {The conormal torus is a complete knot invariant},
   JOURNAL = {Forum Math. Pi},
  FJOURNAL = {Forum of Mathematics. Pi},
    VOLUME = {7},
      YEAR = {2019},
     PAGES = {e6, 16},
      ISSN = {2050-5086},
   MRCLASS = {57M25 (55N30 57M27)},
  MRNUMBER = {4010558},
MRREVIEWER = {Mohamed\ Elhamdadi},
       DOI = {10.1017/fmp.2019.1},
       URL = {https://doi.org/10.1017/fmp.2019.1},
}

@incollection {Ng-intro,
    AUTHOR = {Ng, Lenhard},
     TITLE = {A topological introduction to knot contact homology},
 BOOKTITLE = {Contact and symplectic topology},
    SERIES = {Bolyai Soc. Math. Stud.},
    VOLUME = {26},
     PAGES = {485--530},
 PUBLISHER = {J\'{a}nos Bolyai Math. Soc., Budapest},
      YEAR = {2014},
      ISBN = {978-3-319-02035-8},
      note = {eISBN:978-3-319-02036-5}, 
   MRCLASS = {57M25 (53D42)},
  MRNUMBER = {3220948},
MRREVIEWER = {Quach thi C\^{a}m V\^{a}n},
       DOI = {10.1007/978-3-319-02036-5\_10},
       URL = {https://doi.org/10.1007/978-3-319-02036-5_10},
}

@article {cornwell-1,
    AUTHOR = {Cornwell, Christopher R.},
     TITLE = {Knot contact homology and representations of knot groups},
   JOURNAL = {J. Topol.},
  FJOURNAL = {Journal of Topology},
    VOLUME = {7},
      YEAR = {2014},
    NUMBER = {4},
     PAGES = {1221--1242},
      ISSN = {1753-8416,1753-8424},
   MRCLASS = {57M27},
  MRNUMBER = {3286902},
MRREVIEWER = {Georgios\ Dimitroglou Rizell},
       DOI = {10.1112/jtopol/jtu016},
       URL = {https://doi.org/10.1112/jtopol/jtu016},
}

@article {cornwell-2,
    AUTHOR = {Cornwell, Christopher R.},
     TITLE = {K{CH} representations, augmentations, and {$A$}-polynomials},
   JOURNAL = {J. Symplectic Geom.},
  FJOURNAL = {The Journal of Symplectic Geometry},
    VOLUME = {15},
      YEAR = {2017},
    NUMBER = {4},
     PAGES = {983--1017},
      ISSN = {1527-5256,1540-2347},
   MRCLASS = {57M27 (57R17)},
  MRNUMBER = {3734607},
MRREVIEWER = {Alexander\ Fel\cprime shtyn},
       DOI = {10.4310/JSG.2017.v15.n4.a2},
       URL = {https://doi.org/10.4310/JSG.2017.v15.n4.a2},
}

@article {Ng,
    AUTHOR = {Ng, Lenhard},
     TITLE = {Framed knot contact homology},
   JOURNAL = {Duke Math. J.},
  FJOURNAL = {Duke Mathematical Journal},
    VOLUME = {141},
      YEAR = {2008},
    NUMBER = {2},
     PAGES = {365--406},
      ISSN = {0012-7094},
   MRCLASS = {53D40 (53D12 53D35 57M27)},
  MRNUMBER = {2376818},
MRREVIEWER = {Tobias Ekholm},
       DOI = {10.1215/S0012-7094-08-14125-0},
       URL = {https://doi.org/10.1215/S0012-7094-08-14125-0},
}

@book {LynSch,
    AUTHOR = {Lyndon, Roger C. and Schupp, Paul E.},
     TITLE = {Combinatorial group theory},
    SERIES = {Classics in Mathematics},
      NOTE = {Reprint of the 1977 edition},
 PUBLISHER = {Springer-Verlag, Berlin},
      YEAR = {2001},
     PAGES = {xiv+339},
      ISBN = {3-540-41158-5},
   MRCLASS = {20Fxx (20Exx 57M07)},
  MRNUMBER = {1812024},
       DOI = {10.1007/978-3-642-61896-3},
       URL = {https://doi.org/10.1007/978-3-642-61896-3},
}

@article {BC,
    AUTHOR = {Bourgeois, Fr\'ed\'eric and Chantraine, Baptiste},
     TITLE = {Bilinearized {L}egendrian contact homology and the
              augmentation category},
   JOURNAL = {J. Symplectic Geom.},
  FJOURNAL = {The Journal of Symplectic Geometry},
    VOLUME = {12},
      YEAR = {2014},
    NUMBER = {3},
     PAGES = {553--583},
      ISSN = {1527-5256,1540-2347},
   MRCLASS = {53D42},
  MRNUMBER = {3248668},
MRREVIEWER = {Daniel\ V.\ Mathews},
       DOI = {10.4310/jsg.2014.v12.n3.a5},
       URL = {https://doi.org/10.4310/jsg.2014.v12.n3.a5},
}

@article {Pan,
    AUTHOR = {Pan, Yu},
     TITLE = {The augmentation category map induced by exact {L}agrangian
              cobordisms},
   JOURNAL = {Algebr. Geom. Topol.},
  FJOURNAL = {Algebraic \& Geometric Topology},
    VOLUME = {17},
      YEAR = {2017},
    NUMBER = {3},
     PAGES = {1813--1870},
      ISSN = {1472-2747,1472-2739},
   MRCLASS = {53D42 (53D12 57M50 57R17)},
  MRNUMBER = {3677941},
MRREVIEWER = {Paolo\ Ghiggini},
       DOI = {10.2140/agt.2017.17.1813},
       URL = {https://doi.org/10.2140/agt.2017.17.1813},
}

@misc{hatcher,
  title={Notes on basic 3-manifold topology},
  author={Hatcher, Allen},
  year={2007},
  url={https://pi.math.cornell.edu/~hatcher/3M/3M.pdf},
  options={url=true}
}

@article {ENS,
    AUTHOR = {Ekholm, Tobias and Ng, Lenhard and Shende, Vivek},
     TITLE = {A complete knot invariant from contact homology},
   JOURNAL = {Invent. Math.},
  FJOURNAL = {Inventiones Mathematicae},
    VOLUME = {211},
      YEAR = {2018},
    NUMBER = {3},
     PAGES = {1149--1200},
      ISSN = {0020-9910,1432-1297},
   MRCLASS = {57M27 (53D42)},
  MRNUMBER = {3763406},
MRREVIEWER = {Janko\ Latschev},
       DOI = {10.1007/s00222-017-0761-1},
       URL = {https://doi.org/10.1007/s00222-017-0761-1},
}

@article {HWZ13,
    AUTHOR = {Hamilton, Emily and Wilton, Henry and Zalesskii, Pavel A.},
     TITLE = {Separability of double cosets and conjugacy classes in
              3-manifold groups},
   JOURNAL = {J. Lond. Math. Soc. (2)},
  FJOURNAL = {Journal of the London Mathematical Society. Second Series},
    VOLUME = {87},
      YEAR = {2013},
    NUMBER = {1},
     PAGES = {269--288},
      ISSN = {0024-6107,1469-7750},
   MRCLASS = {57N10 (20E26 57M05)},
  MRNUMBER = {3022716},
MRREVIEWER = {Darren\ D.\ Long},
       DOI = {10.1112/jlms/jds040},
       URL = {https://doi.org/10.1112/jlms/jds040},
}

@article {HarpeWeber,
    AUTHOR = {de la Harpe, Pierre and Weber, Claude},
     TITLE = {On malnormal peripheral subgroups of the fundamental group of
              a 3-manifold},
   JOURNAL = {Confluentes Math.},
  FJOURNAL = {Confluentes Mathematici},
    VOLUME = {6},
      YEAR = {2014},
    NUMBER = {1},
     PAGES = {41--64},
      ISSN = {1793-7434},
   MRCLASS = {57M25 (57N10)},
  MRNUMBER = {3266884},
MRREVIEWER = {Yi\ Liu},
       DOI = {10.5802/cml.12},
       URL = {https://doi-org.kyoto-u.idm.oclc.org/10.5802/cml.12},
}

@article{Friedl,
author="Friedl, Stefan",
title="Centralizers in 3-manifold groups",
journal="RIMS K{\^o}ky{\^u}roku",
ISSN="18802818",
publisher="Research Institute for Mathematical Sciences",
year="2011",
month="06",
number="1747",
pages="23-34",
URL="https://cir.nii.ac.jp/crid/1520290885011381248"
}

@misc{Cerocchi,
      title={On the peripheral subgroups of irreducible 3-manifold groups and acylindrical splittings}, 
      author={Filippo Cerocchi},
      year={2017},
      eprint={1705.06124},
      archivePrefix={arXiv},
      primaryClass={math.GT},
      url={https://arxiv.org/abs/1705.06124}, 
}

@article {ChaSuzuki,
    AUTHOR = {Cha, Jae Choon and Suzuki, Masaaki},
     TITLE = {Non-meridional epimorphisms of knot groups},
   JOURNAL = {Algebr. Geom. Topol.},
  FJOURNAL = {Algebraic \& Geometric Topology},
    VOLUME = {16},
      YEAR = {2016},
    NUMBER = {2},
     PAGES = {1135--1155},
      ISSN = {1472-2747,1472-2739},
   MRCLASS = {57M25 (20F34 20J05 57M05)},
  MRNUMBER = {3493417},
MRREVIEWER = {Tetsuya\ Ito},
       DOI = {10.2140/agt.2016.16.1135},
       URL = {https://doi-org.kyoto-u.idm.oclc.org/10.2140/agt.2016.16.1135},
}

@article {BKN, 
    title={On the genera of symmetric unions of knots}, 
    DOI={10.4153/S0008414X25101740}, 
    JOURNAL = {Canad. J. Math.},
    FJOURNAL = {Canadian Journal of Mathematics. Journal Canadien deMath\'ematiques}, 
    author={Boileau, Michel and Kitano, Teruaki and Nozaki, Yuta}, year={2025}, pages={1–26}}

@article {KitanoSuzukiII,
    AUTHOR = {Kitano, Teruaki and Suzuki, Masaaki},
     TITLE = {A partial order in the knot table. {II}},
   JOURNAL = {Acta Math. Sin. (Engl. Ser.)},
  FJOURNAL = {Acta Mathematica Sinica (English Series)},
    VOLUME = {24},
      YEAR = {2008},
    NUMBER = {11},
     PAGES = {1801--1816},
      ISSN = {1439-8516,1439-7617},
   MRCLASS = {57M25 (57M27)},
  MRNUMBER = {2453061},
MRREVIEWER = {Wilbur\ Whitten},
       DOI = {10.1007/s10114-008-6269-2},
       URL = {https://doi-org.kyoto-u.idm.oclc.org/10.1007/s10114-008-6269-2},
}

@article {KM,
    AUTHOR = {Kronheimer, P. B. and Mrowka, T. S.},
     TITLE = {Witten's conjecture and property {P}},
   JOURNAL = {Geom. Topol.},
  FJOURNAL = {Geometry and Topology},
    VOLUME = {8},
      YEAR = {2004},
     PAGES = {295--310},
      ISSN = {1465-3060,1364-0380},
   MRCLASS = {57M27 (57M25 57R17 57R57 57R58)},
  MRNUMBER = {2023280},
MRREVIEWER = {Jacob\ Andrew\ Rasmussen},
       DOI = {10.2140/gt.2004.8.295},
       URL = {https://doi.org/10.2140/gt.2004.8.295},
}

@misc{Cohn,
      title={Differential Graded Categories are k-linear Stable Infinity Categories}, 
      author={Lee Cohn},
      year={2016},
      eprint={1308.2587},
      archivePrefix={arXiv},
      primaryClass={math.AT},
      url={https://arxiv.org/abs/1308.2587}, 
}

@article {CDHHLMNN2,
    AUTHOR = {Baptiste Calmès and Emanuele Dotto and Yonatan Harpaz and Fabian Hebestreit and Markus Land and Kristian Moi and Denis Nardin and Thomas Nikolaus and Wolfgang Steimle},
     TITLE = {Hermitian {K}-theory for stable {$\infty$}-categories
              {II}: {C}obordism categories and additivity},
   JOURNAL = {Acta Math.},
  FJOURNAL = {Acta Mathematica},
    VOLUME = {235},
      YEAR = {2026},
    NUMBER = {2},
     PAGES = {149--400},
      ISSN = {0001-5962,1871-2509},
   MRCLASS = {19G38 (18N99 55U35)},
  MRNUMBER = {5009505},
       DOI = {10.4310/acta.2025.n235.n2.a1},
       URL = {https://doi-org.kyoto-u.idm.oclc.org/10.4310/acta.2025.n235.n2.a1},
}

@article {EHK,
    AUTHOR = {Ekholm, Tobias and Honda, Ko and K\'{a}lm\'{a}n, Tam\'{a}s},
     TITLE = {Legendrian knots and exact {L}agrangian cobordisms},
   JOURNAL = {J. Eur. Math. Soc. (JEMS)},
  FJOURNAL = {Journal of the European Mathematical Society (JEMS)},
    VOLUME = {18},
      YEAR = {2016},
    NUMBER = {11},
     PAGES = {2627--2689},
      ISSN = {1435-9855},
   MRCLASS = {53D42 (53D10 53D40 57M27 57R90)},
  MRNUMBER = {3562353},
MRREVIEWER = {Georgios Dimitroglou Rizell},
       DOI = {10.4171/JEMS/650},
       URL = {https://doi.org/10.4171/JEMS/650},
}

@article {Ekholm,
    AUTHOR = {Ekholm, Tobias},
     TITLE = {Rational symplectic field theory over {$\mathbb{Z}_2$} for exact
              {L}agrangian cobordisms},
   JOURNAL = {J. Eur. Math. Soc. (JEMS)},
  FJOURNAL = {Journal of the European Mathematical Society (JEMS)},
    VOLUME = {10},
      YEAR = {2008},
    NUMBER = {3},
     PAGES = {641--704},
      ISSN = {1435-9855},
   MRCLASS = {53D35 (53D40 57R17 57R58)},
  MRNUMBER = {2421157},
MRREVIEWER = {Sang Seon Kim},
       DOI = {10.4171/JEMS/126},
       URL = {https://doi.org/10.4171/JEMS/126},
}

@misc{HeyerMann,
      title={6-Functor Formalisms and Smooth Representations}, 
      author={Claudius Heyer and Lucas Mann},
      year={2024},
      eprint={2410.13038},
      archivePrefix={arXiv},
      primaryClass={math.CT},
      url={https://arxiv.org/abs/2410.13038}, 
}

@article {BB13,
    AUTHOR = {Ben-Bassat, Oren and Block, Jonathan},
     TITLE = {Milnor descent for cohesive dg-categories},
   JOURNAL = {J. K-Theory},
  FJOURNAL = {Journal of K-Theory. K-Theory and its Applications in Algebra,
              Geometry, Analysis \& Topology},
    VOLUME = {12},
      YEAR = {2013},
    NUMBER = {3},
     PAGES = {433--459},
      ISSN = {1865-2433,1865-5394},
   MRCLASS = {18D99 (46L87 58B34)},
  MRNUMBER = {3165183},
MRREVIEWER = {Pawel\ Sosna},
       DOI = {10.1017/is013007003jkt236},
       URL = {https://doi-org.kyoto-u.idm.oclc.org/10.1017/is013007003jkt236},
}

@article {Li2,
    AUTHOR = {Li, Wenyuan},
     TITLE = {Lagrangian cobordism functor in microlocal sheaf theory {II}},
   JOURNAL = {J. Symplectic Geom.},
  FJOURNAL = {The Journal of Symplectic Geometry},
    VOLUME = {23},
      YEAR = {2025},
    NUMBER = {3},
     PAGES = {599--672},
      ISSN = {1527-5256,1540-2347},
   MRCLASS = {57R17 (53D12)},
  MRNUMBER = {4931385},
       DOI = {10.4310/jsg.250710021820},
       URL = {https://doi-org.kyoto-u.idm.oclc.org/10.4310/jsg.250710021820},
}

@article {AE22,
    AUTHOR = {Asplund, Johan and Ekholm, Tobias},
     TITLE = {Chekanov-{E}liashberg dg-algebras for singular {L}egendrians},
   JOURNAL = {J. Symplectic Geom.},
  FJOURNAL = {The Journal of Symplectic Geometry},
    VOLUME = {20},
      YEAR = {2022},
    NUMBER = {3},
     PAGES = {509--559},
      ISSN = {1527-5256,1540-2347},
   MRCLASS = {53D12 (53D10 53D40)},
  MRNUMBER = {4563001},
MRREVIEWER = {Wojciech\ Domitrz},
       DOI = {10.4310/jsg.2022.v20.n3.a1},
       URL = {https://doi-org.kyoto-u.idm.oclc.org/10.4310/jsg.2022.v20.n3.a1},
}

@article {EL23,
    AUTHOR = {Ekholm, Tobias and Lekili, Yank\i},
     TITLE = {Duality between {L}agrangian and {L}egendrian invariants},
   JOURNAL = {Geom. Topol.},
  FJOURNAL = {Geometry \& Topology},
    VOLUME = {27},
      YEAR = {2023},
    NUMBER = {6},
     PAGES = {2049--2179},
      ISSN = {1465-3060,1364-0380},
   MRCLASS = {57R17},
  MRNUMBER = {4634745},
MRREVIEWER = {Alexander\ Fel\cprime shtyn},
       DOI = {10.2140/gt.2023.27.2049},
       URL = {https://doi-org.kyoto-u.idm.oclc.org/10.2140/gt.2023.27.2049},
}

@article {BEE12,
    AUTHOR = {Bourgeois, Fr\'ed\'eric and Ekholm, Tobias and Eliashberg,
              Yasha},
     TITLE = {Effect of {L}egendrian surgery},
      NOTE = {With an appendix by Sheel Ganatra and Maksim Maydanskiy},
   JOURNAL = {Geom. Topol.},
  FJOURNAL = {Geometry \& Topology},
    VOLUME = {16},
      YEAR = {2012},
    NUMBER = {1},
     PAGES = {301--389},
      ISSN = {1465-3060,1364-0380},
   MRCLASS = {53D05 (53D40 53D42)},
  MRNUMBER = {2916289},
MRREVIEWER = {Michael\ L.\ Hutchings},
       DOI = {10.2140/gt.2012.16.301},
       URL = {https://doi-org.kyoto-u.idm.oclc.org/10.2140/gt.2012.16.301},
}

\noindent Tomohiro Asano: 
Department of Mathematics, Kyoto University, \linebreak Kitashirakawa-Oiwake-Cho, Sakyo-ku, 606-8502, Kyoto, Japan.

\noindent \textit{E-mail address}: \texttt{tasano[at]math.kyoto-u.ac.jp}, \texttt{tomoh.asano[at]gmail.com}

\medskip

\noindent Yukihiro Okamoto:
Department of Mathematical Sciences,
Tokyo Metropolitan University,
Minami-osawa, Hachioji, Tokyo, 192-0397, Japan

\noindent
\textit{E-mail address}: 
\texttt{yukihiro[at]tmu.ac.jp}

\end{document}